\documentclass[reqno,11pt,final]{amsart}
\usepackage[utf8]{inputenc}
\usepackage[english]{babel}
\usepackage[a4paper,twoside]{geometry}
\usepackage{microtype}
\usepackage[shortlabels]{enumitem}

\usepackage[backend=biber,style=alphabetic,bibencoding=utf8,language=auto,autolang=other,maxbibnames=4, url=false, doi=false, isbn=false]{biblatex}
\usepackage{csquotes}

\usepackage[hyperfootnotes=false]{hyperref} 
\usepackage{color}
\hypersetup{
	colorlinks = true,
	linkcolor = {blue},
	urlcolor = {red},
	citecolor = {blue}
}

\usepackage{mathrsfs} 
\usepackage{bm}
\usepackage{mathtools}

\newcounter{results}[section]

\theoremstyle{plain}
\newtheorem{theorem}[results]{Theorem}

\newtheorem{lemma}[results]{Lemma}
\newtheorem{proposition}[results]{Proposition}
\newtheorem{corollary}[results]{Corollary}

\theoremstyle{remark}
\newtheorem{remark}[results]{Remark}
\newtheorem{example}[results]{Example}

\theoremstyle{definition}
\newtheorem{definition}[results]{Definition}

\numberwithin{equation}{section}

\newcommand{\R}{\ensuremath{\mathbb R}} 
\newcommand{\N}{\ensuremath{\mathbb N}} 

\newcommand{\Borel}{\ensuremath{\mathscr{B}}} 
\newcommand{\meas}{\ensuremath{\mathcal{M}}} 
\newcommand{\prob}{\ensuremath{\mathcal{P}}} 
\newcommand{\D}[1]{\ensuremath{\Delta_+(#1)}} 

\newcommand{\de}{\ensuremath{\,\mathrm d}}
\newcommand{\eps}{\ensuremath{\varepsilon}} 
\newcommand{\weakto}{\ensuremath{\rightharpoonup}} 
\newcommand{\scalprod}[2]{\ensuremath{\langle #1, #2\rangle}} 
\DeclareMathOperator{\diam}{diam}

\newcommand{\rmC}{\mathrm C} 

\newcommand{\net}[1]{\ensuremath{ \left ( #1_{\lambda} \right )_{\lambda \in \mathbb{L}}}}
\newcommand{\cost}[1]{\ensuremath{\mathscr{S}_{#1}}}

\newcommand{\m}{\ensuremath{\mathscr U}} 
\newcommand{\sfH}{\mathsf H} 
\newcommand{\sfc}{\mathsf c} 
\newcommand{\sfd}{\mathsf d} 
\newcommand{\sfX}{\mathsf X} 
\newcommand{\sfx}{\mathsf x} 
\newcommand{\sfr}{\mathsf r} 
\newcommand{\sfY}{\mathsf Y} 
\newcommand{\sfg}{\mathsf g} 

\newcommand{\f}[1]{\mathfrak{#1}} 
\newcommand{\aalpha}{\boldsymbol \alpha} 
\newcommand{\ggamma}{\boldsymbol \gamma} 
\newcommand{\bbeta}{{\boldsymbol\beta}}

\DeclareMathOperator{\dil}{\mathrm{dil}} 
\DeclareMathOperator{\prd}{\mathrm{prd}} 
\DeclareMathOperator{\discr}{\mathcal{M}_{+,f}} %

\newcommand{\cof}[1]{\ensuremath{\widehat{\operatorname{\mathrm{co}}}\left(#1\right)}}

\newcommand{\ce}[1]{\ensuremath{\operatorname{\mathrm{co}}\left(#1\right)}}

\newcommand{\cce}[1]{\ensuremath{\operatorname{\overline{\mathrm{co}}}\left(#1\right)}}

\newcommand{\hmg}[1]{\ensuremath{\operatorname{\mathrm{hom}}\left(#1\right)}}

\newcommand{\restricts}[2] {
	#1 
	\raisebox{-.3ex}{$|$}_{#2}	
}

\newcommand{\pc}{\f{C}[\sfX_1, \sfX_2]} 
\newcommand{\pcr}[1]{\f{C}_{#1}[\sfX_1, \sfX_2]} 
\newcommand{\pco}{\pcr{\boldsymbol\fo}}

\newcommand{\dual}{\mathcal{D}} 

\newcommand{\supp}[1]{ \mathsf{supp} \left ( #1 \right )} 

\newcommand{\intt}[1]{\operatorname{int}\left(#1\right)} 

\newcommand{\rintt}[1]{\operatorname{rad-int}\left(#1\right)} 

\newcommand{\pcs}{\f{C}[\sfX, \sfX]} 

\newcommand{\pcrs}[1]{\f{C}_{#1}[\sfX, \sfX]} 

\newcommand{\nchi}{{\raise.3ex\hbox{$\chi$}}} 

\newcommand{\rmD}{{\mathrm D}} 

\newcommand{\mres}{\mathbin{\vrule height 1.6ex depth 0pt width  0.13ex\vrule height 0.13ex depth 0pt width 1.3ex}} 

\newcommand{\fy}{{\f y}} 
\newcommand{\fo}{{\f o}} 
\newcommand{\bfo}{{\boldsymbol\fo}}
\newcommand{\sfa}{{\mathsf a}} 
\newcommand{\piGamma}[1]{{\Gamma_{#1}}} 

\newcommand{\dcc}{\varrho} 
\newcommand{\dist}{\mathcal{D}_{\dcc,p}} 
\newcommand{\wass}{W_{\dcc, p}} 
\newcommand{\arc}{\overline{\mathcal{A}}} 

\title[A relaxation viewpoint to Unbalanced Optimal Transport]{A relaxation viewpoint to Unbalanced Optimal Transport: duality, optimality and Monge formulation}

\author{Giuseppe Savar\'e}
\address{Giuseppe Savar\'e: Department of Decision Sciences,
  Bocconi University. Via Roentgen 1, 20136 Milan (Italy)}
\email{giuseppe.savare@unibocconi.it}

\author{Giacomo Enrico Sodini}
\address{Giacomo Enrico Sodini: Institut für Mathematik - Fakultät für Mathematik - Universität Wien, Oskar-Morgenstern-Platz 1, 1090 Wien (Austria)}
\email{giacomo.sodini@univie.ac.at}

\subjclass{Primary: 49Q22, 28A33 ; Secondary: 49K27.}
 \keywords{Unbalanced Optimal Transport, Kantorovich-Monge problem, Optimality conditions}

\begin{document}
\begin{abstract}
  We present a general convex relaxation approach to study a wide class of Unbalanced Optimal Transport problems for finite non-negative measures with possibly different masses. These are obtained as the lower semicontinuous and convex envelope of a cost for non-negative Dirac masses. 
 New general primal-dual formulations, optimality conditions, and metric-topological properties are carefully studied and discussed.
\end{abstract}
\maketitle
\tableofcontents
\thispagestyle{empty}

\section{Introduction}
The problem of extending Optimal Transport methods to pairs of unbalanced non-negative measures has been considered in a large number of works with different techniques and different aims. 

For what concerns dynamical formulations, many models inspired by the Benamou-Brenier's \cite{3uu}
fluid dynamic formulation of the classical Kantorovich-Rubinstein-Wasserstein Optimal Transport metric  have been proposed, see for example \cite{msg, 49eee, 18kkk, RosPic, Rospicprop,25eee,CPSV18,LMS18,Vaios-Mielke19}. In such works, the authors consider source terms in the continuity equation, thus leading to gain/loss of mass during the evolution. The models proposed differ in the kind of source or penalization. We refer to \cite{25eee} where a more detailed description of these models is given.

Static formulations of the Unbalanced Optimal Transport problem were proposed already by Kantorovich and Rubinstein \cite{KR:58} and subsequently extended by Hanin \cite{haninee} (see also the dual norm in \cite{38eee}). These approaches can be thought as a classical Optimal Transport problem where a fraction of the mass is allowed to go (or come from) a point at infinity (see also \cite{figig, Guittet}). More recent approaches are given by the so called optimal partial transport \cite{caffax, 11uu}, which was previously related to image retrieval \cite{21kkk, 27kkk}.

Optimal partial transport (see \cite{25eee}) is in turn also related to \cite{RosPic, Rospicprop}, since this latter works also provide a dynamic formulation of optimal partial transport. We also mention that \cite{RosPic, Rospicprop} are also connected to \cite{jdb} where it was proposed to change the marginal constraints and to add a penalization term.

The important case of the Hellinger-Kantorovich metric (a.k.a.~Wasserstein-Fisher-Rao), somehow interpolating the Wasserstein and the Hellinger metrics,
has been proposed independently in \cite{msg,25eee,CPSV18,LMS18}, with different equivalent characterizations. 

A very useful one \cite{LMS18} involves an entropic relaxation of the marginal constraints in the classical static formulation of Optimal Transport, 
it provides a paradigmatic example of the wide class of Optimal Entropy-Transport problems, and suggests the crucial role of the so-called cone geometry.

\cite{CPSV18} introduces a general class of Unbalanced Optimal Transport problems in compact subsets of some Euclidean space induced by a sublinear cost function depending on positions and masses, using a suitable static semi-coupling formulation and proving a general duality result, with interesting applications to various dynamic problems.

It turns out that all these different viewpoints are specific examples of
a general construction based on convex relaxation, which we want to exploit in the present paper.

\subsubsection*{\textit{\bfseries A convex relaxation viewpoint.}}
Adopting the same approach of \cite{SS20} used for the classical Optimal Transport problem and extending some ideas already contained in \cite{LMS18}, the aim of this work is to define 
a general class of Unbalanced Optimal Transport problems
between non-negative measures 
as the \emph{convex and lower semicontinuous envelope} of a cost  initially defined only between weighted Dirac masses, and then to study the corresponding duality formulas, optimality conditions and metric-topologial properties.

Let us briefly recall the result of \cite{SS20} which is the starting point for the present work; given a pair of completely regular spaces $\sfX_1$ and $\sfX_2$, every proper and lower semicontinuous cost function $\mathsf{c}: \sfX_1 \times \sfX_2 \to [0, +\infty]$ can be naturally lifted to a singular cost functional $\mathcal{F}_\mathsf{c}: \meas(\sfX_1) \times \meas(\sfX_2) \to [0,+\infty]$,
where $\meas(\sfX_i)$ denote the 
space of signed and finite Radon measures in 
$\sfX_i$. 
$\mathcal{F}_\mathsf{c}$ is finite only between balanced pairs of Dirac masses
and can be defined as
\begin{equation}\label{eq:singcostbal}
 \mathcal{F}_\mathsf{c} (\mu_1, \mu_2) := \begin{cases} r\sfc(x_1,x_2) \quad &\text{ if } \begin{aligned} &\quad \mu_1= r\delta_{x_1},\, \mu_2 = r\delta_{x_2},\\ &\quad x_1 \in \sfX_1, x_2 \in \sfX_2,\, r \ge 0,\end{aligned} \\ 
 + \infty \quad &\text{ elsewhere.} \end{cases}
\end{equation}
By \cite[Theorem 4.4]{SS20}
the convex and lower semicontinuous (w.r.t.~the product weak topology) envelope of $\mathcal{F}_\mathsf{c}$ coincides with the
Optimal Transport functional 
\begin{equation}\label{eq:intro-OT}
 \mathsf{OT}_\sfc(\mu_1, \mu_2) := \begin{cases} 
 \displaystyle \inf\Big\{\int  \sfc\,\de\ggamma:
 \ggamma\in \Gamma(\mu_1,\mu_2)\Big\} &\text{ if } \mu_i\in \meas_+(X_i),\ 
 \mu_1(\sfX_1) = \mu_2(\sfX_2),\\
 + \infty \quad &\text{ elsewhere}, \end{cases}
\end{equation}
where $\meas_+(\sfX_i)$ denotes
the cone of nonnegative (finite Radon) measures in $\meas(\sfX_i)$ and
$\Gamma(\mu_1,\mu_2)$ is the subset
of couplings between $\mu_1$ and $\mu_2$, i.e.~measures  in $\meas_+(\sfX_1\times \sfX_2)$
with marginals $\mu_1$ and $\mu_2$ respectively.
Clearly $\Gamma(\mu_1,\mu_2)$ is empty
if $\mu(\sfX_1)\neq \mu_2(\sfX_2).$

In order to extend this construction to the case of non-negative measures with possibly different masses, we represent weighted Dirac masses in $\sfX$
of the form
$r\delta_x$, $r\ge0$, as points $(x,r)$ of the so called geometric cone $\f{C}[\sfX]:=\big(\sfX\times [0,+\infty)\big)/\sim$ 
(see Section \ref{sec:theconee}), 
where the equivalence relation $\sim$ identifies all the points of the form
$(x,0)$, $x\in \sfX$, which 
correspond to the null measure in $\sfX$. 
A cost on weighted Dirac masses can thus 
be expressed by a function
\[ \sfH: \f{C}[\sfX_1] \times \f{C}[\sfX_2] \to [0, +\infty],\]
which we will assume to be proper, lower semicontinuous and \emph{radially $1$-homogeneous}, in the sense that the map
\[ \R_2^+ \ni (r_1,r_2) \mapsto \sfH([x_1,r_1],[x_2,r_2])\]
is $1$-homogeneous for every fixed $(x_1,x_2) \in \sfX_1 \times \sfX_2$. While the properness and the lower semicontinuity assumptions are natural, the $1$-homogeneity assumption deserves a comment: from a modeling point of view we are saying that moving $m r_1 \delta_{x_1}$ to $m r_2 \delta_{x_2}$ costs exactly $m$ times moving $r_1 \delta_{x_1}$ to $r_2 \delta_{x_2}$. In analogy with \eqref{eq:singcostbal} we can define the unbalanced singular cost $\cost{\sfH} : \meas (\sfX_1) \times \meas(\sfX_2) \to [0, + \infty]$ as
\[ \cost{\sfH} (\mu_1, \mu_2) := \begin{cases} \mathsf{H}([x_1, r_1], [x_2,r_2]) \quad &\text{ if } \begin{aligned} &\quad \mu_1= r_1\delta_{x_1},\, \mu_2 = r_2\delta_{x_2},\\ &\quad x_1 \in \sfX_1, x_2 \in \sfX_2,\, r_1,r_2 \ge 0,\end{aligned} \\ 
 + \infty \quad &\text{ elsewhere}, \end{cases}   
\]
and look for 
the largest convex and lower semicontinuous 
functional $\m_\sfH$ on $\meas(\sfX_1)\times \meas(\sfX_2)$ dominated by $\cost{\sfH}$. 
It turns out that 
such a functional 
admits an explicit characterization  which 
resembles \eqref{eq:intro-OT}
via the formula
\begin{equation}\label{eq:toconlms}
\m_{\sfH}(\mu_1, \mu_2) := \inf \left \{ \int_{\f{C}[\sfX_1] \times \f{C}[\sfX_2]} \sfH \de \aalpha \mid \aalpha \in \f{H}^1(\mu_1, \mu_2) \right \},
\end{equation}
where $\f{H}^1(\mu_1, \mu_2)$ is the set of \emph{homogeneous couplings},
i.e.~measures $\aalpha$ on $\f{C}[\sfX_1] \times \f{C}[\sfX_2]$ with \emph{$1$-homogeneous marginals} $\mu_1$ and $\mu_2$ (see \eqref{eq:hommarg}):
$\aalpha$ belongs to 
$\f{H}^1(\mu_1, \mu_2)$ if
\[
\left\{
\begin{aligned}
    \mu_1(A_1)&=\int_{A_1\times [0,+\infty)\times \f{C}[\sfX_2] }
    r_1\,\de\aalpha(x_1,r_1;x_2,r_2),\\
    \mu_2(A_2)&=\int_{
    \f{C}[\sfX_1]\times A_2\times [0,+\infty)}
    r_2\,\de\aalpha(x_1,r_1;x_2,r_2),  
    \end{aligned}
    \right.
    \qquad
    \text{for every Borel subsets $A_i\subset \sfX_i$.}
\]
This choice has been inspired 
by \cite{LMS18}, 
where the authors prove that 
the class of Optimal Entropy-Transport problems
can be formulated precisely as in \eqref{eq:toconlms} for a suitable choice of the function $\sfH$ (see in particular \cite[Definition 5.1]{LMS18}).

\subsubsection*{\textit{\bfseries
Structural properties of Unbalanced Optimal Transport problems}}
The following result collects  some of 
the fundamental properties
of the unbalanced framework (see Theorems \ref{ss22:prop:mh}, \ref{ss22:teo:representation} and \ref{ss22:teo:duality}),
in parallel to similar results of the 
classical Optimal Transport theory.

\begin{theorem} Let $\sfX_1, \sfX_2$ be completely regular spaces and let $\sfH:\f{C}[\sfX_1] \times \f{C}[\sfX_2] \to [0, + \infty]$ be a proper, radially $1$-homogeneous and lower semicontinuous function. 
\begin{enumerate}
    \item 
 For every $(\mu_1 ,\mu_2) \in \meas_+(\sfX_1) \times \meas_+(\sfX_2)$ such that $\m_\sfH(\mu_1, \mu_2) <+\infty$, there exists an optimal homogeneous coupling
 $\aalpha \in \f{H}^1(\mu_1, \mu_2)$ such that 
\[ \m_\sfH(\mu_1, \mu_2) = \int_{\f{C}[\sfX_1] \times \f{C}[\sfX_2]}\sfH \de \aalpha.\]
\item $\m_\sfH$ is a lower semicontinuous convex function and satisfies 
\[ \m_\sfH(r_1 \delta_{x_1}, r_2 \delta_{x_2}) =
\cce{\sfH}([x_1,r_1],[x_2,r_2]) \le \sfH([x_1,r_1], [x_2, r_2])
\]
for every $(x_1, x_2) \in \sfX_1 \times \sfX_2$ and every $(r_1, r_2) \in \R_+^2$, where $\cce{\sfH}$
is the l.s.c.~convex envelope of $\sfH([x_1,\cdot],[x_2,\cdot])$ with respect to the 
variables $(r_1,r_2)\in \R^2_+.$ 
If, in addition, $\sfH$ is also radially convex
(i.e.~the map $(r_1,r_2)\mapsto \sfH([x_1,r_1],[x_2,r_2])$ is convex for every fixed $(x_1,x_2) \in \sfX_1 \times \sfX_2$), then the above inequality is an equality. 
\item 
$\m_\sfH$ is the convex l.s.c.~envelope of  $\cost{\sfH}$ and admits the dual formulation
\begin{equation} \label{eq:ladualintro} \begin{aligned}
    \m_\sfH(\mu_1, \mu_2)=
     \m_{\cce{\sfH}}(\mu_1, \mu_2) 
    &=  \cce{\cost{\sfH}}(\mu_1, \mu_2)\\
    &= \sup \Big \{ \int_{\sfX_1}\varphi_1
    \,\de\mu_1+
    \int_{\sfX_2}\varphi_2\,\de\mu_2
    \mid (\varphi_1, \varphi_2) \in \Phi_\sfH \Big \} 
\end{aligned}
\end{equation}
for every $(\mu_1, \mu_2) \in \meas_+(\sfX_1) \times \meas_+(\sfX_2)$,
    where
\begin{equation}\label{eq:constrintro}
  \begin{aligned} 
\Phi_\sfH := \Big\{ 
     (\varphi_1, \varphi_2) \in \rmC_b(\sfX_1) \times \rmC_b(\sfX_2):{}&
  \varphi_1(x_1)r_1+ \varphi_2(x_2)r_2 \le \sfH([x_1, r_1], [x_2,r_2])\\
  &
    \text{ for every } (x_1, x_2) \in \sfX_1 \times \sfX_2, \, r_1, r_2\ge 0
    \Big\}.
    \end{aligned}
  \end{equation}
  \end{enumerate}
\end{theorem}
\noindent
Notice that, if the cost function $\sfH$ is given by
\[
\sfH([x_1, r_1] , [x_2, r_2]) := \begin{cases} r \sfc(x_1, x_2) \quad &\text{ if } r_1 = r_2=r \ge 0, \\ + \infty \quad &\text{ elsewhere,} \end{cases} 
\]
for some proper and lower semicontinuous function $\sfc:\sfX_1 \times \sfX_2 \to [0,+\infty]$, then $\m_{\sfH}= \mathsf{OT}_\sfc$, $\cost{\sfH}=\mathcal{F}_\sfc$, and we recover precisely the analogous results for the classical Optimal Transport theory (see in particular \cite[Theorems 3.5 and 4.4]{SS20}).
The relaxation formula 
\eqref{eq:ladualintro}
extends the corresponding one 
of \cite[Proposition 3.11]{CPSV18}, showing in particular that the primal formulation
\eqref{eq:toconlms}
in terms of homogeneous couplings
is equivalent to the semi-coupling approach of
\cite{CPSV18}.
\eqref{eq:ladualintro} has also nice applications
to the structure of the Hellinger-Kantorovich metric,
see Remark \ref{rem:interesting-app}.

\medskip
\noindent
 Along the same lines of the classical Optimal Transport theory, it is natural to investigate the 
 \emph{existence of optimal potentials maximizing 
\eqref{eq:ladualintro} and thus solving the dual problem.}
We first focus on the class of continuous functions, when $\sfX_1, \sfX_2$ are compact and metrizable and $\sfH$ is finite and
continuous on the whole product cone $\f{C}[\sfX_1] \times \f{C}[\sfX_2]$, radially $1$-homogeneous and convex.
In this case a sufficient condition relies on
the singular behaviour of the normal derivative of $\sfH$ 
at the boundary of the product cone as in 
the fundamental
example
\begin{equation}\label{eq:fundex}
    \sfH([x_1,r_1],[x_2,r_2]):=
    r_1+r_2-2\sqrt{r_1r_2}\, \mathrm e^{-|x_1-x_2|^2/2}, \quad x_i \in \R^d, \, r_i \ge 0.
\end{equation}
We refer to Section \ref{sec:5} for a detailed discussion of these hypotheses and Appendix \ref{app:1} for an alternative set of assumptions (i.e.~when $\sfH$ is finite and continuous on an open conical domain).
In both these situations it is possible to define the analogue of the celebrated $\sfc$-transform for a pair $(\varphi_1, \varphi_2) \in \Phi_\sfH$ as 
\begin{align*}
    \varphi_1^{\sfH}(x_2) &:= \inf_{x_1 \in \sfX_1} \inf_{\alpha \ge 0} \biggl \{ \sfH([x_1, \alpha], [x_2, 1])-\alpha \varphi_1(x_1) \biggr \}, \quad x_2 \in \sfX_2,\\
    \varphi_1^{\sfH \sfH}(x_1) &:=\inf_{x_2 \in \sfX_2} \inf_{\alpha \ge 0} \biggl \{ \sfH([x_1, 1], [x_2, \alpha])-\alpha \varphi_1^\sfH(x_2) \biggr \}, \quad x_1 \in \sfX_1,
\end{align*}
and prove that the transformed potentials are equicontinuous so that the use of a compactness argument
in the space $\rmC(\sfX_1)\times \rmC(\sfX_2)$ is possible. This will produce an optimal pair of potentials, see Theorems \ref{ss22:theo:potexist} and \ref{ss22:theo:potexist2}.

\subsubsection*{\textit{\bfseries Optimality conditions and relaxed duality}}
In this unbalanced setting, general optimality conditions still involve the notion of cyclical monotonicity: a subset $\Gamma \subset \f{C}[\sfX_1] \times \f{C}[\sfX_2] $ is $\sfH$-cyclically monotone if for every finite family of points $\{(\f{y}_1^i, \f{y}_2^i)\}_{i=1}^N \subset \Gamma$  and every permutation $\sigma$ of $\{1, \dots, N\}$ it holds
\[ \sum_{i=1}^N \sfH(\f{y}_1^i, \f{y}_2^i) \le \sum_{i=1}^N \sfH(\f{y}_1^i, \f{y}_2^{\sigma(i)}).\]
Similarly to the classical Optimal Transport case, \emph{every optimal homogeneous coupling is concentrated on a $\sfH$-cyclically monotone set $\Gamma$} (see Proposition \ref{prop:necee}) which in addition is \emph{a radial convex cone} in the sense that
\[ ([x_1,r_1^i], [x_2, r_2^i]) \in \Gamma,\, \lambda_i \ge 0, \, i=1,2 \, \Rightarrow \Big(\big [x_1, \sum_{i=1}^2\lambda_ir_1^i], [x_2, \sum_{i=1}^2\lambda_ir_2^i] \Big ) \in \Gamma.\]
This further property comes from the radial homogeneity and convexity assumption on $\sfH$. 

Formulating a converse statement to the above proposition, thus involving sufficient optimality conditions, requires additional assumptions on the compatibility between the radial cone $\Gamma$, on which an admissible homogeneous coupling $\aalpha$ between measures $\mu_i \in \meas_+(\sfX_i)$ is concentrated, and the cost function $\sfH$: in Section \ref{subs:graph} we study the natural directed graph structures induced by $\sfH$ and $\Gamma$, similarly to what has been done in the classical Optimal Transport theory \cite{beiblo,biancara} for possibly infinite costs. This leads to the notion of $\sfH$-connectedness: $\Gamma$ is $\sfH$-connected if, whenever $\f{y}_1, \f{y}'_1 \in \pi^1(\Gamma)$ (the projection of $\Gamma$ on the first cone $\f{C}[\sfX_1]$), we can find a sequence of points $\{\f{y}_2^1, \f{y}_1^2, \f{y}_2^2, \dots, \f{y}_1^N, \f{y}_2^N\}$ such that
\[ (\f{y}_1, \f{y}_2^1), (\f{y}_1^i, \f{y}_2^i) \in \Gamma, \, i=2, \dots, N, \quad \sfH(\f{y}_1', \f{y}_2^N), \sfH(\f{y}_1^{i+1}, \f{y}_2^i) < + \infty, \, i=1, \dots, N-1. \]
In other words, the sequence $\{\f{y}_2^1, \f{y}_1^2, \f{y}_2^2, \dots, \f{y}_1^N, \f{y}_2^N\}$ ``connects" $\f{y}_1$ to $\f{y}'_1$ keeping the cost finite when moving from a point in $\f{C}[\sfX_2]$ to a point in $\f{C}[\sfX_1]$ and imposing that the considered pair is in $\Gamma$, when moving from a point in $\f{C}[\sfX_1]$ to a point in $\f{C}[\sfX_2]$. For example, if $\sfH$ is everywhere finite, then any $\Gamma$ is $\sfH$-connected (see Theorem \ref{thm:connectedness} for simple conditions implying $\sfH$-connectedness).

Upon assuming that $\Gamma$ is a $\sfH$-cyclically monotone and $\sfH$-connected subset of the effective domain of $\sfH$ also intersecting the interior of such domain, it is possible to prove the existence of \emph{relaxed optimal potentials}:  these are Borel functions $\varphi_i: \sfX_i \to \R \cup \{\pm \infty\}$ satisfying 
\begin{align}  \label{eq:genpot2}
\varphi_1(x_1) r_1 +_o \varphi_2(x_2)r_2 &\le \sfH([x_1,r_1],[x_2,r_2]) \quad \text{ for every } x_i \in \sfX_i, \, r_i \ge 0, \\
\varphi_1(x_1) r_1 + \varphi_2(x_2)r_2 & =\sfH([x_1,r_1],[x_2,r_2]) \quad \text{ if } ([x_1,r_1],[x_2,r_2]) \in \Gamma, 
\end{align}
where the inequality in \eqref{eq:genpot2} corresponds to the constraint as in \eqref{eq:constrintro} and the notation $+_o$ means that, whenever an ambiguity $\pm \infty \mp \infty$ arises, the sum is set equal to $0$.
This existence result, together with some finiteness conditions relating $\mu_i$ and $\sfH$ (see \eqref{ss22:cond1}), imply that 
 $\varphi_i \in \mathcal{L}^1(\sfX_i, \mu_i)$
 solve the dual problem
 \begin{equation}
     \label{eq:genpot1}
\int_{\sfX_1} \varphi_1 \de \mu_1 + \int_{\sfX_2} \varphi_2 \de \mu_2 = \m_\sfH(\mu_1, \mu_2),
 \end{equation}
 and yields
 optimality for the homogeneous plan $\aalpha$. We refer to Theorem \ref{ss22:suff1} for the detailed sufficiency result.

\subsubsection*{\textit{\bfseries A Monge-like formulation via transport-growth maps.}}
As in classical Optimal Transport problems, Unbalanced Optimal Transport admits a more  restrictive
Monge-like formulation in terms of \emph{transport-growth} pairs \cite{lms22}:
they are maps $(\mathsf T,\mathsf g)$ from $\sfX_1$ to $\sfX_2\times [0,+\infty)$ 
acting on measures $\mu_1\in \meas(\sfX_1)$ via the formula
\begin{equation}\label{eq:twmap}
    (\mathsf T,\mathsf g)_\star \mu_1=\mathsf T_\sharp (\mathsf g \mu_1),
    \quad
    \mu_2=(\mathsf T,\mathsf g)_\star \mu_1
    \ \Leftrightarrow\
    \mu_2(B)=\int_{\mathsf T^{-1}(B)}\mathsf g\,\de\mu_1\quad\text{for every Borel set }
    B\subset \sfX_2.
\end{equation}
Whenever $\mathsf g\in L^1(\sfX_1,\mu_1)$, a transport-growth 
pair induces a 
homogeneous coupling $\aalpha\in \mathfrak H^1(\mu_1,\mu_2)$,
$\mu_2=(\mathsf T,\mathsf g)_\star \mu_1$, via the formula
\begin{equation}\label{eq:induced}
    \aalpha=([\text{id}_{\sfX_1},1],[\mathsf T,\mathsf g])_\sharp \mu_1,\quad
    \int \sfH\,\de\aalpha=
    \int \sfH([x,1],[\mathsf T(x),\mathsf g(x)])\,\de\mu_1(x),
\end{equation}
and one can study the corresponding Monge formulation of \eqref{eq:toconlms}:
\begin{equation}
    \label{eq:Monge-intro}
    \mathscr M_{\sfH}(\mu_1,\mu_2):=\inf\Big\{\int_{\sfX_1}
    \sfH([x,1],[\mathsf T(x),\mathsf g(x)])\,\de\mu_1(x) :
    \mu_2=(\mathsf T,\mathsf g)_\star \mu_1\Big\}.
\end{equation}
Clearly, $\mathscr U_\sfH\le \mathscr M_\sfH$; notice also that, whenever an optimal plan $\aalpha$ between $\mu_1$ and $\mu_2$ charges the set $\{\f{y}_1 = \f{o}_1\}$, then it is impossible that such $\aalpha$ can be induced by a transport-growth map $(\mathsf{T}, \sfg)$ as in \eqref{eq:induced}. This corresponds to the fact that some of the mass of $\mu_2$ does not come from $\mu_1$ but it is ``created''. Under the same hypotheses of the classical balanced case \cite{pratelli} (i.e.~the cost function is continuous and $\mu_1$ is atomless),\begin{theorem} Let $\sfX_i$ be Polish spaces, let $\sfH: \f{C}[\sfX_1] \times \f{C}[\sfX_2] \to [0,+\infty]$ be a proper, radially $1$-homogeneous and continuous function, $\mu_i \in \meas_+(\sfX_i)$ with $\mu_1$ diffuse (i.e.~$\mu_1(\{x\})=0$ for every $x\in \sfX_1$) and such that $\mu_1(\sfX_1)>0$. Then
\[ \m_\sfH(\mu_1, \mu_2) = \mathscr M_\sfH(\mu_1, \mu_2).\]
\end{theorem}
In general, the infimum in the definition of $\mathscr M_{\sfH}(\mu_1, \mu_2)$ is not attained; however, as a consequence of the existence of relaxed optimal potentials as in \eqref{eq:genpot1} and \eqref{eq:genpot2}, when $\sfX_1=\sfX_2=\R^d$, $\mu_1$ is absolutely continuous with respect to 
the Lebesgue measure in $\R^d$ and the differential of $\sfH$ satisfies suitable assumptions,
we also obtain the existence of an optimal transport-growth map 
attaining the minimum in \eqref{eq:Monge-intro}. This is the case, for example, of the cost function in \eqref{eq:fundex} associated with the 
Gaussian Hellinger-Kantorovich metric. 
\begin{theorem} Let $\sfH: \f{C}[\R^d]\times \f{C}[\R^d] \to \R$ be given by  
\[ \sfH([x_1,r_1],[x_2,r_2]):=
    r_1+r_2-2\sqrt{r_1r_2}\, \mathrm e^{-|x_1-x_2|^2/2}, \quad x_i \in \R^d, \, r_i \ge 0.\]
If $\mu_i \in \meas_+(\R^d)$ with $\mu_i(\R^d)>0$ for $i=1,2$ and $\mu_1 \ll \mathcal{L}^d$, then there exists a transport-growth map $(\mathsf{T},\mathsf{g}):\R^d \to \R^d \times [0,+\infty)$ such that 
\[ \mu_2 = (\mathsf T,\mathsf g)_\star \mu_1 = \mathsf{T}_\sharp (\mathsf{g} \mu_1),  
\quad \int_{\R^d} \sfH([x, 1], [\mathsf{T}(x),\mathsf g(x)]) \de \mu_1(x) = 
\m_\sfH(\mu_1, \mu_2).\]
\end{theorem}
The above theorem holds also for large class of cost functions $\sfH$, we refer to Theorem \ref{theo:themap} for the general result.

\subsubsection*{\textit{\bfseries Metric and topological properties.}} 
Finally we frame in our general setting  some of the results of \cite{LMS18,CPSV18,DP20, DM22} concerning the metric properties 
of Unbalanced Optimal Transport functionals.  
In particular, we 
prove that, in case $\sfH$ is (the $p$-th power of) a distance on $\f{C}[\sfX]$, then the resulting cost $\m_\sfH$ is itself (the $p$-th power of) a distance on an appropriate subset of $\meas_+(\sfX)$ metrizing the weak convergence of measures (see Theorems \ref{ss22:thmmetric} and \ref{ss22:thmmnarrow}), precisely as it is for the standard Optimal Transport problem \cite[Proposition 7.1.5]{AGS08}.

\medskip
\paragraph{\em\bfseries Plan of the paper.} Section \ref{sec:2} is devoted to establish the general setting and a few technical tools that will be used in the sequel. 

Section \ref{sec:4} contains the core of our results: the convexification approach,  some structural properties
and the comparison with the semi-coupling approach of \cite{CPSV18},
the duality, and the Monge formulation.

Section \ref{sec:5} treats the dual attainment in spaces of continuous potentials under additional regularity assumptions on $\sfH$ and on the spaces $\sfX_i$.

In Section \ref{sec:6} we present the general optimality conditions, the relaxed duality result and, as a consequence of the latter, the existence of optimal transport-growth maps. 

Appendix \ref{sec:7} contains a few remarks on the metric and topological properties of $\m_\sfH$ in case $\sfH$ is (the $p$-th power of) a distance. Finally Appendix \ref{app:1} reproduces the results of Section \ref{sec:5} under a different set of additional assumptions on $\sfH$.

\medskip
\paragraph{\em\bfseries Acknowledgments.}
The authors gratefully acknowledge the support of the Institute for Advanced Study of the Technical University of Munich, funded by the German Excellence Initiative.
G.S. has also been supported by IMATI-CNR, Pavia and by the MIUR-PRIN
202244A7YL project {\em Gradient Flows and Non-Smooth Geometric Structures with Applications to Optimization and Machine Learning}.

\section{Preliminaries, measures and functions on the cone}\label{sec:2}
In this section, $\sfX$ is a completely regular space (i.e.~it is Hausdorff and for every closed set $C$ and point $x\in \sfX\setminus C$ there exists a continuous function $f: \sfX \to [0,1]$ s.t.~$f(x)=0$ and $f(C)=\{1\}$). We denote by $\Borel(\sfX)$ the Borel $\sigma$-algebra on $\sfX$ and by $\meas(\sfX)$ (resp.~$\meas_+(\sfX)$) the vector space of real valued (resp.~the cone of non-negative) Radon measures on $\sfX$ i.e.~the countably additive set functions $\mu: \Borel(\sfX) \to \R$ (resp. $\mu: \Borel(\sfX) \to [0, + \infty)$) s.t.
\begin{equation*} 
\text{ for every } B \in \Borel (\sfX) \text{ and every } \eps >0 \text{ there exists }K \subset B \text{ compact  s.t. } \left |\mu \right |(\sfX \setminus K)  < \eps,
\end{equation*}
where $|\mu|$ denotes the total variation measure of $\mu$.
We say that $\mu\in \meas_+(\sfX)$ is concentrated on a set
$\Gamma\subset \sfX$ if for
every $\eps>0$
there exists a Borel set $\Gamma_\eps\subset \Gamma$ such that
$\mu(\sfX\setminus \Gamma_\eps)<\eps$. The Radon property ensures
that if $\mu$ is concentrated on $\Gamma$ then it is concentrated on
a $\sigma$-compact subset $\Gamma'\subset \Gamma$. 

We denote by $\prob(\sfX)$ the set of probability measures on $\sfX$
i.e.~the elements $\mu$ of $\meas_+(\sfX)$ s.t. $\mu(\sfX)=1$. When $\sfX$ is a Polish space (i.e.~its topology is induced by
a metric $\mathsf d$ such that $(\sfX,\mathsf d)$ is complete and
separable),
then every finite Borel measure is Radon so that
$\meas(\sfX)$ coincides with the space of real valued Borel measures.

The zero measure in $\sfX$ is denoted by $\bm{0}_\sfX$. We define the sets of non-negative discrete measures and non-negative weighted Dirac masses respectively as
\begin{align*} 
\discr(\sfX) &:= \left \{ \sum_{j=1}^m c_i \delta_{x_j} : \{c_j\}_{j=1}^m \subset [0, + \infty) \, , \, \{x_j\}_{j=1}^m \subset \sfX \, , \, m \in \N \right \}, \\ 
\D{\sfX} &:= \left \{ r \delta_x : r \in [0, + \infty) \, , \, x \in \sfX \right \}.
\end{align*}
Notice that $\discr(\sfX), \D{\sfX} \subset \meas_+(\sfX)$. If $\sfX,\sfY$ are two completely regular spaces and $f:\sfX \to \sfY$ is a Borel function we denote by $f_{\sharp} : \meas(\sfX) \to \meas(\sfY)$ the push forward operator defined, for every $\mu \in \meas(\sfX)$, as
\[ f_{\sharp}\mu (B) = \mu(f^{-1}(B)) \quad \text{ for every } \, B \in \Borel(\sfY).\]
We denote by $\rmC_b(\sfX)$ the vector space of continuous and bounded real functions.
There is a natural duality pairing $\langle\cdot,\cdot\rangle$ between $\meas(\sfX)$ and $\rmC_b(\sfX)$ 
\begin{equation}
  \label{ss22:eq:2}
  \langle \mu,\varphi\rangle:=\int_\sfX \varphi\de\mu\quad
  \text{for every }\mu\in \meas(\sfX),\ \varphi\in \rmC_b(\sfX).
\end{equation}
\eqref{ss22:eq:2} defines a real nondegenerate bilinear form in $\meas(\sfX)\times
\rmC_b(\sfX)$, for if a Radon measure $\mu\in \meas(\sfX)$ satisfies 
$\int_\sfX \varphi\de\mu=0$ for every $\varphi\in \rmC_b(\sfX)$,
then $|\mu|(B)=0$ for every $B\in \Borel(\sfX)$
(e.g.~by the approximation result
\cite[Lemma 7.2.8]{Bogachev}) so that $\mu$
is the null measure.
Hence we can endow $\meas(\sfX)$ with
the weak Hausdorff topology $\sigma(\meas(\sfX),\rmC_b(\sfX))$:
the coarsest topology on $\meas(\sfX)$ for which the maps
$\mu \mapsto \int_\sfX \varphi \de \mu$ are continuous for every $\varphi \in \rmC_b(\sfX)$. 
Since in general $\meas(\sfX)$ is not first-countable (but $\meas_+(\sfX)$ is Polish, thus metrizable, if $\sfX$ is
Polish), 
we will mostly deal with general nets $(\mu_\lambda)_{\lambda\in
  \mathbb L}$, i.e.~maps $\lambda\to \mu_\lambda$ defined in a
directed set $\mathbb L$ with values in $\meas(\sfX)$, see
e.g.~\cite[][\S 4.3]{Folland99}.
By definition, a net $\net{\mu} \subset \meas(\sfX)$ converges to $\mu \in \meas(\sfX)$  in the weak topology if 
\begin{equation*}
\lim_{\lambda \in \mathbb{L}} \int_{\sfX} \varphi \de \mu_{\lambda} = \int_{\sfX} \varphi \de \mu \quad \text{ for every } \, \varphi \in \rmC_b(\sfX).
\end{equation*}
If $f: \sfX \to (-\infty, + \infty]$ is a function, we denote by $D(f)$ its effective domain, defined as
\[ D(f) := \left \{ x \in \sfX \mid f(x) < + \infty \right \}\]
and by $\bar{f}$ its lower semicontinuous envelope i.e.
\begin{equation}
\label{eq:lsc}
 \bar{f}(x) := \inf \left \{ \liminf_{\lambda} f(x_{\lambda}) : \net{x} \subset \sfX \, , \, \, x_{\lambda} \to x \right \}, \quad x \in \sfX;
\end{equation}
$\bar{f}$ is the largest lower semicontinuous function below $f$. If $\sfX$ is also a topological vector space and $A \subset \sfX$, we denote by $\ce{A}$ (resp.~by $\cce{A}$) the convex (resp.~closed and convex) envelope of $A$. If $g: \sfX \to (-\infty, + \infty]$ is a function, we denote its convex envelope and its closed convex envelope by $\ce{g}$ and $\cce{g}$, defined by
\begin{align}
 \ce{g}(x)&:= \inf \left \{ \sum_{i=1}^n \alpha_i g(x_i) : \{x_i\}_{i=1}^n \subset \sfX, \, \{\alpha_i\}_{i=1}^n \in S_n,\, \sum_{i=1}^n \alpha_i x_i = x, \, n \ge 1 \right \},\ x \in \sfX
 \label{eq:co}\\
 \cce{g}&:=\overline{\ce{g}},
 \label{eq:cco}
 \end{align}
 where
 \[ S_n := \left \{ \{\alpha_i\}_{i=1}^n : \alpha_i \in [0,1], \, \sum_{i=1}^n \alpha_i = 1 \right \}.\]
 $\ce{g}$ is the largest convex function below $g$ and $\cce{g}$ is the largest lower semicontinuous and convex function below $g$. The following proposition is a simple density result whose proof easily follows adapting \cite[Example 8.1.6]{Bogachev}.
  \begin{proposition}\label{ss22:prop:density}
Let $\sfX$ be a completely regular space. Then for every $c \ge 0$ we have
\[
  \overline{\{ \mu \in   \discr(\sfX) : \mu(\sfX)
    = c \}}
  = \{ \mu \in \meas_+(\sfX) : \mu(\sfX) = c \},
  \quad
  \overline{\discr(\sfX)}= \meas_+(\sfX).
\]
\end{proposition}
\begin{remark}  It holds 
\[ \cce{\D{\sfX}} =\meas_+(\sfX). \]
This is an immediate consequence of the fact that $\ce{\D{\sfX}} = \discr(\sfX)$ and Proposition \ref{ss22:prop:density}.
\end{remark}
The following Lemma is a refinement of Proposition \ref{ss22:prop:density} showing that, given a Borel function $f$ and a non-negative measure $\alpha$, we can construct an approximating sequence of discrete measures for which we have convergence also of the integral of $f$.

\begin{lemma}\label{ss22:lemma:reti}
Let $\sfX$ be a completely regular space and let $\alpha \in
\meas_+(\sfX)$. Let $f: \sfX \to [0, + \infty]$ be a Borel
function. Then there exists a net $\net{\ggamma} \subset
 \{ \mu \in \discr(\sfX) : \mu(\sfX) = \alpha(\sfX)\}$ s.t.
\[ \lim_{\lambda \in \mathbb{L}} \ggamma_{\lambda} = \alpha, \quad \lim_{\lambda \in \mathbb{L} } \int_\sfX f \de \ggamma_{\lambda} = \int_\sfX f \de \alpha.\]
\end{lemma}
\begin{proof}
By Lusin's theorem, we can find an increasing sequence of compact sets such that
\[
 \sfX_k \subset \sfX_{k+1} \quad \text{ for every } \, k \ge 1, \quad \alpha(\sfX \setminus \sfX_k) \le \frac{1}{k}, \quad \restricts{f}{\sfX_k} \text{ is bounded and continuous.}
\]
Consider now the family of measures $\{\alpha_k\}_{k \ge 1} \subset \{ \mu \in \meas_+(\sfX) : \mu(\sfX) = \alpha(\sfX)\}$ defined as
\[ \alpha_k := \frac{\alpha(\sfX)}{\alpha(\sfX_k)} \restricts{\alpha}{\sfX_k} \quad \text{ for every } \, k \ge 1.\]
We can easily observe that
\[
\lim_{k \to +\infty} \alpha_k = \alpha
\]
indeed, if $\varphi \in \rmC_b(\sfX)$, we have
\[ \lim_{k \to +\infty} \int_{\sfX} \varphi \de \alpha_k = \lim_{k \to +\infty} \frac{\alpha(\sfX)}{\alpha(\sfX_k)}  \int_\sfX \varphi \nchi_{\sfX_k} \de \alpha = \int_\sfX \varphi \de \alpha\]
by monotone convergence. The same argument shows that we also have
\begin{equation} \label{ss22:alpha}
\lim_{k \to +\infty} \int_\sfX f \de \alpha_k = \int_\sfX f \de \alpha.
 \end{equation}
By Proposition \ref{ss22:prop:density}, for every $k \ge 1$, we can find a net $\{\ggamma^k_{\lambda}\}_{\lambda \in \mathbb{L}_k} \subset \discr(\sfX_k)  \cap \{ \mu \in \meas_+(\sfX) : \mu(\sfX) = \alpha(\sfX)\}$, such that
\[ \lim_{\lambda \in \mathbb{L}_k } \ggamma^k_{\lambda} = \alpha_k.\]
Moreover, since $\restricts{f}{\sfX_k}$ is bounded and continuous, it holds
\[ \lim_{\lambda \in \mathbb{L}_k} \int_\sfX f \de \ggamma^k_{\lambda} = \lim_{\lambda \in \mathbb{L}_k} \int_{\sfX_k} f \de \ggamma^k_{\lambda} = \int_{\sfX_k} f \de \alpha_k = \int_\sfX f \de \alpha_k.\]
This allows us to find, for every $k \ge 1$, some $\bar{m}(k) \in \mathbb{L}_k$ s.t. 
\[ \left | \int_\sfX f \de \ggamma^k_{\lambda} - \int_\sfX f \de \alpha_k \right | \le \frac{1}{k} \quad \text{ for every } \, \lambda \ge \bar{m}(k).\]
Hence we can consider, for every $k \ge 1$, the directed sets $\mathbb{E}_k := \{ \lambda \in \mathbb{L}_k : \lambda_k \ge \bar{m}(k) \}$ and the corresponding new sequence of nets $\{ \ggamma^k_{\lambda}\}_{\lambda \in \mathbb{E}_k}$. Obviously it holds
\begin{equation} \label{ss22:1k}
\lim_{\lambda \in \mathbb{E}_k} \ggamma^k_{\lambda} = \alpha_k, \quad \left | \int_\sfX f \de \ggamma^k_{\lambda} - \int_\sfX f \de \alpha_k \right | \le \frac{1}{k} \quad \text{ for every } \, \lambda \in \mathbb{E}_k.
\end{equation}
Define now the directed set 
\[ \N \otimes \mathbb{E}_k := \{ (k, \lambda) : \lambda \in \mathbb{E}_k \} \text{ with order } (k,\lambda) \le (k',\lambda') \iff k < k' \text{ or } (k=k' \land \lambda \le \lambda').\]
By the diagonal principle for nets, we can find a directed set $\mathbb{B}$ and a monotone final function 
\[ h: \mathbb{B} \to \N \otimes \mathbb{E}_k, \quad h(\beta) = (h_1(\beta), h_2(\beta)) \text{ with } h_2(\beta) \in \mathbb{E}_{h_1(\beta)} \quad \text{ for every } \, \beta \in \mathbb{B}\]
such that the diagonal net $\{ \ggamma_{\beta} \}_{\beta \in \mathbb{B}} := \{ \ggamma^{h_1(\beta)}_{h_2(\beta)} \}_{\beta \in \mathbb{B}} \subset \discr(\sfX) \cap \{ \mu \in \meas_+(\sfX) : \mu(\sfX) = \alpha(\sfX)\}$ converges to $\alpha$. We only need to prove that also the integral of $f$ converges:
\begin{align*}
\left | \int_\sfX f \de \ggamma_{\beta} - \int_\sfX f \de \alpha \right | &\le \left | \int_\sfX f \de \ggamma_{\beta} - \int_\sfX f \de \alpha_{h_1(\beta)} \right | + \left | \int_\sfX f \de \alpha_{h_1(\beta)} -  \int_\sfX f \de \alpha \right |\\
&= \left | \int_\sfX f \de \ggamma^{h_1(\beta)}_{h_2(\beta)} - \int_\sfX f \de \alpha_{h_1(\beta)} \right | + \left | \int_\sfX f \de \alpha_{h_1(\beta)} -  \int_\sfX f \de \alpha \right |\\
& \le \frac{1}{h_1(\beta)} + \left | \int_\sfX f \de \alpha_{h_1(\beta)} -  \int_\sfX f \de \alpha \right |, 
\end{align*}
where we have used \eqref{ss22:1k}. Now it is enough to observe that $h_1 : \mathbb{B} \to \N$ is a final monotone function i.e. it is an increasing monotone sequence converging to $+ \infty$. Passing to $\lim_{\beta \in \mathbb{B}}$ and using \eqref{ss22:alpha}, we conclude.
\end{proof}

\subsection{The cone construction}\label{sec:theconee}
It will be natural to state some definitions and results in the context of the so called geometric cone: we introduce on $\sfX \times \R_+$ the equivalence relation 
\begin{equation}
\label{eq:cone-relation}
(x,r) \sim (y,s) \quad \overset{\text{def}}{\Leftrightarrow} \quad \left [ x=y, r=s \ne 0 \quad \vee \quad r=s=0 \right ] 
\end{equation}
and the corresponding geometric cone $\f{C}[\sfX]:=(\sfX \times \R_+)/\sim$, whose points are denoted by gothic letters as $\f{y}$. We denote by $\f{p}$ the quotient map $\f{p}: \sfX \times \R_+ \to \f{C}[\sfX]$ sending a point $(x,r)$ to its equivalence class $[x,r]$. Notice that $\f{p}$ is just the identity map except for those points with $r=0$, which are all sent to the same equivalence class, the so called vertex of the cone, that we denote with $\f{o}$. We will consider the natural product operation on the cone given by
\[ \lambda [x,r] := [x,\lambda r] \quad \text{ for every } \lambda, r \ge 0, \, x \in \sfX.\]
 
In certain cases it is useful to 
add a further isolated point $o$ 
(which plays the role of a source or a sink) 
to $\sfX$ and to consider the spaces $\sfX_o:=
\sfX\sqcup\{o\}$ and
$\sfY_o:= (\sfX\times \R_+)\sqcup 
\{(o,0)\}\subset \sfX_o\times \R_+$. We can identify  $\sfX$ (resp.~$\sfX \times \R_+$) 
with a (closed) subset of $\sfX_o$ (resp.~$\sfY_o$) and 
$\meas_+(\sfX)$ (resp.~$\meas_+(\sfX\times \R_+)$) with the (closed) subset 
of $\meas_+(\sfX_o)$ (resp.~$\meas_+(\sfY_o)$) given by
all the measures not charging $\{o\}$
(resp.~$\{(o,0)\}$). 
Notice that $\f C[\sfX]=\sfY_o/\sim$
and we can trivially extend $\f p$ to $\sfY_o$.

On the cone we introduce the projections on $\R_+$ and $\sfX_o$ simply
defined as $\sfr([x,r]) = r$ and $\sfx([x,r]) = x$ if $r>0$ and
$\sfx([x,r])=
\sfx(\f o)=o$ if $r=0$.

We can define a right inverse of $\f{p}:\sfY_o\to \f C[\sfX]$ as
\[
\f{q}([x,r]) = (\sfx([x,r]), \sfr([x,r])).
\]
Notice that also that $\f{q}(\f C[\sfX]\setminus \{\fo\})=\sfX\times (0,+\infty).$ 

On $\f{C}[\sfX]$ we consider the following topology, weaker then the quotient one: a local system of neighbourhoods of a point $[x,r]$ is just the image trough $\f{p}$ of the local system of neighbourhoods given by the product topology at $(x,r) \in \sfX \times \R_+$, if $r>0$. A local system of neighbourhoods at $\f{0}$ is given by
\begin{equation}
\label{eq:o-neighborhoods}
\left \{ \left \{ [x,r] \in \f{C}[\sfX] : 0 \le r < \eps \right \} \right \}_{\eps >0}.
\end{equation}
If the topology of $\sfX$ is induced by a metric $\sfd$, then the topology of $\f{C}[\sfX]$ is induced by the metric $\sfd_\f{C}: \f{C}[\sfX]\times \f{C}[\sfX] \to [0, +\infty)$ defined as
\begin{equation}\label{ss22:eq:distcone} \sfd_\f{C}([x,r],[y,s]) := \left ( r^2+s^2-2rs\cos(\sfd(x,y) \wedge \pi) \right )^{\frac{1}{2}}, \quad [x,r], [y,s] \in \f{C}[\sfX].
\end{equation}
With the above topology, $\f{C}[\sfX]$ is completely regular and it is the right object to consider when one wants to represent elements in $\Delta_+(\sfX)$; in particular we have the following result.
\begin{lemma} \label{ss22:lem:homeo} Let $\sfX$ be a completely regular space. Then $\Delta_+(\sfX)$ is homeomorphic to $\f{C}[\sfX]$.
\end{lemma}
\begin{proof} The homeomorphism is given by the map $\varphi : \f{C}[\sfX] \to \Delta_+(\sfX)$ defined as
\[ \varphi([x,r]) := \begin{cases} r\delta_x \quad &\text{ if } r>0, \\ \bm{0}_\sfX \quad &\text{ if } r=0. \end{cases}\]
\end{proof}
\begin{remark}\label{rem:ossessivo}
Let us point out an obvious but useful fact that we will extensively use in the following without having to recall it further:
if a measure $\bbeta\in \meas_+\big(\f C[\sfX]\big)$
does not charge $\{\fo\}$
then 
$\sfx_\sharp \bbeta$ 
does not charge $\{o\}$
and therefore we will identify it with a measure
in $\meas_+(\sfX)$. 
In particular, this construction does not depend on the choice of the point $o$
and the map $\sfx$ takes value in $\sfX$ $\bbeta$-a.e.
\end{remark}
If $R>0$, we define 
\begin{equation} \label{ss22:eq:coner}
    \f{C}_R[\sfX] := \left \{ [x,r] \in \f{C}[\sfX] : 0 \le r \le R \right \}
\end{equation}
and we will often identify measures on $\f{C}[\sfX]$ with support contained in $\f{C}_R[\sfX]$ with elements of $\meas(\f{C}_R[\sfX])$.
For every $p \ge 1$, we introduce moreover the set
\begin{equation*}
\mathcal{\f{M}}_+^p(\f{C}[\sfX]) := \left \{ \alpha \in \meas_+(\f{C}[\sfX]) : \int_{\f{C}[\sfX]} \sfr^p \de \alpha < + \infty  \right \}, 
\end{equation*}
and the map (see the above Remark
\ref{rem:ossessivo})
\begin{equation*}
\f{h}^p : \mathcal{\f{M}}_+^p(\f{C}[\sfX]) \to \meas_+(\sfX), \quad \f{h}(\alpha) = \sfx_\sharp (\sfr^p \alpha).
\end{equation*}

We stress the fact that the map $\f{h}^p$ does not depend on the point $o$ 
occuring in the definition of $\sfx$. \\

\noindent
We introduce now the product cone: given $\sfX_1$ and $\sfX_2$ completely regular spaces, we define $\pc:= \f{C}[\sfX_1]\times \f{C}[\sfX_2]$ endowed with the product topology. Points in the product cone are denoted by bold gothic letters as $\bm{\f{y}}= (\f{y}_1, \f{y}_2)= ([x_1, r_1], [x_2, r_2])$. We denote with $\f o_i$ 
(resp.~$o_i$) the vertex of $\f C[\sfX_i]$ (resp.~isolated points added to $\sfX_i$
and forming the disjoint union $\sfX_{i,o}:=
\sfX_i\sqcup\{o_i\}$),
$i=1,2$ and we set $\boldsymbol\fo:=(\fo_1,\fo_2),$
$\boldsymbol o:=(o_1,o_2),$
$\f C_{\boldsymbol \fo}[\sfX_1,\sfX_2]:=
\pc\setminus \{\boldsymbol\fo\}.$

On the product cone we can consider the projections on the two
components
$\pi^i : \pc \to \f{C}[\sfX_i]$ sending $([x_1,
r_1], [x_2, r_2])$ to $[x_i, r_i]$ and the projections on $\R_+$ and
$\sfX_i$ simply defined as $\sfr _i := \sfr \circ
\pi^i$ and $\sfx_i := \sfx \circ \pi^i$ ($\sfx_i$ maps $\f C[\sfX_1,\sfX_2]$ into
$\sfX_{i,o}$).
In analogy with \eqref{ss22:eq:distcone}, if the topologies of $\sfX_1$ and $\sfX_2$ are induced by distances $\sfd_1$ and $\sfd_2$ respectively, the topology of the product cone is induced by the distance
\begin{equation}\label{ss22:eq:prcdist}
 (\sfd_1 \otimes_{\f{C}} \sfd_2)((\f{y}_1, \f{y}_2),(\f{w}_1, \f{w_2})) := \left (\sfd_{1,\f{C}}^2(\f{y}_1, \f{w}_1) + \sfd_{2,\f{C}}^2(\f{y}_2, \f{w}_2)  \right )^{\frac{1}{2}}, \quad (\f{y}_1, \f{y}_2), (\f{w}_1, \f{w}_2) \in \pc.
 \end{equation}
As in \eqref{ss22:eq:coner}, given $R>0$, we define
\[
    \pcr{R} := \f{C}_R[\sfX_1] \times \f{C}_R[\sfX_2] = \left \{
      \bm{\f{y}} \in \pc : 0 \le \sfr_i(\bm{\f{y}}) \le R, \, i=1,2\right \}
\]
and we will identify measures on $\pc$ with support contained in $\pcr{R}$ with elements of $\meas(\pcr{R})$.
For every $p \ge 1$, we introduce the set
\begin{equation*}
\mathcal{\f{M}}_+^p(\pc) := \left \{ \aalpha \in \meas_+(\pc) : \int_{\pc} (\sfr_1^p + \sfr_2^p) \de \aalpha < + \infty  \right \}, 
\end{equation*}
and the maps (see the above Remark \ref{rem:ossessivo})
\begin{equation*}
\f{h}_i^p : \mathcal{\f{M}}_+^p(\pc) \to \meas_+(\sfX_i), \quad \f{h}_i^p(\aalpha) = (\sfx_i)_\sharp (\sfr_i^p \aalpha) \quad i=1,2.
\end{equation*}
Finally we define, for every $(\mu_1, \mu_2) \in \meas_+(\sfX_1) \times \meas_+(\sfX_2)$ and every $p \ge 1$, the set 
\begin{equation}\label{eq:hommarg}
\f{H}^p(\mu_1, \mu_2) := \left \{ \aalpha \in  \mathcal{\f{M}}_+^p(\pc) : \f{h}_i^p(\aalpha) = \mu_i, \, i=1,2\right \}.
\end{equation}
If $\aalpha \in \f{H}^p(\mu_1, \mu_2)$, we say that $\mu_1$ and
$\mu_2$ are the $p$-homogeneous marginals of $\aalpha$ and that $\aalpha$ is a $p$-homogeneous plan/coupling between $\mu_1$ and $\mu_2$. 
Notice that if $\aalpha \in \f{H}^p(\mu_1, \mu_2)$
then also
the restriction of $\aalpha$ to
$\pco$, $\aalpha_{\bfo}:=\aalpha\mres \pco$,
belongs to $\f{H}^p(\mu_1, \mu_2)$.

\begin{remark}[Trivial cases]
    \label{rem:trivial-cases}
    If $\mu_i(\sfX_i)=0$ for $i=1,2$ ($\mu_i$ are the null
    measures in $\sfX_i$), then 
    $\f{H}^p(\mu_1, \mu_2)=\{\lambda \delta_{\bfo}:\lambda\ge 0\}.$

    If $\mu_1(\sfX_1)>\mu_2(\sfX_2)=0$
    then 
    $\f{H}^p(\mu_1, \mu_2)
    = \left \{\alpha\otimes \delta_{\fo_2}:
    \alpha\in \meas_+(\f C[\sfX_1]),\ 
    \f h^p(\alpha)=\mu_1\right \}.$
    A similar characterization holds 
    when 
    $0=\mu_1(\sfX_1)<\mu_2(\sfX_2).$

    We say that $(\mu_1,\mu_2)$ is a non-trivial pair
    if $\mu_i(\sfX_i)>0$ for $i=1,2$.
\end{remark}
\begin{remark}
\label{rem:invariance1}
    The class $\f H^p(\mu_1,\mu_2)$ 
    satisfies a natural invariance property with respect 
    to inclusion of the ambient spaces $\sfX_i$. 
    First of all,
    if $\sfX_i$ are subsets of 
    completely regular spaces
    $\tilde \sfX_i$
    (and the topology of
    $\sfX_i$ coincides with the
    relative topology induced by the inclusion in $\tilde \sfX_i$)
    there is a natural 
    identification between 
    (pair of) Radon measures
    $\mu_i\in \meas_+(\sfX_i)$
    and (pair of) Radon measures
    $\tilde \mu_i\in \meas_+(\tilde \sfX_i)$
    concentrated on $\sfX_i$
    (i.e.~such that $\tilde\sfX_i\setminus 
    \sfX_i$ is $\mu_i$-negligible).
    Since $\f C[\sfX_1,\sfX_2]$ can be considered as a subset of 
    $\f C[\tilde\sfX_1,\tilde\sfX_2]$ (with identification of the corresponding vertexes $\f o_i$ and $\tilde{\f o}_i$)
    every homogeneous coupling
    $\aalpha\in \f H^p(\mu_1,\mu_2)$
    can be considered as a measure
    in $\meas_+(\f C[\tilde\sfX_1,\tilde\sfX_2])$ 
    which belongs to 
    $\f H^p(\tilde\mu_1,\tilde\mu_2)$.
    Conversely, 
    if $\tilde\mu_i\in \meas_+(\tilde\sfX_i)$ 
    are concentrated 
    in $\sfX_i$
    then every plan $\tilde\aalpha\in \f H^p(\tilde\mu_1,\tilde\mu_2)$
    is concentrated in 
    $\f C[\sfX_1,\sfX_2]$
    and its restriction to $\f C[\sfX_1,\sfX_2]$
    belongs to $\f H^p(\mu_1,\mu_2).$
    In particular, recalling 
    the construction 
    of the spaces $\sfX_{i,o}$
    at the beginning of Section
    \ref{sec:theconee}, for every
    pair of measures $\mu_i\in \meas_+(\sfX_i)$
    the set
    $\f H^p(\mu_1,\mu_2)$
    can also be considered 
    as a subset of
    $\meas_+\big(\f C[\sfX_{1,o},\sfX_{2,o}]\big)$.
\end{remark}
 
The following renormalization result comes from \cite{LMS18}.
\begin{lemma} \label{ss22:le:dialations} Let $\sfX_i$ for $i=1, 2$ be completely regular spaces and let $p \ge 1$. Given $\aalpha \in \meas_+(\pc)$ and $\vartheta: \pc \to (0, + \infty)$ Borel measurable in $L^p(\pc,\aalpha)$ we can define
\begin{align*}
\prd_{\vartheta}(\bm{\f{y}}) := \left ( \vartheta(\bm{\f{y}})^{-1} \f{y}_1, \vartheta(\bm{\f{y}})^{-1} \f{y}_2 \right ), \quad \bm{\f{y}} \in \pc, \quad \quad \quad 
\dil_{\vartheta, p}(\aalpha) :=(\prd_{\vartheta})_\sharp (\vartheta^p \aalpha).
\end{align*}
Then we have
\[
\f{h}_i(\dil_{\vartheta,p}(\aalpha)) = \f{h}_i^p(\aalpha), \quad i=1,2.
\]
In particular, if we define
\[
 \vartheta_{\aalpha, p}(\bm{\f{y}}) := \frac{1}{r^*(\aalpha)} \begin{cases} \sfr_1^p(\bm{\f{y}}) + \sfr_2^p(\bm{\f{y}}) \quad &\text{ if } \bm{\f{y}} \ne (\f o_1, \f o_2) \\ 1 \quad &\text{ if } \bm{\f{y}} = (\f o_1, \f o_2) \end{cases}, 
\]
where $r^*(\aalpha)$ is a normalization constant s.t.~$\int_{\pc}\vartheta_{\aalpha, p}^p \de \aalpha =1$ given by
\[
r^*(\aalpha) := \int_{\pc} (\sfr_1^p + \sfr_2^p) \de \aalpha + \aalpha(\{(\f o_1, \f o_2)\}),
\]
we have that $\dil_{\vartheta_{\aalpha, p},p}(\aalpha) \in \prob(\pc)$, it has the same $p$-homogeneous marginals of $\aalpha$ and its support is contained in $\pcr{r^*(\aalpha)}$.
\end{lemma}

The following Lemma is a compactness result for homogeneous marginals (the corresponding statement for classical marginals is \cite[Theorem 3.1]{SS20}). 
\begin{lemma} \label{ss22:lemma:generalmarginals} Let $p \ge 1$, $R>0$ and let $\sfX, \sfX_i$, $i=1, 2$ be completely regular spaces. Then:
\begin{enumerate}
    \item if $\net{\alpha}$ is a net in $\prob(\f{C}_R[\sfX])$ such that $\mu_\lambda:= \f{h}^p(\alpha_\lambda)\in \meas_+(\sfX)$, $\lambda \in \mathbb{L}$, converges to some $\mu \in \meas_+(\sfX)$, then there exists a subnet $(\alpha'_\beta)_{\beta \in\mathbb B}$ of $\net \alpha$ convergent to some $\alpha \in \prob(\f{C}_R[\sfX])$ with $\f{h}^p(\alpha) = \mu$;
    \item if $\net \aalpha$ is a net in $\prob(\pcr{R})$ such that $\mu_{i,\lambda}:=\f{h}_i^p(\aalpha_\lambda)\in \meas_+(\sfX_i)$,
  $i=1,2$, $\lambda\in \mathbb L$, converge to some $\mu_i$ in $\meas(\sfX_i)$, 
  then there exists a subnet $(\aalpha'_\beta)_{\beta \in\mathbb B}$ of $\net \aalpha$
  convergent to some
  $\aalpha \in \f{H}^p(\mu_1, \mu_2)$.
\end{enumerate}
  \end{lemma}
\begin{proof} Thanks to \cite[Theorem 3.1]{SS20}, it is enough to prove only the first claim. Define, for every $\lambda \in \mathbb{L}$,
\begin{equation*}
\vartheta_{\lambda} := \f{q}_\sharp (\sfr^p \alpha_{\lambda}) \in \meas_+(\sfX \times [0,R]).
\end{equation*}
Notice that this definition does not depend on the point $\bar{x}$
w.r.t.~$\f{q}$ is defined. Observe that $\pi^{[0,R]}_\sharp
\vartheta_{\lambda} \in \meas_+([0,R])$ with mass bounded by $R^p$ and
$\pi^\sfX_\sharp \vartheta_{\lambda} = \mu_{\lambda}$. Then we can
apply \cite[Theorem 3.1]{SS20} to $\net{\vartheta}$ and obtain that,
up to passing to a subnet, there exists
$\vartheta \in \meas_+(\sfX \times [0,R])$ s.t. $\lim_{\lambda \in \mathbb{L}} \vartheta_{\lambda} = \vartheta$. Now we define
\[
O_n := \left \{ [x,r] \in \f{C}[\sfX] : 0 \le r \le \frac{1}{n} \right \}, \quad n \ge 1
\]
and, for every $n \ge 1$, the nets of real numbers 
\[
m_{\lambda, n} := \aalpha_{\lambda}(O_n).
\]
Observe that $0 \le m_{\lambda, n} \le 1$ for every $n \ge 1$ and $\lambda \in \mathbb{L}$ then, up to passing to a subnet (the same for every $n \in \N$), they converge in $\lambda \in \mathbb{L}$ to some $m_n \in [0,1]$. Define then $m := \inf_{n \ge 1} m_n$. We claim then that 
\[
\lim_{\lambda \in \mathbb{L}} \aalpha_{\lambda} = \frac{1}{\sfr^p} \f{p}_\sharp \vartheta + m \delta_{\f{o}} =: \aalpha.
\]
Take any $\Omega \subset \f{C}[\sfX]$ open; if $\f{o} \notin \Omega$, we have
\begin{align*}
\liminf_{\lambda \in \mathbb{L}} \alpha_{\lambda}(\Omega) &= \liminf_{\lambda \in \mathbb{L}} \int_{\sfX \times [0,R]} (\nchi_{\Omega} \circ \f{q})(x,r) \frac{1}{(\sfr^p \circ \f{q})(x,r)} \de \vartheta_{\lambda}(x,r)\\
&\ge \int_{\sfX \times [0,R]} (\nchi_{\Omega} \circ \f{q})(x,r) \frac{1}{(\sfr^p \circ \f{q})(x,r)} \de \vartheta(x,r)\\
&= \aalpha(\Omega).
\end{align*}
If, on the other hand, $\f{o} \in \Omega$, we have that $O_N \subset \Omega$ for some $N \ge 1$; calling $\Omega_n := \Omega \setminus O_n$ (which is an open set), we have, for every $n \ge N$, that 
\begin{align*}
\liminf_{\lambda \in \mathbb{L}} \aalpha_{\lambda} (\Omega) &\ge \liminf_{\lambda \in \mathbb{L}} \aalpha_{\lambda}(O_n) +  \liminf_{\lambda \in \mathbb{L}} \aalpha_{\lambda}(\Omega_n) \\
&\ge \aalpha(\Omega_n) + m_n.
\end{align*}
Now we pass to the limit as $n \to + \infty$ and, using the monotone convergence theorem and the fact that $\Omega_n \uparrow \Omega \setminus \{\f{o}\}$, we obtain
\[
\liminf_{\lambda \in \mathbb{L}} \aalpha_{\lambda} (\Omega) \ge \aalpha(\Omega \setminus \{\f{o}\}) + m = \aalpha(\Omega),
\]
and this concludes the proof thanks to Portmanteau theorem (see e.g.~\cite[Corollary 8.2.10]{Bogachev}).
 \end{proof}

\subsection{Functions on the product cone} We describe a few
properties of functions defined on the product cone that will be
useful in the sequel. In this subsection $\sfX_1$ and $\sfX_2$ are
completely regular spaces.
Recall that the indicator function
$\mathsf I_G:X\to \R\cup\{+\infty\}$
associated with 
a subset $G$ of a set $X$ is defined by
\[
    \mathsf I_G(x):=
    \begin{cases}
      0&\text{if }x \in G\\
      +\infty&\text{otherwise}
    \end{cases},\quad
    x\in X.
\]

\begin{definition} \label{ss22:def:cof} Consider a function $\sfH: \pc \to [0, +
  \infty]$ and a subset $\Gamma\subset \pc$.
  For every $(x_1, x_2) \in \sfX_1 \times \sfX_2$ we define
\begin{align*}
    \sfH_{x_1, x_2} &:{}\, \,  \R_+^2 \to [0, + \infty]:\qquad
                                                       (r_1, r_2) \mapsto \sfH([x_1, r_1], [x_2, r_2]),\\
  \Gamma_{x_1,x_2}&:={}\Big\{(r_1,r_2)\in \R^2_+:([x_1,r_1],[x_2,r_2])\in \Gamma\Big\}.
\end{align*}
We say that
\begin{itemize}
\item $\sfH$ is radially 
   $p$-homogeneous, $p \in [1,+\infty)$,  if $\sfH_{x_1, x_2}$ is positively  $p$-homogeneous for every $(x_1, x_2) \in \sfX_1 \times \sfX_2$ i.e.
    \[ \sfH_{x_1,x_2}(\lambda r_1, \lambda r_2) = \lambda^{ p} \sfH_{x_1,x_2}(r_1, r_2) \quad \text{ for every } \lambda >0, \, (r_1, r_2) \in \R_+^2;\] 
  \item $\sfH$ is radially
    convex if $\sfH_{x_1, x_2}$ is convex for every $(x_1, x_2) \in \sfX_1 \times \sfX_2$.
  \end{itemize}
   We say that $\Gamma\subset \pc$ is a radial cone
  (resp.~radial convex set) if its indicator function
  is radially
  $1$-homogeneous (resp.~radially convex).
Notice that we do not assume
  that 
  a radially $p$-homogenoeus function
  vanishes at $(\fo_1,\fo_2)$ 
  (or, similarly, that a radial cone contains $(\fo_1,\fo_2)$);
  however such a property follows immediately 
  under lower semicontinuity and properness conditions, see
  Remark 
  \ref{rem:non-si-finisce-mai-di-dimenticare} below.  
  \\
We define the (radially)
$1$-homogeneous, convex, and closed convex envelopes $\hmg{\sfH}, \ce{\sfH}, \cce{\sfH} : \pc \to [0, + \infty]$ of $\sfH$ as
(recall (\ref{eq:lsc},\ref{eq:co},\ref{eq:cco}))
\begin{align*}
  \hmg{\sfH}([x_1, r_1], [x_2, r_2]) &:=
\inf_{\lambda\ge0}\sfH([x_1,\lambda                                       r_1],[x_2,\lambda r_2]),\\
  \ce{\sfH}([x_1, r_1], [x_2, r_2]) &:= \ce{\sfH_{x_1, x_2}}(r_1, r_2) \quad \text{ for every } (x_1, x_2) \in \sfX_1 \times \sfX_2, \, (r_1, r_2) \in \R_+^2, \\
    \cce{\sfH} &:= \overline{ \ce{\sfH}}.
\end{align*}
Similarly, in the case of $\Gamma\subset \pc$, we define  $\hmg\Gamma,
\ce\Gamma,$
and $\cce\Gamma$ so that $\mathsf I_{\hmg \Gamma}:=\hmg{\mathsf
  I_\Gamma}$,
$\mathsf I_{\ce \Gamma}:=\ce{\mathsf
  I_\Gamma}$ and
$\mathsf I_{\cce \Gamma}:=\cce{\mathsf   I_\Gamma}$ respectively.
\end{definition}

\begin{remark}
\label{rem:non-si-finisce-mai-di-dimenticare}
 We added the terms \emph{radial and radially} just to avoid
  ambiguities in the case when $\sfX_1,\sfX_2$ are linear spaces and
  the notions of $1$-homogeneity and convexity could also refer to the joint behaviour of
  $\sfH$ w.r.t.~all the variables $x_i,r_i$.
  In this paper, we will
  always interpret convexity and $1$-homogeneity w.r.t.~the radial
  variables $r_i$.\\
  Notice that $\hmg\Gamma$ (resp.~$\ce\Gamma,\cce\Gamma$)
  is the set whose sections $\hmg\Gamma_{x_1,x_2}$ are
  the cone (resp.~convex, closed convex) envelopes of the corresponding
  sections $\Gamma_{x_1,x_2}$ of $\Gamma$. In particular, racalling Carath\'eodory
  Theorem in $\R^2$, we have
  \begin{align*}
    \hmg\Gamma:={}&\bigcup_{\lambda>0}\Big\{([x_1,\lambda
                    r_1],[x_2,\lambda r_2]): ([x_1,r_1],[x_2,r_2])\in \Gamma\Big\},\\
    \ce\Gamma:={}&\bigcup\Big\{\Big(\big[x_1,\sum_{i=0}^2 \alpha_i r_1^i\big]
                   ,\big[x_2,\sum_{i=0}^2 \alpha_ir_2^i\big]\Big):
                   ([x_1,r_1^i],[x_2,r_2^i])\in \Gamma,\
                   \alpha_i\ge0,\ \sum_{i=0}^2\alpha^i=1\Big\}.
  \end{align*}
  We also note that if $\sfH$ is proper (i.e.~not identically $+\infty$), lower semicontinuous, and radially
  $1$-homogeneous, then $\sfH(\f{0}_1, \f{0}_2)=0$.
\end{remark}

 The following result is a simple consequence of the $1$-homogeneity property.
\begin{lemma}  \label{ss22:le:mineq} Let $\sfX_i$, $i=1, 2$ be completely regular spaces, let $\sfH: \pc \to [0, + \infty]$ be a radially $1$-homogeneous Borel function 
and let $(\mu_1, \mu_2) \in \meas_+(\sfX_1) \times \meas_+(\sfX_2)$. Then 
\[ \inf \left \{ \int\sfH \de \aalpha : \aalpha \in \f{H}^1(\mu_1, \mu_2) \right \} = \inf \left \{ \int \sfH \de \aalpha : \aalpha \in \f{H}^1(\mu_1, \mu_2) \cap \prob(\pcr{R(\mu_1, \mu_2)}) \right \}, \]
where 
\begin{equation}\label{eq:theradi}
R(\mu_1, \mu_2):= \mu_1(\sfX_1) + \mu_2(\sfX_2).   
\end{equation}
\end{lemma}
\begin{proof}
It is of course enough to prove the $\ge$ inequality. If $\aalpha \in \f{H}^1(\mu_1, \mu_2)$ we can assume that $\aalpha(\{(\f o_1, \f o_2)\}) =0$
(if not, we can replace $\aalpha$ with 
$\aalpha_{\bfo}=\aalpha\mres \pco$ which has the same homogeneous marginals and a lower $\sfH$-cost, 
since $\sfH(\f o_1, \f o_2)\ge 0$). 
By Lemma \ref{ss22:le:dialations}, we have that \[ \tilde{\aalpha}:=\dil_{\vartheta_{\aalpha,1}, 1}(\aalpha) \in \f{H}^1(\mu_1, \mu_2) \cap \prob(\pcr{R(\mu_1, \mu_2)}) \]
and
\[ \int_{\pc} \sfH \de \aalpha = \int_{\pc} \sfH \de \tilde{\aalpha}.\]
This concludes the proof.
\end{proof}

\begin{remark}[$1$-homogeneity and $q$-homogeneity]\label{rem:tp} 
Given $q \in [1,+\infty)$, we define the map $\mathsf T_q: \pc \to \pc$ as
\begin{equation}
    \label{eq:tq}
\mathsf T_q([x_1,r_1],[x_2,r_2]) := ([x_1,r_1^{1/q}], [x_2, r_2^{1/q}]).
\end{equation} 
It is easy to check that $(\mathsf T_q)_\sharp$ is a bijective transformation from $\f{H}^1(\mu_1, \mu_2)$ to $\f{H}^q(\mu_1, \mu_2)$ for any pair $(\mu_1, \mu_2) \in \meas_+(\sfX_1) \times \meas_+(\sfX_2)$. Moreover, if $\sfH_q: \pc \to [0,\+\infty]$ is radially $q$-homogeneous, then $\sfH:=\sfH_q \circ \mathsf T_q$ is radially $1$-homogeneous and 
\begin{equation}
    \label{eq:qhom}
\inf \left \{ \int_{\pc} \sfH_q \de \aalpha_q : \aalpha_q \in \f{H}^q(\mu_1, \mu_2) \right \} = \inf \left \{ \int_{\pc} \sfH
\de \aalpha : \aalpha \in \f{H}^1(\mu_1, \mu_2) \right \}.
\end{equation}
As an immediate consequence of the above formula, 
many of the results of this work can also be stated in the case of 
a $q$-homogeneous cost function $\sfH_q$, provided that 
the unbalanced optimal transport functional is defined by using the $q$-homogeneous marginals $\f h^q$ 
and the corresponding set $\f H^q$
as in \eqref{eq:qhom}.
For this reason, in order to keep a simpler notation we will limit our analysis to the $1$-homogeneous case.
\end{remark}
 
\subsection{Examples of cost functions} In this subsection we present some examples of functions $\sfH: \pc \to [0, + \infty]$, where $\sfX_1$ and $\sfX_2$ are completely regular spaces, satisfying (some of) the hypotheses we will assume throughout the paper. See also the examples of \cite[Section5]{CPSV18}.

Notice that given two radially $1$-homogeneous, convex, lower semicontinuous and proper functions $\sfH_1, \sfH_2: \pc \to [0,+\infty)$, and two lower semicontinuous and proper functions $\sfc_1, \sfc_2 : \sfX_1 \times \sfX_2 \to [0,+\infty)$ also the function
\[ \sfH([x_1,r_1], [x_2,r_2]) = \sfc_1(x_1,x_2)\sfH_1([x_1,r_1], [x_2,r_2]) +\sfc_2(x_1,x_2)\sfH_2([x_1,r_1], [x_2,r_2]) \]
is radially $1$-homogeneous, convex, lower semicontinuous and proper. In this way, many other examples can be obtained starting from the ones presented below. Notice moreover 
that the radial behavior of a $1$-homogeneous convex function
$\sfH_{x_1,x_2}:\R_+\times \R_+\to [0,+\infty]$ can be characterized by the reduced function and its recession slope
\begin{equation}
    \label{eq:reduced}
    \mathsf h(x_1,x_2;g)=\mathsf h_{x_1,x_2}(g):=
    \sfH_{x_1,x_2}(1,g),\quad
    h^\infty_{x_1,x_2}:=\lim_{g\to+\infty}g^{-1}
    \mathsf h_{x_1,x_2}(g),
\end{equation}
since
\begin{equation}
    \label{eq:expand
    }
    \sfH_{x_1,x_2}(r_1,r_2)=
    \begin{cases}
        r_1\mathsf h_{x_1,x_2}(r_2/r_1)&\text{if }r_1>0,\\
        r_2 \mathsf h^\infty_{x_1,x_2}&\text{if }r_1=0,\ r_2>0,\\
        0&\text{if }r_1=r_2=0.
    \end{cases}
\end{equation}

\subsubsection{Mass-space product costs}
We consider cost functions of the form 
\[ \sfH([x_1, r_1], [x_2, r_2]):= \sfH_+(r_1,r_2)+\sfH_-(r_1,r_2)\sfc(x_1,x_2), \quad ([x_1, r_1], [x_2,r_2]) \in \pc,\]
where $\sfH_+, \sfH_-: \R^2_+ \to \R$ are convex, $1$-homogeneous and lower semicontinuous and $\sfc:\sfX_1 \times \sfX_2\to [0,+\infty]$ is a lower semicontinuous function satisfying
\[ \sfH_+(r_1,r_2) +\sfH_-(r_1,r_2) \sfc(x_1,x_2) \ge 0 \quad \text{ for every } (x_1,x_2) \in \sfX_1\times \sfX_2 , \, r_1, r_2 \ge 0.\]
Possible choices of $\sfH_+$ and $\sfH_-$ (to be then coupled with a suitable cost $\sfc$) are given by e.g.
\begin{enumerate}
    \item $\displaystyle m_p(r_1,r_2):=\left (r_1^{p} + r_2^{p} \right )^{\frac{1}{p}}, \quad p \in [1, +\infty)$,
    \item $\displaystyle m_p(r_1,r_2):=-m_{p}(r_1,r_2), \quad p \in (-\infty, 0) \cup (0,1)$,
    \item $\displaystyle m_\infty(r_1,r_2)=r_1 \vee r_2, \quad m_{-\infty}(r_1,r_2)=-(r_1 \wedge r_2), \quad m_0 =-\sqrt{r_1r_2}$,
    \item $\displaystyle n_\alpha(r_1,r_2):=\frac{|r_1-r_2|^{\alpha}}{(r_1+r_2)^{\alpha-1}}, \quad \alpha \ge 1$,
    \item $\displaystyle |r_1^{\alpha}-r_2^{\alpha}|^{1/\alpha}, \quad -|r_1^{\alpha}+r_2^{\alpha}|^{1/\alpha}, \quad 0<\alpha \le 1$.
\end{enumerate}
In the various examples we are adopting standard conventions when the expressions are not defined. In particular, for $p<0$ we set $m_p(r_1,r_2)=0$ if $r_1r_2=0$, $n_\alpha(0,0)=0$.

\subsubsection{Homogeneous marginal perspective functional}
Following \cite[Section 5]{LMS18} we can build $\sfH$ starting from two entropy functions $F_i : \sfX_i \to [0, +\infty]$, $i=1,2$ and a proper and lower semicontinuous cost function $\sfc: \sfX_1 \times \sfX_2 \to [0, + \infty]$. Assuming that each $F_i$, $i=1,2$ is convex, lower semicontinuous and finite in at least one positive point, we can define, for every number $c \in [0, + \infty]$, the function $\sfH_c: \R_+^2 \to [0, +\infty]$, as the lower semicontinuous envelope of 
\[ \tilde{\sfH}_c(r_1,r_2) := \begin{cases} \inf_{\theta >0} \left \{ r_1 F_1(\theta/r_1) + r_2 F_2(\theta/r_2) + \theta c \right \}, \quad &\text{ if }c \in [0, + \infty)\\ F_1(0)r_1+F_2(0)r_2\quad &\text{ if }c =\infty, \end{cases} \quad r_1, r_2 \in \R_+^2.\]
The function $\sfH: \pc \to [0, +\infty]$ is then defined as
\[ \sfH([x_1, r_1], [x_2, r_2]):= \sfH_{\sfc(x_1,x_2)}(r_1,r_2), \quad ([x_1, r_1], [x_2, r_2]) \in \pc.\]
Such function $\sfH$ is radially
convex and $1$-homogeneous (see \cite[Lemma 5.3]{LMS18}). Possible choices (see e.g.~\cite{LMS16, DP20}) for $F_i$ are given by: 
\begin{enumerate}
    \item power like entropies: for $p \in \R$ we define \[ U_p(s) := \begin{cases} \frac{1}{p(p-1)} \left (s^p - p(s-1)-1) \right ) \quad &\text{ if }p \ne 0,1,\\
    s\log s -s+1, \quad &\text{ if }p =1,\\
    s-1-\log s, \quad &\text{ if }p =0,\\
    \end{cases} \quad \text{ for } s>0,\] with $U_p(0)=1/p$ if $p>0$ and $U_p(0)=+\infty$ if $p\le0$;
    \item indicator functions: for numbers $0 \le a \le 1 \le b \le +\infty$ we define \[ I_{[a,b]}(s):= \begin{cases} 0 \quad &\text{ if } s \in [a,b],\\
    + \infty \quad &\text{ if } s \notin [a,b];\end{cases}\]
    \item $\nchi^\alpha$ divergences: for a parameter $\alpha \ge 1$ we define
    \[ \nchi^\alpha(s):= |s-1|^\alpha, \quad s \in \R.\]
    \end{enumerate}

Some of the corresponding expression for $\sfH$ are for example
\begin{enumerate}
    \item In case of power like entropies with $F_1=F_2=U_p$:
    \[ \sfH([x_1,r_1],[x_2,r_2])= \begin{dcases} \frac{1}{p} \left [ (r_1+r_2)- \frac{r_1r_2}{(r_1^{p-1}+r_2^{p-1})^{\frac{1}{p-1}}}(2-(p-1)\sfc(x_1,x_2))_+^{\frac{p}{p-1}} \right ], \quad &\text{ if } p \ne 0,1,\\
    (\sqrt{r_1}-\sqrt{r_2})^2+2\sqrt{r_1 r_2} (1-e^{-\sfc(x_1,x_2)/2}), \quad &\text{ if }p=1,\\
    r_1\log(r_1)+r_2\log(r_2)-(r_1+r_2)\log\left (\frac{r_1+r_2}{2+\sfc(x_1,x_2)}\right ), \quad &\text{ if } p=0.
    \end{dcases}\]
    In particular, in case $p=1$, for $\sfX_1=\sfX_2=\R^d$ we can chose as cost functions $\sfc$
    \[ \sfc_{\mathsf{GHK}}(x_1,x_2):= |x_1-x_2|^2, \quad \sfc_{\mathsf{HK}}(x_1,x_2):=\begin{cases} -\log(\cos^2(|x_1-x_2|)) \quad &\text{ if } |x_1-x_2| < \pi/2, \\ + \infty \quad & \text{ else}.\end{cases}\]
    The resulting cost functions $\sfH$ are thus given respectively by
    \begin{align} \label{eq:ghk}
        \sfH_{\mathsf {GHK}}([x_1,r_1],[x_2,r_2])&= r_1+r_2-2\sqrt{r_1r_2}e^{-|x_1-x_2|^2/2},\\ \label{eq:hk}
        \sfH_{\mathsf{HK}}([x_1,r_1],[x_2,r_2])&= r_1+r_2-2\sqrt{r_1r_2}\cos(|x_1-x_2|\wedge \pi/2).
    \end{align}
\eqref{eq:ghk} and 
\eqref{eq:hk} are metrics on $\f C[\R^d]$, inducing the same canonical cone topology.  
In particular, $\sfH_{\mathsf{HK}}$ is related to \eqref{ss22:eq:distcone} via the transformation in Remark \ref{rem:tp}
with $q=2$ (apart from the specific value of the truncation constant).
Both functions are considered in \cite{LMS18} and they generate the Gaussian Hellinger-Kantorovich and the Hellinger-Kantorovich metrics on non-negative measures, respectively. 
    \item In case of indicator functions with $F_1=F_2=I_{[a,b]}$:
    \[ \sfH([x_1,r_1],[x_2,r_2])= \begin{cases} 0 \quad &\text{ if } \frac{a}{b} \le \frac{r_1}{r_2} \le \frac{b}{a}, \\ + \infty \quad &\text{ else}, \end{cases}\]
    where $\frac{b}{a}=+\infty$ if $a=0$ and $\frac{a}{b}=0$ if $b=+\infty$.
    \item In case of the $\chi^1$ divergence with $F_1=F_2=\chi^1$:
    \[ \sfH([x_1,r_1],[x_2,r_2])= |r_2-r_1|+ (\sfc(x_1,x_2) \wedge 2)(r_1 \wedge r_2).\]
\end{enumerate}

\section{Convexification and duality}\label{sec:4} This section presents the main convexification and duality results. In this section $\sfX_1$ and $\sfX_2$ are completely regular spaces.

\subsection{Unbalanced Optimal Transport as convex relaxation of
  singular cost on Dirac masses}
\begin{definition}[Homogeneous conical formulation of Unbalanced Optimal Transport problems]
\label{ss22:def:uh} Let $\sfH: \pc \to [0, + \infty]$ be a proper (i.e.~not identically equal to $+\infty$) Borel function. We define the singular cost and the Unbalanced Optimal Transport cost $\cost{\sfH}, \m_\sfH: \meas(\sfX_1) \times \meas(\sfX_2) \to [0, + \infty]$ respectively as:
\begin{align*}
    \cost{\sfH}(\mu_1, \mu_2)&:= \begin{dcases} \sfH([x_1, r_1], [x_2, r_2]) \quad &\text{ if } \mu_1 = r_1 \delta_{x_1},\, \mu_2= r_2 \delta_{x_2}, \, (x_1, x_2) \in \sfX_1 \times \sfX_2, \, (r_1, r_2) \in \R_+^2, \\
    + \infty \quad &\text{ elsewhere}.\end{dcases}\\
    \m_{\sfH}(\mu_1, \mu_2)&:= \begin{dcases} \inf \left \{ \int_{\pc} \sfH \de \aalpha : \aalpha \in \f{H}^1(\mu_1, \mu_2) \right \} \quad &\text{ if } (\mu_1, \mu_2) \in \meas_+(\sfX_1) \times \meas_+(\sfX_2),\\
    + \infty \quad &\text{ elsewhere}.\end{dcases}
\end{align*}
\end{definition}
The aim of this section is to study the relation between $\cost{\sfH}$ and $\m_\sfH$; in particular we are interested in studying the lower semicontinuous and convex envelope of $\cost{\sfH}$. 
 
Let us first recall two simple properties in the following remarks.
\begin{remark}
    \label{rem:invariance2}
    The value of $\m_{\sfH}(\mu_1,\mu_2)$
    does not depend on the ambient
    spaces $\sfX_i$ (recall the discussion of Remark 
    \ref{rem:invariance1}):
    if $\sfX_i\subset \tilde\sfX_i$,
    $\sfH$ is the restriction
    to $\f C[\sfX_1,\sfX_2]$
    of a function $\tilde \sfH$
    defined in $\f C[\tilde\sfX_1,\tilde\sfX_2]$,
    and $\tilde \mu_i\in \meas_+(\tilde\sfX_i)$
    are the canonical extensions of $\mu_i$, then
    Remark \ref{rem:invariance1}
    yields
    \begin{equation}
        \label{eq:che-noia}
        \m_{\sfH}(\mu_1,\mu_2)=
        \m_{\tilde\sfH}(\tilde \mu_1,\tilde\mu_2).
    \end{equation}
    In particular, we can always
    ``embed'' Unbalanced Optimal Transport problems
    in $\sfX_{i,o}$ or, equivalently, suppose
    that there is
    at least one $\mu_i$-negligible 
    point $o_i$ in each space $\sfX_i$.
\end{remark}

\begin{remark} If $\sfH: \pc \to [0, + \infty]$ is a proper Borel function, then 
\[
    \overline{\cost{\sfH}} = \cost{\overline{ \sfH}}.
\]
Indeed, both are equal to $+ \infty$ outside the closed set $\D{\sfX_1} \times \D{\sfX_2}$ and the equality on $\D{\sfX_1} \times \D{\sfX_2}$ follows by Lemma \ref{ss22:lem:homeo}.
\end{remark}
For this reason and to exploit Lemma \ref{ss22:le:mineq}, we will usually assume that 
\begin{equation}\label{ss22:eq:assumption} \sfH: \pc \to [0, + \infty] \text{ is a proper, radially $1$-homogeneous and l.s.c.~function.}
\end{equation}

In the following result we prove that the Unbalanced Optimal Transport cost $\m_\sfH$ is a lower
semicontinuous convex functional in $\meas_+(\sfX_1)\times
\meas_+(\sfX_2)$
and the infimum in the Definition of $\m_\sfH$ is attained, when the
problem is feasible.
\begin{theorem}[Existence of solutions
to the Unbalanced Optimal Transport problem] \label{ss22:prop:mh} Let $\sfH$ be as in
  \eqref{ss22:eq:assumption} and let $\m_\sfH$ be as in Definition
  \ref{ss22:def:uh}.
  \begin{enumerate}
  \item For every
    $(\mu_1 ,\mu_2) \in \meas_+(\sfX_1) \times \meas_+(\sfX_2)$ such
    that $\m_\sfH(\mu_1, \mu_2) <+\infty$, there exists an optimal
    $1$-homogeneous coupling
    $\aalpha \in \f{H}^1(\mu_1, \mu_2) \cap \prob(\pcr{R(\mu_1,
      \mu_2)})$ 
     not charging $\bfo=(\fo_1,\fo_2)$ 
       such that
    \[ \m_\sfH(\mu_1, \mu_2) = \int_{\pc}\sfH \de \aalpha,\] where
    $R(\mu_1, \mu_2)$ is as in \eqref{eq:theradi}.
    \item $\m_\sfH$
      is a lower semicontinuous convex functional
      in $\meas_+(\sfX_1)\times
      \meas_+(\sfX_2)$
      and satisfies 
      \begin{equation}
        \label{ss22:eq:recovering} \m_\sfH(r_1
      \delta_{x_1}, r_2 \delta_{x_2}) \le \sfH([x_1,r_1], [x_2, r_2])
    \end{equation}
    for every $(x_1, x_2) \in \sfX_1 \times \sfX_2$ and every
    $(r_1, r_2) \in \R_+^2$.
    \item If, in addition, $\sfH$ is also radially convex,
      then \eqref{ss22:eq:recovering} is an equality:
            \[ \m_\sfH(r_1
      \delta_{x_1}, r_2 \delta_{x_2}) = \sfH([x_1,r_1], [x_2, r_2]).
    \]
  \end{enumerate}

\end{theorem}
\begin{proof}
  (1)
  Let $(\mu_1 ,\mu_2) \in \meas_+(\sfX_1) \times \meas_+(\sfX_2)$; by Lemma \ref{ss22:le:mineq}, it holds
  \begin{equation*}
    \m_\sfH(\mu_1, \mu_2) = \inf \left \{ \int_{\pc} \sfH \de \aalpha : \aalpha \in \f{H}^1(\mu_1, \mu_2) \cap \prob(\pcr{R(\mu_1, \mu_2)}) \right \}.
\end{equation*}
Thanks to Lemma \ref{ss22:lemma:generalmarginals}, we have that $\f{H}^1(\mu_1, \mu_2) \cap \prob(\pcr{R(\mu_1, \mu_2)})$ is compact and the lower semicontinuity of $\sfH$ gives that the functional
\begin{equation}\label{ss22:eq:lsc} \aalpha \mapsto \int_{\pc}\sfH \de \aalpha
\end{equation}
is lower semicontinuous. We can thus conclude that a minimizer exists
by the direct method in Calculus of Variations.
Possibly replacing $\aalpha$ by 
$\aalpha_{\fo}=\aalpha\mres \pco$ we obtain an optimal homogeneous coupling not charging $(\fo_1,\fo_2).$

\medskip\noindent
(2)
The convexity of $\m_{\sfH}$ follows by the convexity of the
constraints and the linearity of the objective function characterizing
$\m_\sfH$: 
if $(\mu_1^1, \mu_2^1), (\mu_2^1, \mu_2^2) \in \meas_+(\sfX_1) \times \meas_+(\sfX_2)$ and $t\in [0,1]$, we can take, thanks to point (1),  $\aalpha_1 \in \f{H}^1(\mu_1^1, \mu_2^1)$ and $\aalpha_2 \in \f{H}^1(\mu_1^2, \mu_2^2)$ such that
\[  \m_\sfH(\mu_1^1, \mu_2^1)
  = \int_{\pc} \sfH \de \aalpha_1, \quad
   \m_\sfH(\mu_1^2, \mu_2^2)
   = \int_{\pc} \sfH \de \aalpha_2.\]
It is then enough to observe that $\aalpha := (1-t) \aalpha_1 + t
\aalpha_2 \in \f{H}^1((1-t)\mu_1^1+ t \mu_1^2, (1-t)\mu_2^1+ t
\mu_2^2)$.

The lower semicontinuity of $\m_\sfH$ is a consequence of Lemma \ref{ss22:le:mineq}: if $\{(\mu_1^{\lambda}, \mu_2^{\lambda})\}_{\lambda \in \mathbb{L}} \subset \meas_+(\sfX_1) \times \meas_+(\sfX_2)$ is a net converging to $(\mu_1, \mu_2) \in \meas_+(\sfX_1) \times \meas_+(\sfX_2)$, we can set $R:=\sup_\lambda R(\mu_1^\lambda, \mu_2^\lambda)$ (see \eqref{eq:theradi}) and consider, for every $\lambda \in \mathbb{L}$, some $\aalpha_{\lambda} \in \f{H}^1(\mu_1^{\lambda}, \mu_2^{\lambda}) \cap \prob(\pcr{R})$ such that
\[ \m_\sfH(\mu_1^{\lambda}, \mu_2^{\lambda}) = \int_{\pc}\sfH \de \aalpha_\lambda.\]
We can thus use Lemma \ref{ss22:lemma:generalmarginals} to extract a
convergent subnet of $\net{\aalpha}$ with limit $\aalpha \in
\f{H}^1(\mu_1, \mu_2)$. Using again the lower semicontinuity of the
functional in \eqref{ss22:eq:lsc}, we can conclude that $\m_\sfH$ is
lower semicontinuous.

\eqref{ss22:eq:recovering} follows by the fact that 
\[ \aalpha = \delta_{([x_1, r_1],[x_2, r_2])}\]
is an element of $\f{H}^1(r_1\delta_{x_1}, r_2 \delta_{x_2})$.

\medskip\noindent (3)
Let us assume that $\sfH$ is radially convex and let
$\aalpha \in \f{H}^1(r_1 \delta_{x_1}, r_2 \delta_{x_2})$ be such that 
\[ \m_\sfH(r_1 \delta_{x_1}, r_2 \delta_{x_2}) = \int_{\pc} \sfH \de \aalpha;\]
we observe that $\aalpha$ is concentrated on 
\[ \big \{ \lambda_1 [x_1,1], \lambda_2 [x_2, 1] : \lambda_1,
  \lambda_2 \ge 0
  \big\} \]
with
\[ \int_{\pc} \sfr_1 \de \aalpha = r_1, \quad \int_{\pc} \sfr_2 \de \aalpha = r_2. \]
Hence, using Jensen's inequality and the convexity of $\sfH_{x_1, x_2}$, we have
\begin{align*}
    \m_\sfH(r_1 \delta_{x_1}, r_2 \delta_{x_2}) &= \int_{\pc} \sfH \de \aalpha \\
    &= \int_{\R_+^2} \sfH_{x_1,x_2} \de (\sfr_1, \sfr_2)_\sharp  \aalpha \\
    &\ge \sfH_{x_1,x_2} \left (\int_{\R_+^2} (\sfr_1,\sfr_2) \de \aalpha \right )\\
    &= \sfH_{x_1, x_2}(r_1, r_2) \\
    &= \sfH([x_1, r_1], [x_2, r_2]).\qedhere
\end{align*}
\end{proof}
\noindent
Thanks to Theorem \ref{ss22:prop:mh}, given $(\mu_1,\mu_2) \in \meas_+(\sfX_1) \times \meas_+(\sfX_2)$ such that $\m_\sfH(\mu_1, \mu_2) <+\infty$, the set 
\begin{equation}\label{ss22:eq:optplan}
\f{H}^1_o(\mu_1, \mu_2) := \left \{ \aalpha \in \f{H}^1(\mu_1, \mu_2) : \int_{\pc}\sfH \de \aalpha = \m_\sfH(\mu_1, \mu_2)\right \}
\end{equation}
is not empty.

In the following we give a direct proof that the Unbalanced Optimal Transport cost $\m_\sfH$ is the lower
semicontinuous convex envelope of the singular cost $\cost{\sfH}$. This result will also
be obtained in Theorem \ref{ss22:teo:duality} as a simple consequence of the dual characterization of $\m_\sfH$. However, the following proof highlights the role played by the discrete $1$-homogeneous marginals, which is not evident in the proof of Theorem \ref{ss22:teo:duality}.
\begin{theorem}[Convex l.s.c.~envelope]
\label{ss22:teo:representation} Let $\sfH$ be as in \eqref{ss22:eq:assumption} and let $\cost{\sfH}$ and $\m_\sfH$ be as in Definition \ref{ss22:def:uh}. Then
\[\cce{\cost{\sfH}}= \m_\sfH \text{ in } \meas(\sfX_1) \times \meas(\sfX_2). \]
\end{theorem}
\begin{proof} We only need to prove equality on $\meas_+(\sfX_1)
  \times \meas_+(\sfX_2)$, being both functions equal to $+ \infty$
  outside it. First of all let us compute $\ce{\cost{\sfH}}$ on
  $\ce{\D{\sfX_1} \times \D{\sfX_2}} = \ce{\D{\sfX_1}} \times
  \ce{\D{\sfX_2}}$ (outside this set $\ce{\cost{\sfH}}$ is equal to
  $+\infty$).
   We take $(\mu_1, \mu_2) \in  \ce{\D{\sfX_1}} \times
  \ce{\D{\sfX_2}}$
  and we observe that any finite set in
  $ \D{\sfX_1} \times \D{\sfX_2}$ is always contained in
  the cartesian product of two finite sets
  $M_1=\{\mu_1^i=r_1^i\delta_{x_1^i}:i=1,\cdots, I\}$
  and
  $M_2=\{\mu_2^j=r_1^j\delta_{x_2^j}:j=1,\cdots, J\}$,
   so that 
\begin{align*} \begin{split}
\ce{\cost{\sfH}}(\mu_1, \mu_2) ={}& \inf \Biggl \{ \sum_{ij} \gamma_{ij}
\cost{\sfH}(r^i_1\delta_{x_1^i},r_2^j\delta_{x_2^j}) :
\Big(r_1^i,x_1^i,r_2^j,x_2^j,\gamma_{ij}\Big)_{ij} \in A(\mu_1,
\mu_2)\Biggr\},\\
A(\mu_1,
\mu_2):={}& \Biggr\{\Big(r_1^i,x_1^i,r_2^j,x_2^j,\gamma_{ij}\Big)_{ij}:
(\mu_1,\mu_2)=\sum_{ij} \gamma_{ij} (r^i_1\delta_{x_1^i},r_2^j\delta_{x_2^j}) \, , \, \sum_{ij} \gamma_{ij} =1, 
\gamma_{ij} \ge 0 
\Biggr \}.
\end{split} \end{align*}
In particular it holds
that
\[
\mu_1 =
  \sum_{ij} \gamma_{ij} r_1^i
  \delta_{{x}_1^i}, \quad \mu_2 =
  \sum_{ij} \gamma_{ij} r_2^j \delta_{x_2^j}
\]
for some $r_1^i, r_2^j \in \R_+$ and some ${x}_1^i \in \sfX_1$,
${x}_2^j \in \sfX_2$ not necessarily distinct points.
Setting
\begin{equation}
  \label{eq:1}
  \aalpha:=\sum_{i,j}\gamma_{ij}\delta_{([x_1^i,r_1^i],[x_2^j,r_2^j])}\in
  \mathcal P_{\mathrm f}(\pc)
\end{equation}
one immediately checks that
\begin{equation}
  \label{eq:2}
  \aalpha\in \f H^1(\mu_1,\mu_2),\quad
  \sum_{ij} \gamma_{ij}
  \cost{\sfH}(r^i_1\delta_{x_1^i},r_2^j\delta_{x_2^j}) =
  \int \sfH\,\de\aalpha.
\end{equation}
Conversely, every $\aalpha\in   \mathcal P_{\mathrm f}(\pc)\cap \f
H^1(\mu_1,\mu_2) $ 
can be written as in \eqref{eq:1}
with coefficients $(r_1^i,x_1^i,r_2^j,x_2^j,\gamma_{ij})_{ij}$
in $A(\mu_1,\mu_2)$; the integral identity of \eqref{eq:2}
eventually shows that
$\ce{\cost{\sfH}}$ can be written as
\[\ce{\cost{\sfH}}(\mu_1, \mu_2) = \inf \left \{ \int_{\pc} \sfH \de \aalpha : \aalpha \in \f{H}^1(\mu_1, \mu_2) \cap 
    \prob_{\mathrm f}(\pc) \right \}. \]
Reasoning as in the proof of Lemma \ref{ss22:le:mineq}, we have that
\begin{align*}
 \ce{\cost{\sfH}}(\mu_1, \mu_2) &= \inf \left \{ \int_{\pc} \sfH \de
                                  \aalpha : \aalpha \in
                                  \f{H}^1(\mu_1, \mu_2) \cap
                                  \prob_{\mathrm f}(\f{C}_*) \right \},
\end{align*}
where $\f{C}_* := \pcr{R(\mu_1, \mu_2)}$, with $R(\mu_1, \mu_2)$ as in \eqref{eq:theradi}.
Thus
\[ \m_\sfH(\mu_1, \mu_2) \le \ce{\cost{\sfH}}(\mu_1, \mu_2) \quad \text{ for every } \, (\mu_1, \mu_2) \in \ce{\D{\sfX_1}} \times \ce{\D{\sfX_2}}.\]
Moreover $\m_\sfH$ is lower semicontinuous and convex hence, by definition of $\cce{\cost{\sfH}}$, it must hold
\[ \cce{\cost{\sfH}}(\mu_1, \mu_2) \ge \m_\sfH (\mu_1, \mu_2) \quad \text{ for every } \, (\mu_1, \mu_2) \in \meas_+(\sfX_1) \times \meas_+(\sfX_1).\]
Then, in order to prove equality, we only need to prove the other inequality. To do so, fixed $(\mu_1, \mu_2) \in \meas_+(\sfX_1) \times \meas_+(\sfX_2)$, we prove that there exists a net $\{(\mu_1^{\eta}, \mu_2^{\eta})\}_{\eta \in \mathbb{E}} \subset \ce{\D{\sfX_1}} \times \ce{\D{\sfX_2}}$ s.t. $\lim_{\eta}(\mu_1^{\eta}, \mu_2^{\eta})= (\mu_1, \mu_2)$ and a net 
\[\{\ggamma_{\eta}\}_{\eta \in \mathbb{E}} \subset \discr(\f{C}_*) \cap \prob(\f{C}_*)\]
s.t. $\ggamma_{\eta} \in \f{H}^1(\mu_1^{\eta}, \mu_2^{\eta})$ for every $\eta \in \mathbb{E}$ satisfying 
\[ \lim_{\eta \in \mathbb{E}} \int_{\pc} \sfH \de \ggamma_{\eta} = \int_{\pc} \sfH \de \aalpha^* \]
where $\aalpha^* \in \f{H}^1_o(\mu_1, \mu_2) \cap \prob(\f{C}_*)$ (see \eqref{ss22:eq:optplan}). To do so, we use Lemma \ref{ss22:lemma:reti} with $\sfX:=\f{C}_*$, $f:= \sfH$, $\alpha:=\aalpha^*$ and we find $\{\ggamma_{\eta}\}_{\eta \in \mathbb{E}} \subset \discr(\f{C}_*) \cap \prob(\f{C}_*)$ s.t.
\[ \lim_{\eta \in \mathbb{E}} \ggamma_{\eta} = \aalpha^*, \quad \lim_{\eta \in \mathbb{E}} \int_{\f{C}_*} \sfH \de \ggamma_{\eta} = \int_{\f{C}_*} \sfH \de \aalpha^*. \]
Finally we can define
\[ \mu_1^{\eta} := \f{h}_1^1(\ggamma_{\eta}), \quad \mu_2^{\eta} := \f{h}_2^1(\ggamma_{\eta}) \quad \text{ for every } \, \eta \in \mathbb{E}.\]
Obviously $\ggamma_{\eta} \in \f{H}^1(\mu_1^\eta, \mu_2^\eta)$ and $\mu_i = \lim_{\eta \in \mathbb{E}} \mu_i^{\eta}$, indeed if $\varphi_i \in \rmC_b(\sfX_i)$, then
\begin{align*}
\lim_{\eta \in \mathbb{E}} \int_{\sfX_i} \varphi_i \de \mu_i^{\eta} &= \lim_{\eta \in \mathbb{E}} \int_{\sfX_i} \varphi_i \de \f{h}^1_i(\ggamma_{\eta}) = \lim_{\eta \in \mathbb{E}} \int_{\pc} (\varphi \circ \sfx_i) \sfr_i \de \ggamma_{\eta} \\
& = \lim_{\eta \in \mathbb{E}} \int_{\f{C}_*} (\varphi_i \circ \sfx_i) \sfr_i \de \ggamma_{\eta} = \int_{\f{C}_*} (\varphi_i \circ \sfx_i) \sfr_i \de \aalpha^* \\
& = \int_{\pc} (\varphi_i \circ \sfx_i) \sfr_i \de \aalpha^* = \int_{\sfX_i} \varphi_i \de \f{h}^1_i(\aalpha^*) \\
& = \int_{\sfX_i} \varphi_i \de \mu_i,
\end{align*}
where we have used that $(\varphi \circ \sfx_i) \sfr_i \in \rmC_b(\f{C}_*)$ and the convergence of $\ggamma_{\eta}$ to $\aalpha^*$ in $\prob(\f{C}_*)$. Notice that, in general, it is not true that  $(\varphi \circ \sfx_i) \sfr_i \in \rmC_b(\pc)$.
Finally
\begin{align*}
\cce{\cost{\sfH}}(\mu_1, \mu_2) &=  \inf \biggl \{\liminf_{\lambda} \ce{\cost{\sfH}}(\mu_1^{\lambda}, \mu_2^{\lambda}) : \{ (\mu_1^{\lambda}, \mu_2^{\lambda}) \}_{\lambda \in \mathbb{L}} \subset \ce{\D{\sfX_1} \times \D{\sfX_2}} , \\
& \quad \quad \quad  (\mu_1, \mu_2) = \lim_{\lambda}(\mu_1^{\lambda}, \mu_2^{\lambda}) \biggr \} \\
& \le \liminf_{\eta} \ce{\cost{\sfH}}(\mu_1^{\eta}, \mu_2^{\eta}) \le \liminf_{\eta} \int_{\pc} \sfH \de \ggamma_{\eta} \\
&=  \int_{\pc} \sfH \de \aalpha^* = \m_\sfH(\mu_1,\mu_2). \qedhere
\end{align*}
\end{proof}
Theorem \ref{ss22:teo:representation}
immediately yields the following 
useful property.

\begin{theorem}[Sublinearity of $\m_{\sfH}$]
\label{th:sublinearity}
The functional $\m_{\sfH}$ is sublinear
(i.e.~convex and positively $1$-homogeneous) in
$\meas_+(\sfX_1)\times \meas_+(\sfX_2)$:
for every $\mu_i',\mu_i''\in \meas_+(\sfX_i)$
and every $\lambda',\lambda''\in [0,+\infty)$
\begin{equation}
    \label{eq:sublinearity}
    \m_\sfH(\lambda'\mu_1'+\lambda''\mu_1'',\lambda'\mu_2'+\lambda''\mu_2'')
    \le \lambda'\m_\sfH(\mu_1',\mu_2')+
    \lambda''\m_{\sfH}(\mu_1'',\mu_2'').
\end{equation}
\end{theorem}

\subsection{Disintegration, barycentric projection and decomposition of homogeneous couplings}
Every plan $\aalpha\in \meas_+(\f C[\sfX_1,\sfX_2])$
can be expressed as a superposition of 
a Borel family of probability measures 
$(\alpha_{\boldsymbol x})_{\boldsymbol x}$ in $\R^2_+$,
$\boldsymbol x=(x_1,x_2)\in \sfX_{1,o}\times \sfX_{2,o}.$
It is sufficient to consider
the map $\boldsymbol \sfx:=(\sfx_1,\sfx_2)
:\pc\to \sfX_{1,o}\times \sfX_{2,o}$, 
where $\sfx_i:\f C[\sfX_1,\sfX_2]\to\sfX_{i,o}$ are defined in Section \ref{sec:theconee}, 
and the plan
$\ggamma:=(\sfx_1,\sfx_2)_\sharp \aalpha\in 
\meas_+(\sfX_{1,o}\times \sfX_{2,o})$.
The disintegration of $\aalpha$ w.r.t.~$\gamma$ yields the Borel family 
$(\alpha_{\boldsymbol x})_{\boldsymbol x} \subset \mathcal P(\R^2_+)$ satisfying 
\begin{equation}
    \label{eq:disintegration}
    \aalpha=\int_{\sfX_{1,o}\times \sfX_{2,o}} \alpha_{x_1,x_2}\,\de \ggamma(x_1,x_2),\quad
\ggamma=\boldsymbol \sfx_\sharp \aalpha.
\end{equation} 
Notice that if 
$\aalpha(\{\bfo\})=0$ then 
$\ggamma$ is concentrated on $\sfX_{1,o}\times 
\sfX_{2,o}\setminus\{\boldsymbol o\}.$ 
If moreover $\aalpha\in \meas_+^1(\pc)$ then
we can define 
\begin{equation}
\label{eq:baricenters}
    \boldsymbol\varrho:=(\varrho_1,\varrho_2),\quad
    \varrho_i(x_1,x_2):=
    \int_{\R^2_+}r_i\,\de\alpha_{x_1,x_2}(r_1,r_2),\quad
    \varrho_i\in L^1(\sfX_{1,o}\times \sfX_{2,o},\ggamma), \, \, \varrho_i \ge 0.
\end{equation}
\begin{definition}[Barycentric projection and reduced couplings]
    The barycentric projection
    of a nontrivial plan $\aalpha\in \meas_+^1(\pc)$
    is the plan $\aalpha_b\in \meas_+^1(\pc)$
    defined by
    \begin{equation}
    \label{eq:b-representation}
    \begin{aligned}
    \aalpha_b:={}\mathsf T[\boldsymbol\varrho]_\sharp\ggamma,\quad
    \text{where}\quad
    \mathsf T[\boldsymbol\varrho](\boldsymbol x):={}
    \big([x_1,\varrho_1(\boldsymbol x)],[x_2,\varrho_2(\boldsymbol x)]\big),\quad
    \boldsymbol x\in \sfX_{1,o}\times \sfX_{2,o}.
    \end{aligned}
\end{equation}
When $\aalpha(\pc)=0$ (i.e.~$\aalpha$ is the null measure) we set $\aalpha_b:=\aalpha$.\\
We say that $\aalpha$ is a reduced plan
if $\aalpha=\aalpha_b$ or, equivalently, if $\alpha_{\boldsymbol x}$ is a Dirac mass
$\delta_{\boldsymbol\varrho(\boldsymbol x)}$
for $\ggamma$-a.e.~$\boldsymbol x\in \sfX_{1,o}\times \sfX_{2,o}.$
\end{definition}
The interest in the barycentric projection
and in reduced couplings
is justified by the following result.
\begin{proposition}[Reduced homogeneous couplings and minimizers]
    \label{prop:bari}
    If $\mu_i \in \meas_+(\sfX_i)$, $\aalpha\in \f H^1(\mu_1,\mu_2)$
    then also $\aalpha_b\in \f H^1(\mu_1,\mu_2).$ If $\sfH$ satisfies \eqref{ss22:eq:assumption}
    and it is radially convex then
    \begin{equation}
        \label{eq:bar-is-better}
        \int_{\pc} \sfH\,\de\aalpha_b
        \le \int_{\pc} \sfH\,\de\aalpha.
    \end{equation}
    In particular, if $\m_\sfH(\mu_1,\mu_2)<+\infty$, then the minimum of the Unbalanced Optimal Transport problem 
    \ref{ss22:def:uh} is attained at a reduced
    coupling.
\end{proposition}
\begin{proof}
    We can assume that $\aalpha$ is nontrivial.
    Let us first check that 
    $\aalpha_b$ preserves homogeneous marginals.
    For every $\zeta_i\in \mathrm C_b(\sfX_i)$, we have
\begin{align*}
    \int (\zeta_i \circ \sfx_i) \sfr_i\,\de\aalpha_b&=
    \int \zeta_i(x_i)\varrho_i(x_1,x_2)\,\de\ggamma(x_1,x_2)=
    \int \zeta_i(x_i)\Big(\int r_i\,\de\alpha_{x_1,x_2}(r_1,r_2)\Big)\,\de\ggamma(x_1,x_2)
    \\&=
    \int (\zeta_i \circ \sfx_i) \sfr_i\,\de\aalpha=
    \int \zeta_i\,\de\mu_i.    
\end{align*}
By radial convexity, using Jensen's inequality, we have
\begin{align*}
\int_{\pc} \sfH \de \aalpha &= \int_{\sfX_{1,o} \times \sfX_{2,o}} 
\Big(\int_{\R_+^2} \sfH \de \alpha_{x_1,x_2}(r_1,r_2) \Big)\de \gamma(x_1,x_2) \\
& \ge \int_{\sfX_{1,o} \times \sfX_{2,o}} \sfH ([x_1, \varrho_1(x_1,x_2)], [x_2,\varrho_2(x_1,x_2)]) \de \gamma(x_1,x_2) 
=
\int_{\pc} \sfH \de \aalpha_b.\qedhere
\end{align*}
\end{proof}
We can now show that 
it is possible to improve the previous representation by decomposing
the measures $\mu_i$ in two parts, corresponding to complete or partial distruction/creation of mass. 
We introduce the 
partition $\f C',\f C''_{1},\f C''_{2}$ 
of $\pco$, where
\begin{displaymath}
    \f C':=(\f C[\sfX_1]\setminus \{\fo_1\})
    \times (\f C[\sfX_2]\setminus \{\fo_2\}),\quad
    \f C_{1}'':=
    (\f C[\sfX_1]\setminus \{\fo_1\})
    \times \{\fo_2\},\quad
    \f C_{2}'':=
    \{\fo_1\}\times (\f C[\sfX_2]\setminus \{\fo_2\}).
\end{displaymath}
Every homogeneous plan $\aalpha\in \meas_+^1(\pc)$
not charging $\bfo=(\fo_1,\fo_2)$
can be decomposed into the corresponding sum 
$\aalpha=\aalpha'+\aalpha''_{1}+\aalpha''_{2}$,
where $\aalpha'=\aalpha\mres\f C'$
and $\aalpha_{i}'':=\aalpha\mres\f C_{i}''.$
Correspondingly we can write
$\ggamma=\ggamma'+\ggamma_{1}''+\ggamma_{2}''$
\begin{equation}
\label{eq:gamma-decomposition}
    \ggamma':=\boldsymbol\sfx_\sharp\aalpha'
    \in \meas_+(\sfX_1\times \sfX_2),\quad
    \ggamma_{1}''=
    \boldsymbol\sfx_\sharp\aalpha_{1}''
    \in \meas_+( \sfX_1
    \times \{o_2\}),\quad
    \ggamma_{2}''=
    \boldsymbol\sfx_\sharp\aalpha_{2}''
    \in \meas_+(\{o_1\}\times \sfX_2).
\end{equation}
\begin{theorem}[Distinguished optimal couplings]
    \label{theo:decomposition}
    Let us suppose that $\sfH$ 
    satisfies \eqref{ss22:eq:assumption}
    and it is radially convex and let 
    $\mu_i\in \meas_+(\sfX_i)$ non trivial
    such that $\m_\sfH(\mu_1,\mu_1)<+\infty.$
    
    There exist Borel partitions $\{S_i',S_i''\}$
    of $\sfX_i$ and an optimal homogeneous and reduced  
    coupling
    $\aalpha\in \f H^1(\mu_1,\mu_2)
    $ not charging $\bfo$ such that 
    setting 
    \begin{equation}
        \label{eq:prelim}
        \mu_i':=\mu_i\mres S_i',\quad
        \mu_i(S_i')>0,\quad
         \mu_i'':=\mu_i\mres S_i'',    
    \end{equation}
    and keeping the above notation
    for $\varrho_i,\aalpha',\aalpha_{i}'',
    \ggamma,\ggamma',\ggamma_{i}''$ we have
    \begin{align}
    \label{eq:singular-alpha}
        \aalpha_{1}''&=
        \f p_\sharp(\mu_1''\otimes \delta_1)\otimes \delta_{\fo_2},&
        \aalpha_{2}''&=
        \delta_{\fo_1}\times \f p_\sharp(\mu_2''\otimes \delta_1),\\
    \label{eq:singular-gamma1}
    \ggamma_{1}''&=\mu_1''\otimes \delta_{o_2},&
        \ggamma_2''&=
        \delta_{o_1}\otimes \mu_2'',
    \end{align}
    \begin{equation}
    \label{eq:check1}
        \f h^1_i(\aalpha')=\mu_i',\quad
       \ggamma'=\ggamma\mres(S_1'\times S_2'),
        \quad
    \end{equation}
    \begin{equation}
    \begin{aligned}
    \aalpha={}\mathsf T[\boldsymbol\varrho]_\sharp\ggamma,\quad
    \aalpha'={}\mathsf T[\boldsymbol\varrho]_\sharp\ggamma',\quad
    \mathsf T[\boldsymbol\varrho](x_1,x_2):={}
    \big([x_1,\varrho_1(x_1,x_2)],[x_2,\varrho_2(x_1,x_2)]\big)
    \end{aligned}
\end{equation}
and
\begin{equation}
    \varrho_i>0\text{ on $S_1'\times S_2'$},
    \quad
    \varrho_1=0\text{ on $S_1''\times \sfX_2$},\quad
    \varrho_2=0\text{ on $\sfX_1\times S_2''$},
    \end{equation}
    \begin{equation}
    \varrho_1(x_1,o_2)=1 \text{ on $S_1''$},
    \quad
    \varrho_2(o_1,x_2)=1\text{ on $S_2''$}\quad
    \text{$\ggamma$-a.e.}.
\end{equation}
\end{theorem}
\begin{proof}
    Let $\tilde\aalpha=
    \tilde\aalpha'+\tilde\aalpha_{1}''+
    \tilde\aalpha_{2}''$ be an optimal reduced coupling
    in $\f H^1(\mu_1,\mu_2)$ 
    not charging $\bfo$.

    In the degenerate case when $\tilde\aalpha'$
    vanishes, then
    setting $\mu_i=m_i\nu_i$, 
    $\nu_i\in \prob(\sfX_i),$ $m_i=\mu_i(\sfX_i)>0,$
    we have
    \begin{displaymath}
        \m_\sfH(\mu_1,\mu_2)=
        \int_{\sfX_1}\sfH([x_1,m_1],\fo_2)
        \,\de\nu_1+
        \int_{\sfX_2}\sfH(\fo_1,[x_2,m_2])\,\de\nu_2\ge
        \int_{\sfX_1\times \sfX_2}\sfH([x_1,m_1],[x_2,m_2])\,\de(\nu_1\otimes \nu_2)
    \end{displaymath}
    so that the reduced coupling
    $\aalpha=([x_1,m_1],[x_2,m_2])_\sharp(\nu_1\otimes \nu_2)$
    satisfies all the assumptions
    with $S_i'=\sfX_i.$
    
    We can therefore suppose that $\aalpha'$ 
    is nontrivial and we set
    \begin{displaymath}
        \tilde\mu_i':=\f h^1_i(\tilde\aalpha'),
        \quad
        \tilde\mu_i'':=\f h^1_i(\tilde\aalpha_i''),\quad
        \tilde\ggamma:=\boldsymbol\sfx_\sharp\tilde\aalpha,\quad
        \tilde\ggamma':=\boldsymbol\sfx_\sharp
        \tilde\aalpha':=
        \tilde\ggamma\mres(\sfX_1\times \sfX_2),
    \end{displaymath}
    and we can write $\tilde\aalpha=
    \mathsf T[\tilde{\boldsymbol\varrho}]_\sharp\tilde\ggamma$
    as in \eqref{eq:b-representation}.
    
    The optimal cost $\m_\sfH(\mu_1,\mu_2)$ 
    can be decomposed as 
    \begin{equation}
        \label{eq:reduced-cost}
        \int \sfH\,\de\tilde\aalpha=
        \int_{\sfX_1\times \sfX_2}
        \sfH([x_1,\tilde\varrho_1],[x_2,\tilde\varrho_2])\,\de\tilde\ggamma'
        +
        \int_{\sfX_1} \sfH([x_1,1],\fo_2)\,\de\mu_1''+
        \int_{\sfX_2} \sfH(\fo_2,[x_2,1])\,\de\mu_2''.
    \end{equation}

    Since 
    $\f h^1_1(\tilde \aalpha_2'')=\boldsymbol 0_{\sfX_1}$ and 
    $\f h^1_2(\tilde \aalpha_1'')=\boldsymbol 0_{\sfX_2}$
    we have 
    $\mu_i=\tilde\mu_i'+\tilde\mu_i''$:
    we call $\theta_i',\theta_i''$
    Borel representatives of the Lebesgue densities 
    of $\mu_i',\mu_i''$ w.r.t.~$\mu_i$
    and we call
    \begin{displaymath}
        S_i'':=\Big\{x_i\in \sfX_i:
        \theta_{i}''(x_i)=1\Big\},\quad
        S_i':=\sfX_i\setminus S_i''.
    \end{displaymath}
    Since $\theta_i'+\theta_{i}''=1$
    $\mu_i$-a.e.~in $\sfX_i$, we clearly have
    $\theta_i'=0$ $\mu_i$-a.e.~in $S_i''$
    and 
    $\theta_i'>0$ $\mu_i$-a.e.~in $S_i'$.
    We also have $\mu_1'(S_1'')=\f h^1_1(\tilde\aalpha')(S_1'')=0$ 
    and since $\f h^1_1(\tilde\aalpha')=(\sfx_1)_\sharp(\tilde\varrho_1
    \tilde\ggamma')$
    we deduce that $\tilde\varrho_1=0$
    $\tilde\ggamma'$-a.e.~in $S_1''\times \sfX_2.$
    With a similar argument, we deduce that 
    $\tilde\varrho_2=0$
    $\tilde\ggamma'$-a.e.~in $\sfX_1\times S_2''.$
    
    We define $\mu_i',\mu_i''$ according
    to \eqref{eq:prelim}. 
    Notice that $\mu_i'=(\theta_i')^{-1}\tilde\mu_i'\ll \tilde\mu_i'$,
    whereas $\mu_i''=\tilde\mu_i''\mres S_i''.$
    The quotients $q_i(x_i):=\theta_i''/\theta_i'$
    are well defined $\mu_i$-a.e.~in $S_i'$
    and we have
    \begin{displaymath}
        \mu_i'=(1+q_i)\tilde\mu_i',\quad
        \tilde\mu_i''=
        \mu_i''+\theta_i''\mu_i'=
        \mu_i''+q_i\tilde\mu_i'.
    \end{displaymath}
    We set
    \begin{displaymath}
        \varrho_1:=
        \begin{cases}
            \tilde\varrho_1(1+q_1)
            &\text{in }S_1'\times \sfX_2,\\
            0&\text{in }S_1''\times \sfX_2,
        \end{cases}
        \qquad
        \varrho_2:=
        \begin{cases}
            \tilde\varrho_2(1+q_2)
            &\text{in }\sfX_1\times S_2',\\
            0&\text{in }\sfX_1\times S_2'',
        \end{cases}
    \end{displaymath}
    and we define $\aalpha_i''$ as in 
    \eqref{eq:singular-alpha}. Finally
    \begin{displaymath}
        \ggamma':=\tilde\ggamma'\mres(S_1'\times S_2'),\quad
        \aalpha':=\mathsf T[\boldsymbol\varrho]_\sharp\tilde\ggamma'=
    \mathsf T[\boldsymbol\varrho]_\sharp\ggamma'      ,
        \quad
        \aalpha:=\aalpha'+\aalpha_{1}''+
        \aalpha_{2}''.        
    \end{displaymath}
    Let us check \eqref{eq:check1}; since $\ggamma=\boldsymbol\sfx_\sharp\aalpha$, the equality $\ggamma'=\ggamma\mres(S_1\times S_2')$ is immediate, so that we have to prove that 
    $\f h^1_i(\aalpha')=\mu_i'$. 
    We just consider the case $i=1$ since the
    calculations in the case $i=2$ are completely analogous. For every test function $\zeta\in \rmC_b(\sfX_1)$ we get
    \begin{align*}
    \int (\zeta \circ \sfx_1) \sfr_1\,\de\aalpha'&=
    \int_{\sfX_1\times \sfX_2} \zeta(x_1)\varrho_1(x_1,x_2)\,\de\tilde\ggamma'(x_1,x_2)
    \\&=
    \int_{S_1'\times \sfX_2} \zeta(x_1)\tilde\varrho_1(x_1,x_2)\Big(1+
    q_1(x_1)\Big)\,\de\tilde\ggamma'(x_1,x_2)\\
    &=
    \int_{S_1'}\zeta(x_1)(1+q_1(x_1))\,\de
    \tilde\mu_1(x_1)
    =\int_{\sfX_1}\zeta\,\de
    \mu_1'.
\end{align*}
Let us now compute the $\sfH$-cost of $\aalpha$:
\begin{align*}
    \int\sfH \,\de\aalpha'&=
    \int 
    \sfH([x_1,\tilde\varrho_1+q_1\tilde\varrho_1],
    [x_2,\tilde\varrho_2+q_2\tilde\varrho_2])\,
    \de\tilde\ggamma'
    \\&\le 
    \int 
    \sfH([x_1,\tilde\varrho_1],
    [x_2,\tilde\varrho_2])\,
    \de\tilde\ggamma'
    +
    \int 
    q_1\tilde\varrho_1\sfH([x_1,1],
    \fo_2)\,
    \de\tilde\ggamma'+
    \int 
    q_2\tilde\varrho_2\sfH(\fo_1,
    [x_2,1])\,
    \de\tilde\ggamma'
    \\&=
    \int\sfH\,\de\tilde\aalpha'+
    \int 
    q_1\sfH([x_1,1],
    \fo_2)\,
    \de\tilde\mu_1'+
    \int 
    q_2\sfH(\fo_1,
    [x_2,1])\,
    \de\tilde\mu_2',\\
    \int\sfH \,\de\aalpha&=
    \int\sfH \,\de\aalpha'+
    \int\sfH \,\de\aalpha_1''+
    \int\sfH \,\de\aalpha_2''\\
    &=
    \int\sfH\,\de\aalpha'+
    \int\sfH([x_1,1],\fo_2) \,\de\mu_1''+
    \int\sfH(\fo_1,[x_2,1]) \,\de\mu_2''
    \\&\le 
    \int\sfH\,\de\tilde\aalpha'+
    \int 
    \sfH([x_1,1],
    \fo_2)\,
    \de(q_1\tilde\mu_1'+\mu_1'')
    \int 
    \sfH(\fo_1,
    [x_2,1])\,
    \de(q_2\tilde\mu_2'+\mu_2'')
    \\&=
    \int\sfH\,\de\tilde\aalpha'+
    \int 
    \sfH([x_1,1],
    \fo_2)\,
    \de \tilde\mu_1''+
    \int 
    \sfH(\fo_1,
    [x_2,1])\,
    \de \tilde\mu_2''=
    \int\sfH\,\de\tilde\aalpha=\m_\sfH(\mu_1,\mu_2).
\end{align*}
\end{proof}
As an application of the previous properties, we compare the 
formulation \ref{ss22:def:uh} of 
Unbalanced Optimal Transport 
via homogeneous couplings with
the formulation 
 based on the notion of semi-couplings
introduced by \cite{CPSV18} in a compact Euclidean setting. 
We obtain a considerable extension 
of the semi-coupling approach to 
a general, possibly non-compact, setting.

Given $\mu_i\in \meas_+(\sfX_i)$, $i=1,2$, we can consider the set
\begin{equation}
    \label{eq:semi-couplings}
    \hat \Gamma(\mu_1,\mu_2):=
    \Big\{(\gamma_1,\gamma_2)\in 
    (\meas_+(\sfX_1\times \sfX_2))^2:
    \pi^i_\sharp(\gamma_i)=\mu_i,\ i=1,2\Big\}.
\end{equation}
Every radially $1$-homogeneous cost function $\sfH$ as 
in \eqref{ss22:eq:assumption}
induces a functional on $\hat\Gamma(\mu_1,\mu_2)$ given by
\begin{equation}
    \label{eq:CSPV-functional}
    J_\sfH(\gamma_1,\gamma_2):=
    \int_{\sfX_1\times \sfX_2}
\sfH \Big(\Big[x_1,\frac{\de\gamma_1}{\de\gamma}\Big],
    \Big[x_2,\frac{\de\gamma_2}{\de\gamma}\Big]\Big)\,\de\gamma,\quad
    \gamma\in \meas_+(\sfX_1\times \sfX_2),\ 
    \gamma_i\ll\gamma.
\end{equation}
Since $\sfH$ is radially $1$-homogeneous,
\eqref{eq:CSPV-functional}
does not depend on the choice of the dominating measure $\gamma$ and one can consider the problem of minimizing $J_\sfH$ in $\hat\Gamma(\mu_1,\mu_2):$
\begin{equation}
    \label{eq:semi}
    \widehat{\mathscr U}_{\sfH}(\mu_1,\mu_2):=\inf\Big\{J_\sfH(\gamma_1,\gamma_2):
    (\gamma_1,\gamma_2)\in \hat\Gamma(\mu_1,\mu_2)\Big\}.    
\end{equation}
When $\widehat{\mathscr{U}_\sfH}(\mu_1,\mu_2)<\infty$, $\sfH$ satisfies \eqref{ss22:eq:assumption} and it is also radially convex,
and $\sfX_1,\sfX_2$ are compact, 
it is not difficult to prove the
existence of an optimal pair $(\gamma_1,\gamma_2)\in \hat\Gamma(\mu_1,\mu_2)$ attaining the infimum in \eqref{eq:semi} 
(see \cite[Prop.~3.4]{CPSV18}).
The direct approach when $\sfX_i$ are not compact is less clear, but existence of an optimal pair can be obtained as a consequence of the argument we are going to detail below in Theorem
\ref{theo:comparison}.
\begin{remark}
    The formulation \eqref{eq:semi} also allows for more general
    (l.s.c., radially $1$-homogeneous and convex) functions $\hat\sfH$
defined in 
$(\sfX_1\times [0,+\infty))\times (\sfX_2\times [0,+\infty))$ which are not compatible with the equivalence relation
\eqref{eq:cone-relation} inducing the 
topological cone $\f C[\sfX_1,\sfX_2]$.
This definition, however, 
may not satisfy the invariance property stated in Remark \ref{rem:invariance2}: consider, e.g., the case when
$\sfX_1=\sfX_2=[0,a]$, $a>0$,
$\mu_1=\delta_0,$ $\mu_2=\boldsymbol 0$,
and 
$\hat\sfH((x_1,r_1),(x_2,0))=r_1\mathrm e^{-|x_1-x_2|}$. 
We clearly have 
$\widehat{\m_{\hat \sfH}}(\mu_1,\mu_2)=\mathrm e^{-a}$.
In order to have an intrinsic formulation,
we will assume that $\sfH$ is compatible with the cone structure.
\end{remark}
\begin{theorem}
\label{theo:comparison}    If 
    $\sfH$ satisfies 
    \eqref{ss22:eq:assumption}
    then for every $\mu_i\in \meas_+(\sfX_i)$ non trivial
    we have 
    \begin{equation}
    \label{eq:equivalence1}
    \m_{\sfH}(\mu_1,\mu_2)\le \widehat{\mathscr{U}}_{\sfH}(\mu_1,\mu_2).
\end{equation}
If, in addition to \eqref{ss22:eq:assumption}, $\sfH$ is also radially convex,
then 
\begin{equation}
    \label{eq:equivalence2}
    \m_{\sfH}(\mu_1,\mu_2)= \widehat{\mathscr{U}}_{\sfH}(\mu_1,\mu_2).    
\end{equation}
and the infimum in \eqref{eq:semi} is attained.
\end{theorem}
\begin{proof}
  We first observe that 
  every triple 
$(\gamma,\gamma_1,\gamma_2)\in 
\meas_+(\sfX_1\times \sfX_2)\times
\hat\Gamma(\mu_1,\mu_2)$ with
$\gamma_i\ll\gamma$ induces a
homogeneous coupling $\aalpha\in \f H^1(\mu_1,\mu_2)$ via the formula
\begin{equation}
\label{eq:semi-to-hom}
    \aalpha:=G_\sharp \gamma,\quad G(x_1,x_2):=\Big(\big[x_1,\varrho_1(x_1,x_2)\big],
    \big[x_2,\varrho_2(x_1,x_2)\big]\Big),
    \quad
    \varrho_i:=\frac{\de\gamma_i}{\de\gamma},
\end{equation}
such that
\begin{equation}
    \label{eq:cost}
    \int_{\f C[\sfX_1,\sfX_2]}\sfH\,\de\aalpha=
    J_\sfH(\gamma_1,\gamma_2).
\end{equation}
Therefore we immediately get
\eqref{eq:equivalence1}.

Let us now assume that $\sfH$ 
is also radially convex and let $\aalpha$
be a distinguished optimal coupling as in Theorem
\ref{theo:decomposition}.
Keeping the same notation of that Theorem
let us set $\mu_i(S_i')=m_i>0$
and let us set $\nu_i:=m_i^{-1}\mu_i'.$
We define
$\gamma:=\ggamma'+\mu_1''\otimes \nu_2+
\nu_1\otimes \mu_2''$,
$\gamma_1:=\varrho_1\gamma'+\mu_1''\otimes \nu_2$,
$\gamma_2:=\varrho_2\gamma'+\nu_1\otimes \mu_2''$,
obtaining
$\pi^i_\sharp\gamma_i=
\mu_i'+\mu_i''=\mu_i$ so that 
$(\gamma_1,\gamma_2)\in \hat\Gamma(\mu_1,\mu_2).$
Moreover
\begin{align*}
    J_\sfH(\gamma_1,\gamma_2)&=
    \int \sfH([x_1,\varrho_1],[x_2,\varrho_2])\,\de\ggamma'
    \\&\qquad+
    \int\sfH([x_1,1],[x_2,0])\,\de(\mu_1''\otimes\nu_2)+
    \int\sfH([x_1,0],[x_2,1])\,\de(\nu_1\otimes\mu_2'')
    \\&=\int\sfH\,\de\aalpha=
    \m_\sfH(\mu_1,\mu_2).\qedhere
\end{align*}
\end{proof}

\noindent
\subsection{The dual problem}
In the next results we study the dual definition of the Unbalanced Optimal Transport cost $\m_\sfH$, for which we need the following definitions.

\begin{definition}\label{def:thesetadm} Let $\sfH:\pc \to [0,+\infty]$ be a function; we define the set of continuous functions
\begin{equation}\label{eq:constr}
  \Phi_\sfH := \left\{ (\varphi_1, \varphi_2) \in \rmC_b(\sfX_1) \times \rmC_b(\sfX_2)\ :  \begin{array}{l}
    \varphi_1(x_1)r_1+ \varphi_2(x_2)r_2 \le \sfH([x_1, r_1], [x_2, r_2]) \\
    \text{for every } (x_1, x_2) \in \sfX_1 \times \sfX_2, \, (r_1, r_2) \in \R_+^2
  \end{array}\right\}
  \end{equation}
  and, for every $(\mu_1, \mu_2) \in \meas_+(\sfX_1) \times \meas_+(\sfX_2)$, the functional $\dual(\cdot; \mu_1, \mu_2): \rmC_b(\sfX_1) \times \rmC_b(\sfX_2) \to \R$ given by
\[
      \dual(\varphi_1, \varphi_2; \mu_1, \mu_2) := \int_{\sfX_1} \varphi_1 \de \mu_1 + \int_{\sfX_2}\varphi_2 \de \mu_2, \quad (\varphi_1, \varphi_2) \in \rmC_b(\sfX_1) \times \rmC_b(\sfX_2).
\]
\end{definition}

Before stating the main duality result, let
us briefly recall
the Fenchel-Moreau Theorem in the
framework of a pair of vector spaces $E,F$ placed in duality
by a nondegenerate bilinear map $\langle\cdot,\cdot,\rangle$,
see e.g.~\cite{Temam}.
We endow $E$ with the weak topology
$\sigma(E,F)$, the coarsest topology for which all
the functions $e\mapsto \langle e,f\rangle$, $f\in F$,
are continuous.

\begin{definition}
  Let $\mathscr F: E \to (-\infty, + \infty]$
  be not identically $+\infty$ and satisfying
  \begin{equation}\label{ss22:eq:5}
    \mathscr F(e)\ge \langle e,f\rangle-c\quad
    \text{for some $f\in F$, $c\in \R$ and every }e\in E.
  \end{equation}
  The polar (or conjugate) function of
  $\mathscr F$ is
  the function $\mathscr F^*: F \to (-\infty, + \infty]$
  defined by
  \begin{equation*}
  \mathscr F^*(f) := \sup_{e \in E}  \scalprod{e}{f} -
  \mathscr F(e) \quad
  \text{for every }f\in F.
\end{equation*}
\end{definition}

\begin{theorem}[Fenchel-Moreau] \label{ss22:teo:fenchel}
  Let $E$ and $F$ be vector spaces placed in duality by a nondegenerate bilinear map $\langle\cdot,\cdot,\rangle$ and let $\mathscr F: E \to (-\infty, + \infty]$
  be satisfying \eqref{ss22:eq:5} and not identically $+\infty$.
  Then the lower semicontinuous (w.r.t.~the topology $\sigma(E,F)$) and convex envelope of $\mathscr F$
  is given by the dual formula
  \begin{equation*}
  \cce{\mathscr F} = \mathscr F^{**}(e):=\sup_{f\in F}\langle e,f\rangle-
  \mathscr F^*(f)\quad\text{for every }e\in E.
\end{equation*}
In particular,
\begin{equation*}
  \text{if $\mathscr F$ is convex and lower semicontinuous
    then}\quad
  \mathscr F=\mathscr F^{**}.
\end{equation*}
\end{theorem}
The following is the main duality result and it gives also an independent proof of Theorem \ref{ss22:teo:representation}.
\begin{theorem}[Duality] \label{ss22:teo:duality} Let $\sfH$ be as in \eqref{ss22:eq:assumption} and let $\cost{\sfH}$ and $\m_\sfH$ be as in Definition \ref{ss22:def:uh}. Then 
\begin{equation} \label{ss22:eq:dualformulation} \m_\sfH(\mu_1, \mu_2) = \cce{\cost{\sfH}}(\mu_1, \mu_2) = \sup \left \{ \dual(\varphi_1, \varphi_2; \mu_1, \mu_2) : (\varphi_1, \varphi_2) \in \Phi_\sfH \right \} \end{equation}
for every $(\mu_1, \mu_2) \in \meas_+(\sfX_1) \times \meas_+(\sfX_2)$.
\end{theorem}
\begin{proof} Set $E:= \meas(\sfX_1) \times \meas(\sfX_2)$ and $F:= \rmC_b(\sfX_1) \times \rmC_b(\sfX_2)$ with the bilinear form
\[ \scalprod{\cdot}{\cdot} : E \times F \to \R, \quad \quad ((\mu_1, \mu_2), (\varphi_1, \varphi_2)) \mapsto \dual(\varphi_1, \varphi_2; \mu_1, \mu_2). \]
This is a well defined nondegenerate bilinear form. We endow then $E$ with the topology $\sigma(E,F)$ which coincides exactly with the product weak topology.\\
Consider then the function $\cost{\sfH}: E \mapsto (-\infty, + \infty]$ defined as in Definition \ref{ss22:def:uh}. Then we have
\begin{align*}
 \cost{\sfH}^*(\varphi_1, \varphi_2) &= \sup_{(\mu_1, \mu_2) \in E} \left \{ \scalprod{(\mu_1, \mu_2)}{(\varphi_1, \varphi_2)} - \cost{\sfH}(\mu_1, \mu_2) \right \} \\
 &= \sup_{(\mu_1, \mu_2) \in \D{\sfX_1} \times \D{\sfX_2}} \left \{ \scalprod{(\mu_1, \mu_2)}{(\varphi_1, \varphi_2)} - \cost{\sfH}(\mu_1, \mu_2) \right \} \\
 &= \sup_{x_1, r_1, x_2, r_2} \left \{ \varphi_1(x_1) r_1 + \varphi_2(x_2) r_2 - \sfH([x_1,r_1], [x_2, r_2]) \right \} \\
 &= \begin{cases} 0 \quad &\text{ if } (\varphi_1, \varphi_2) \in \Phi_{\sfH} \\ + \infty \quad &\text{ elsewhere} \end{cases}. 
 \end{align*}
 Hence by Theorem \ref{ss22:teo:fenchel} we have
 \begin{equation}\label{ss22:eq:dualityfirst}
  \cce{\cost{\sfH}}(\mu_1, \mu_2) = \cost{\sfH}^{**}(\mu_1, \mu_2) = \sup \left \{ \dual(\varphi_1, \varphi_2; \mu_1, \mu_2) \mid (\varphi_1, \varphi_2) \in \Phi_\sfH \right \}.    
 \end{equation}
By Theorem \ref{ss22:prop:mh} we know that $\m_\sfH$ is convex and lower semicontinuous and stays below $\cost{\sfH}$ so that we clearly have $\m_\sfH \le \cce{\cost{\sfH}}$. The other inequality is immediate: take any $\aalpha \in \f{H}^1(\mu_1, \mu_2)$ and any $(\varphi_1, \varphi_2) \in \Phi_{\sfH}$; then
\[
\int_{\sfX_1} \varphi_1 \de \mu_1 + \int_{\sfX_2} \varphi_2 \de \mu_2 = \int_{\pc} \left ( (\varphi_1\circ \sfx_1)\sfr_1+ (\varphi_2\circ \sfx_2) \sfr_2 \right ) \de \aalpha\le \int_{\pc} \sfH \de \aalpha.
\]
Passing to the supremum in $\Phi_{\sfH}$ and to the infimum in $\f{H}^1(\mu_1, \mu_2)$ and using \eqref{ss22:eq:dualityfirst}, we conclude that $\m_\sfH \ge \cce{\cost{\sfH}}$.
\end{proof}

\begin{remark} If $\sfH$ has the form
\[ \sfH([x_1, r_1] , [x_2, r_2]) := \begin{cases} r_1 \sfc(x_1, x_2) \quad &\text{ if } r_1 = r_2 \ge 0 \\ + \infty \quad &\text{ elsewhere } \end{cases}, \]
for some $\sfc: \sfX_1 \times \sfX_2 \to [0, + \infty]$ lower semicontinuous function, we have that $\restricts{\m_\sfH}{\prob(\sfX_1) \times \prob(\sfX_2)}$ is exactly the classical Optimal Transport cost induced by $\sfc$. This was indeed exploited in \cite{SS20}.
\end{remark}

In Definition \ref{ss22:def:cof} we have introduced the closed and convex envelope of $\sfH$ which is obtained first convexifying $\sfH$ ``slice by slice" and then closing its graph globally in $\pc$. One could also compute first the closed convex envelope ``slice by slice" and then glue together the resulting functions (see the definition of $\cof{\sfH}$ below). In the following result we show, making use of Theorem \ref{ss22:teo:duality}, that the two procedures give raise to same object. 

We also show that, even if $\sfH$ is not convex, taking the closed and convex envelope of the singular cost $\cost{\sfH}$ is sufficient to also recover the Unbalanced Optimal Transport cost induced by the closed and convex envelope of $\sfH$.

\begin{corollary} Let $\sfH$ be as in \eqref{ss22:eq:assumption} and let $\cce{\sfH}$ be as in Definition \ref{ss22:def:cof}. Let us define $\cof{\sfH} : \pc \to [0, +\infty]$ as
\[
\cof{\sfH}([x_1, r_1], [x_2, r_2]) := \cce{\sfH_{x_1, x_2}}(r_1, r_2) \quad \text{ for every } (x_1, x_2) \in \sfX_1 \times \sfX_2, \, (r_1, r_2) \in \R_+^2.\]
Then
\[ \cce{\sfH}([x_1, r_1], [x_2, r_2]) = \cof{\sfH}([x_1, r_1], [x_2, r_2]) = \sup \left \{ \varphi(x_1)r_1+ \varphi_2(x_2)r_2 : (\varphi_1, \varphi_2) \in \Phi_\sfH \right \}\]
for every $(x_1, x_2) \in \sfX_1 \times \sfX_2$ and every $(r_1, r_2) \in \R_+^2$. Moreover
 \begin{equation}
     \label{eq:conv-cost}
\cce{\cost{\sfH}} = \cce{\cost{\cce{\sfH}}} =\m_{\sfH} = \m_{\cce{\sfH}} \text{ in } \meas(\sfX_1) \times \meas(\sfX_2).
 \end{equation}
\end{corollary}
\begin{proof} We denote by $\mathsf{U}_\sfH : \pc \to [0, + \infty]$ the restriction of $\m_\sfH$ to $\Delta_+(\sfX_1) \times \Delta_+(\sfX_2) \cong \pc$ (see Lemma \ref{ss22:lem:homeo}).\\
It is clear that $\cce{\sfH} \le \cof{\sfH} \le \sfH$ so that
\[ \m_{\cce{\sfH}} \le  \m_{\cof{\sfH}} \le \m_{\sfH}\] 
and 
\[ \cce{\sfH} =\mathsf{U}_{\cce{\sfH}} \le \cof{\sfH} = \mathsf{U}_{\cof{\sfH}} \le \mathsf{U}_\sfH \le \sfH,\]
where we used Theorem \ref{ss22:prop:mh} and the radial convexity of $\cce{\sfH}$ and $\cof{\sfH}$.
 Moreover, since $\mathsf{U}_\sfH$ is, lower semicontinuous, convex and stays below $\sfH$, we have that $\mathsf{U}_\sfH \le \cce{\sfH}$. This gives that
 \[ \cce{\sfH} =\mathsf{U}_{\cce{\sfH}} = \cof{\sfH} = \mathsf{U}_{\cof{\sfH}} = \mathsf{U}_\sfH\]
 and in particular, using Theorem \ref{ss22:teo:duality}, that 
 \[ \cce{\sfH}([x_1, r_1], [x_2, r_2]) = \cof{\sfH}([x_1, r_1], [x_2, r_2]) = \sup \left \{ \varphi(x_1)r_1+ \varphi_2(x_2)r_2 : (\varphi_1, \varphi_2) \in \Phi_\sfH \right \}.\]
 The fact that $\cce{\sfH} = \mathsf{U}_\sfH$ gives that $\m_{\sfH}
 \le \cost{\cce{\sfH}}$ so that $\m_{\sfH} \le
 \cce{\cost{\cce{\sfH}}}$. However, by Theorem \ref{ss22:prop:mh}, we
 know that $\m_{\sfH} = \cce{\cost{\sfH}}$ so that $\cce{\cost{\sfH}}
 \le \cce{\cost{\cce{\sfH}}}$. Since, obviously, the other inequality
 holds, we have $\cce{\cost{\sfH}} =
 \cce{\cost{\cce{\sfH}}}$. Applying again
 Theorem \ref{ss22:prop:mh} to $\cce{\sfH}$ we conclude that
 \[ \cce{\cost{\sfH}} = \cce{\cost{\cce{\sfH}}} =\m_{\sfH} =
   \m_{\cce{\sfH}}.
\qedhere\]
\end{proof}
\begin{remark}
\label{rem:interesting-app}
As an interesting application of the previous Corollary, we obtain a clarifying  justification of two equivalent characterizations of the Hellinger-Kantorovich metric
(see (7.23) and (7.54) of \cite{LMS18}) in a (complete, separable) metric space
$(\sfX,\sfd)$, namely the formula
\begin{align}
\label{eq:cone-identity}
    \mathsf{H\kern-3pt K}^2
    (\mu_1,\mu_2)&=
    \min_{\aalpha_2\in \f H^2(\mu_1,\mu_2)}
    \int \mathsf d_{\f C}^2\,\de\aalpha_2=
     \min_{\bbeta_2\in \f H^2(\mu_1,\mu_2)}
    \int \mathsf d_{\pi/2,\f C}^2\,\de\bbeta_2,
\end{align}
where $\sfd_{\f C}$ is the canonical cone metric introduced in \eqref{ss22:eq:distcone}
and 
$\mathsf d_{\pi/2,\f C}$ is
the cone metric obtained by trucating the 
argument of the $\cos$ function at $\pi/2$:
\begin{equation}\label{ss22:eq:distconetrunc} \sfd_{\pi/2,\f{C}}([x,r],[y,s]) := \left ( r^2+s^2-2rs\cos(\sfd(x,y) \wedge \pi/2) \right )^{\frac{1}{2}}, \quad [x,r], [y,s] \in \f{C}[\sfX].
\end{equation}
By Remark \ref{rem:tp},
the identity \eqref{eq:cone-identity} is equivalent to 
\begin{align}
\label{eq:cone-identity2}
    \mathsf{H\kern-3pt K}^2
    (\mu_1,\mu_2)&=
    \min_{\aalpha\in \f H^1(\mu_1,\mu_2)}
    \int \mathsf d_{\f C}^2\circ \mathsf T_2\,\de\aalpha=
     \min_{\bbeta\in \f H^1(\mu_1,\mu_2)}
    \int \mathsf d_{\pi/2,\f C}^2 \circ \mathsf T_2\,\de\bbeta,
\end{align}
where $\mathsf T_2$ is defined in \eqref{eq:tq}; 
it is then not difficult to check that 
\begin{equation}
    \label{eq:convexification-app}
    \mathsf d^2_{\pi/2,\f C}\circ \mathsf T_2=
    \cce{\mathsf d^2_{\f C}\circ \mathsf T_2}.
\end{equation}
Therefore, the constant $\pi/2$ in \eqref{eq:cone-identity}
appears as a natural effect of the convexification of the cost function $\sfd_{\f C}^2\circ \mathsf T_2$ 
and the identity \eqref{eq:conv-cost}.
\end{remark}
\subsection{The Monge formulation}
We define the Monge formulation of the Unbalanced Optimal Transport problem. 
\begin{definition}[Transport-growth maps and Monge formulation]\label{def:tgmap}
    Given $\mu_1 \in \meas_+(\sfX_1)$ and Borel maps $(\mathsf{T}, \sfg): \sfX_1 \to \sfX_2 \times [0,+\infty)$ with $\sfg \in L^1(\sfX_1, \mu_1)$, we denote by $(\mathsf T,\sfg)_\star \mu_1$ the measure
    \[ (\mathsf{T}, \sfg)_\star \mu_1 := \mathsf{T}_\sharp (\sfg \mu_1) \in \meas_+(\sfX_2).\]
    We say that $(\mathsf{T}, \sfg)$ is a \emph{transport-growth map} connecting $\mu_1$ to $\mu_2$. If $\sfH : \pc \to [0,+\infty]$ is a proper Borel function, we define the Monge formulation of the Unbalanced Optimal Transport problem as
    \[ \mathscr M_{\sfH}(\mu_1,\mu_2):=\inf\Big\{\int_{\sfX_1}
    \sfH([x,1],[\mathsf T(x),\mathsf g(x)])\,\de\mu_1(x)\mid
    \mu_2=(\mathsf T,\mathsf g)_\star \mu_1\Big\}.\]
\end{definition}
While in this section we only study the relation between the Kantorovich (i.e.~the primal) and the Monge fofmulation of the Unbalanced Optimal Transport problem, in Theorem \ref{theo:themap} we will provide sufficient conditions for the existence of an optimal transport-growth map realizing the infimum above. Clearly we have
\[ \m_\sfH(\mu_1, \mu_2) \le \mathscr M_{\sfH}(\mu_1,\mu_2), \quad \mu_i \in \meas_+(\sfX_i)\]
as a consequence of the fact that any transport-growth map $(\mathsf{T},\sfg)$ induces an admissible $\aalpha \in \f{H}^1(\mu_1, \mu_2)$ via the formula
\begin{displaymath}
    \aalpha=([\text{id}_{\sfX_1},1],[\mathsf T,\mathsf g])_\sharp \mu_1.
    \end{displaymath}
     Recalling the definition
    \eqref{eq:reduced}
    of the reduced cost function $\mathsf h$, we also have
    \begin{displaymath}
        \int_{\pc} \sfH\,\de\aalpha=
        \int_{\sfX_1} \sfH([x,1],[\mathsf T(x),\mathsf g(x)])\,\de\mu_1(x)=
         \int_{\sfX_1} \mathsf h(x,\mathsf T(x);\mathsf g(x))\,\de\mu_1(x).        
    \end{displaymath}

We derive the main result of this subsection from the analogous one \cite{pratelli} in the classical Optimal Transport theory. To do that we need the following definition.

\begin{definition}\label{def:wass} Let $\sfH: \pc \to [0.+\infty]$ be a proper and lower semicontinuous function. We define
\[ \mathsf{OT}_\sfH(\alpha_1, \alpha_2):= \min \left \{ \int_{\pc} \sfH \de \ggamma : \ggamma \in \Gamma(\alpha_1, \alpha_2) \right \}, \quad  \alpha_i \in \meas_+(\f{C}[\sfX_i]),\]
where $\Gamma(\alpha_1, \alpha_2)$ is the set of transport plans from $\alpha_1$ to $\alpha_2$ defined as
\[ \Gamma(\alpha_1, \alpha_2) := \left \{ \ggamma \in \meas_+(\pc) : \pi^1_\sharp \ggamma = \alpha_1, \, \pi^2_\sharp \ggamma = \alpha_2 \right \}.\]
\end{definition}
Notice that $\Gamma(\alpha_1, \alpha_2)$ is not-empty if and only if $\alpha_1(\f{C}[\sfX_1])=\alpha_2(\f{C}[\sfX_2])$.
The following is an immediate consequence of \cite[Proof of Theorem B]{pratelli}.

\begin{theorem}\label{thm:pra} Let $\sfX_i$ be Polish spaces, let $\sfH: \pc \to [0,+\infty]$ be continuous and let $\alpha_i \in \meas_+(\f{C}[\sfX_i])$ be such that $\alpha_1$ is diffuse and $\alpha_1(\f C[\sfX_1]) >0$. Then for every $\ggamma \in \Gamma(\alpha_1, \alpha_2)$ such that $\int_{\pc} \sfH \de \ggamma <+\infty$ and every $\eps>0$ there exists a Borel map $\mathsf{F}: \f{C}[\sfX_1] \to \f{C}[\sfX_2] $ such that 
\[ \mathsf{F}_\sharp \alpha_1 = \alpha_2, \quad \int_{\f{C}[\sfX_1]} \sfH(\f{y}_1, \mathsf F (\f{y}_2)) \de \alpha_1(\f{y}_1) \le \int_{\pc} \sfH \de \ggamma  + \eps.\]    
\end{theorem}
Choosing in particular an optimal $\ggamma$ realizing the minimum in the definition of $\mathsf{OT}_\sfH$ one can prove the equivalence between the Monge and Kantorovich formulations of the Optimal Transport problem. The connection between balanced and unbalanced Optimal Transport is of course given by the fact that if $\aalpha \in \f{H}^1(\mu_1, \mu_2)$ for measures $\mu_i \in \meas_+(\sfX_i)$, then $\aalpha \in \Gamma(\alpha_1, \alpha_2)$, where $\alpha_i:= \pi^i_\sharp \aalpha$. A stronger connection actually holds, as reported in the following result which is the analogue of \cite[Corollary 7.7, Corollary 7.13]{LMS18}; 
the proof of the second statement 
is identical and thus omitted,
the proof of the first statement follows immediately by Lemma 
\ref{ss22:le:mineq}.

\begin{proposition}
\label{ss22:lem:comparison} Let $\sfH$ be as in \eqref{ss22:eq:assumption}. Then for every $\mu_i \in \meas_+(\sfX_i)$ with $R=R(\mu_1,\mu_2)$ as in \eqref{eq:theradi}, it holds
\[ \m_{\sfH}(\mu_1, \mu_2) = \min \left \{ \mathsf{OT}_\sfH(\alpha_1, \alpha_2) : \alpha_i \in \prob(\f{C}_{R}[\sfX_i]) \, , \,  \f{h}^1(\alpha_i) = \mu_i, \, \, i=1, 2 \right \}.\]
In particular we have $\m_{\sfH}(\mu_1, \mu_2) = \mathsf{OT}_\sfH(\pi^1_\sharp \aalpha, \pi^2_\sharp \aalpha)$ for every optimal $\aalpha \in \f{H}^1_o(\mu_1, \mu_2)$ (cf.~\eqref{ss22:eq:optplan}).
If $\sfX=\sfX_1=\sfX_2$ and $(\mu_i)_{i=1}^N \subset \meas_+(\sfX)$ with $N \ge 2$, then there exist $(\alpha_i)_{i=1}^N \subset \prob(\f{C}[\sfX])$ such that 
 \[ \f{h}^1(\alpha_i) = \mu_i, \quad \m_\sfH(\mu_{i-1}, \mu_i) = \mathsf{OT}_\sfH(\alpha_{i-1}, \alpha_{i}) \quad \text{ for every }  i \in \{2, \dots, N\}.\]
\end{proposition}

We first show that the cost of any coupling $\aalpha \in \f{H}^1(\mu_1, \mu_2)$, $\mu_i \in \meas_+(\sfX_i)$, can be approximated by the cost of couplings with the same homogeneous marginals not charging the vertex of the first cone.

\begin{proposition}\label{prop:approxnov}
Let $\sfH: \pc \to [0,+\infty]$ be a proper, radially $1$-homogeneous and continuous function, and let $\mu_i \in \meas_+(\sfX_i)$ be such that $\mu_1(\sfX_1)>0$.
Then for every $\eps>0$ and every $\aalpha \in \f{H}^1(\mu_1, \mu_2)$ such that $\int_{\pc} \sfH \de \aalpha < +\infty$ there exists $\aalpha_\eps \in \f{H}^1(\mu_1, \mu_2)$ with $\aalpha_\eps(\{\f{y}_1 = \f{o}_1\})=0$ such that
\[  \int_{\pc} \sfH \de \aalpha_\eps  \le  \int_{\pc} \sfH \de \aalpha
+\eps .\]
\end{proposition}
\begin{proof}
     We define the measures
\[ \aalpha^{tg}:= \aalpha \mres \{ \f{y}_1 \ne \f{o}_1\}, \quad \aalpha^{o}:= \aalpha \mres \{ \f{y}_1= \f{o}_1\}, \quad \mu_2^{tg}:= \f{h}^1_2(\aalpha^{tg}), \quad \mu_2^o:= \f{h}^1_2(\aalpha^o);\]
it is not restrictive to assume that $\mu_2^o(\sfX_2)>0$ so that $\aalpha^o(\pc)>0$. For a bounded Borel map $\vartheta:\pc\to (0,\infty)$ and a measure
$\bbeta\in \meas_+(\pc)$
we set 
\begin{align*}
\mathrm{prd}_{\vartheta,2}(\f y_1,\f y_2):={}&
(\f y_1,\f y_2/\vartheta(\f y_1,\f y_2)),\quad\f y_i\in \f C[\sfX_i],\\
\dil_{\vartheta,2}(\bbeta):={}&
(\mathrm{prd}_{\vartheta,2})_\sharp(\vartheta\,\bbeta).
\end{align*}
Notice that $\f h^1_2(\dil_{\vartheta,2}(\bbeta))=\f h^1_2(\bbeta)$
and 
\begin{equation}
    \label{eq:dil-cost}
    \int_{\pc} \sfH\,\de\big(
    \dil_{\vartheta,2}(\bbeta)\big)=
    \int_{\pc} \sfH(\vartheta(\f y_1,\f y_2)\f y_1,\f y_2)\,\de\bbeta(\f y_1, \f y_2).
\end{equation}
We fix $\eps>0$ and we set
\begin{align*}
    \lambda_\eps'(\f y_1,\f y_2):=
    \sup
    \Big\{&\lambda\in [0,1/2]:
    \sfH(r\f{y}_1,\f y_2)\le 
    \sfH(\f o_1,\fy_2)+\eps_0
    \text{ for every }
    r\in [0,\lambda]\Big\},\\
    \lambda_\eps''(\f y_1,\f y_2):=
    \sup
    \Big\{&\lambda\in [0,1/2]:
    \sfH(r\f y_1,\f y_2)\le \sfH(\f y_1,\fy_2)+\eps_0
    \text{ for every }
    r\in [1-\lambda,1]\Big\},
\end{align*}
where $\eps_0$ is given by
\[ \eps_0:= \frac{1}{2} \frac{\eps}{\mu_1(\sfX_1) + \aalpha^{tg}(\f C [\sfX_1, \sfX_2])}.\]
Since $\sfH$ is continuous, it is not difficult to check that $\lambda_\eps',\lambda_\eps''$
define strictly positive upper semicontinuous maps (thus Borel)
from $\pc$ to $(0,1/2].$ Now we set $c:=\mu_2^o(\sfX_2)/\mu_1(\sfX_1)$ and we consider any $ \ggamma \in \Gamma(\mu_1, c^{-1}\mu_2^o)$. 
We ``lift'' $\ggamma$ to a plan 
$\bbeta:=\ell_\sharp\ggamma\in \f H^1(\mu_1,\mu_2^o)$
using the map
$\ell(x_1,x_2):=([x_1,1],[x_2,c])$ and we consider the rescaled plan $\bbeta'_\eps$ and its first homogeneous marginal
\begin{displaymath}
    \bbeta_\eps'=\dil_{\lambda_\eps',2}(\bbeta),\quad
    \mu_{1,\eps}':=\f h^1_1(\bbeta_\eps').
\end{displaymath}
Since $\lambda_\eps'\le 1/2 $
we easily get 
$\mu_{1,\eps}'\le \frac 12 \mu_1\ll \mu_1$ and we call
$\varrho_{1,\eps}'$ the Lebesgue density of $\mu_{1,\eps}'$ with respect to $\mu_1$, i.e.
\begin{equation}
    \label{eq:Lebdens1}
    \mu_{1,\eps}'=\varrho_{1,\eps}'\mu_1,\quad
    \varrho_{1,\eps}'\quad\text{is a Borel map with values in $(0,1/2]$}.
\end{equation}
In a similar way we 
define the homogeneous plan
$\bbeta_\eps''$ and its first homogeneous marginal
\[  \bbeta_\eps'':= 
\dil_{(1-\lambda_\eps''),2}(\aalpha^{tg}),\quad
\mu_{1,\eps}'':=\f h^1_1(\bbeta''_\eps),\quad
\mu_{1,\eps}''=\varrho_{1,\eps}''\mu_1.\]
We can select a Borel representative of $\varrho_{1,\eps}''$ 
 such that 
 $(1-\varrho_{1,\eps}'')$ 
 takes values in $(0,1/2]$.
We eventually consider
\begin{align}
    \label{eq:defrho}
    \varrho_\eps:=\varrho_{1,\eps}'\land (1-\varrho_{1,\eps}''),\quad
    \vartheta_\eps':={}&
    \frac{\varrho_\eps}{\varrho_{1,\eps}'},&\quad
    \vartheta_\eps'':={}&
    \frac{\varrho_\eps}{1-\varrho_{1,\eps}''},\\
    \sigma_\eps':={}&
    \vartheta_\eps'\lambda_\eps',&\quad    \sigma_\eps'':={}&\vartheta_\eps''(1-\lambda_\eps'')+(1-\vartheta_\eps''),\\
    \aalpha_\eps^{o}:={}&
    \dil_{\sigma_\eps',2}
    (\bbeta),&
    \aalpha_\eps^{tg}:={}&
    \dil_{\sigma_\eps'',2}(\aalpha^{tg}).
    \end{align}
Since the $\dil_{\cdot,2}$
dialations preserve the second homogeneous marginals, 
we easily get
\begin{equation}
    \label{eq:second-marginal-ok}
    \f h^1_2(\aalpha_\eps^o)=\mu_2^o,
    \quad
    \f h^1_2(\aalpha_\eps^{tg})=
    \mu_2^{tg}.
\end{equation} 
Concerning the first marginals, for every test function $\zeta \in \rmC_b(\sfX_1)$ we have
\begin{align*}
    \int_{\pc} (\zeta \circ \sfx_1)\sfr_1\,\de\aalpha_\eps^o&=
    \int_{\pc} (\zeta \circ \sfx_1) \sfr_1\vartheta_\eps'\lambda_\eps' \,\de\bbeta=
    \int_{\pc} (\zeta \circ \sfx_1) \sfr_1\vartheta_\eps'\,\de\bbeta_\eps'
    \\&=
    \int_{\sfX_1}
    \zeta\vartheta_\eps'\,\de \mu_{1,\eps}'
    =
    \int_{\sfX_1}  \zeta\vartheta_\eps'
    \varrho_\eps'\,\de \mu_{1}=
    \int_{\sfX_1}    \zeta\rho_\eps \,\de \mu_{1},\\
    \int_{\pc} (\zeta \circ \sfx_1) \sfr_1\,\de\aalpha_\eps^{tg}&=
    \int_{\pc} (\zeta \circ \sfx_1) \sfr_1\sigma_\eps''\,\de\aalpha^{tg}\\&=
    \int_{\pc} (\zeta \circ \sfx_1) \sfr_1\vartheta_\eps''(1-\lambda_\eps'') \,\de\aalpha^{tg}
    +
    \int_{\pc} (\zeta\circ \sfx_1) \sfr_1(1-\vartheta_\eps'') \,\de\aalpha^{tg}
    \\&=
    \int_{\pc}(\zeta\circ\sfx_1)\sfr_1\vartheta_\eps''\,\de\bbeta_\eps''
    +
    \int_{\sfX_1} \zeta(1-\vartheta_\eps'') \,\de\mu_1
    \\&=
    \int_{\sfX_1} \zeta\vartheta_\eps''\,\de\mu_{1,\eps}''
    +
    \int_{\sfX_1} \zeta(1-\vartheta_\eps'') \,\de\mu_1
    \\&=
    \int_{\sfX_1} \zeta \Big(\vartheta_\eps''\varrho_{1,\eps}''
    +
    (1-\vartheta_\eps'')\Big) \,\de\mu_1
        \\&
        =
        \int_{\sfX_1} \zeta (1-\varrho_\eps)\,\de\mu_1
\end{align*}
so that 
\begin{align*}
    \f h^1_1(\aalpha_\eps^o)&=
    \varrho_\eps\mu_1,\quad
    \f h^1_1(\aalpha_\eps^{tg})=
    (1-\varrho_\eps)\mu_1.
\end{align*}
We deduce that the plan 
\[ \aalpha_\eps:= \aalpha_\eps^{tg} + \aalpha_\eps^o\quad\text{belongs to}\quad
\f H^1(\mu_1,\mu_2).\] 
By linearity, the $\sfH$-cost associated with $\aalpha_\eps$
is the sum of the corresponding costs 
associated with $\aalpha_\eps^{tg}$ and 
$\aalpha_\eps^o$;
using \eqref{eq:dil-cost}
and observing that 
$0< \sigma_\eps'\le \lambda_\eps'$,
$1-\lambda_\eps''\le \sigma_\eps''\le 1$ we obtain
\begin{align*}
    \int_{\pc} \sfH \de \aalpha_\eps^{o} 
    & = \int_{\pc} 
    \sfH\big(\sigma'_\eps \f{y}_1,\f y_2\big) \de \bbeta(\f y_1, \f y_2)\le 
    \int_{\pc} 
    \Big(\sfH\big(\f o_1,\f y_2\Big)+\eps_0\Big) \de \bbeta(\f y_1, \f y_2)
    \\&
    =
    \int_{\pc}\sfH \de \aalpha^o + \eps_0 \mu_1(\sfX_1) \le \int_{\pc}\sfH \de \aalpha^o + \eps/2, 
    \\
    \int_{\pc} \sfH \de \aalpha_\eps^{tg} & = 
    \int_{\pc} \sfH(\sigma_\eps''\f y_1,\f y_2) \de \aalpha^{tg}(\f y_1, \f y_2)
    \le 
    \int_{\pc}(\sfH +\eps) \de \aalpha^{tg} \\
    & = \int_{\pc} \sfH \de \aalpha^{tg} + \eps \aalpha^{tg}(\pc) \le \int_{\pc} \sfH \de \aalpha^{tg} + \eps/2.
    \qedhere
\end{align*}
\end{proof}

\begin{theorem}\label{thm:attained} Let $\sfX_i$ be Polish spaces, let $\sfH: \pc \to [0,+\infty]$ be a proper, radially $1$-homogeneous and continuous function, and let $\mu_i \in \meas_+(\sfX_i)$ be such that $\mu_1$ is a diffuse measure and $\mu_1(\sfX_1)>0$.
Then for every $\eps>0$ and every $\aalpha \in \f{H}^1(\mu_1, \mu_2)$ such that $\int_{\pc} \sfH \de \aalpha < +\infty$ there exists a transport-growth map $(\mathsf{T}, \sfg)$ connecting $\mu_1$ to $\mu_2$ such that
\[  \int_{\sfX_1} \sfH([x_1, 1],[\mathsf{T}(x_1), \sfg(x_1)]) \de \mu_1(x_1)
\le \int_{\pc} \sfH \de \aalpha + \eps.\]
In particular
\[ \m_\sfH(\mu_1, \mu_2) = \mathscr M_\sfH(\mu_1, \mu_2).\]
\end{theorem}
\begin{proof}
Let $\aalpha \in \f{H}^1(\mu_1, \mu_2)$ with $\int_{\pc} \sfH \de \aalpha <+\infty$ and $\eps>0$ be fixed. By Proposition \ref{prop:approxnov} we can find $\aalpha_\eps \in \f{H}^1(\mu_1, \mu_2)$ with $\aalpha_\eps(\{\f{y}_1 = \f{o}_1\})=0$ such that
\[  \int_{\pc} \sfH \de \aalpha_\eps  \le  \int_{\pc} \sfH \de \aalpha
+\eps/2 .\]
Using Lemma \ref{ss22:le:dialations} applied to $\aalpha_\eps$ and $\vartheta$ defined as
\[ \vartheta([x_1,r_1],[x_2,r_2]) = \begin{cases} r_1 \quad &\text{ if } r_1 >0, \\ 1 \quad &\text{ if } r_1=0, \end{cases}\]
we can assume that $\aalpha_\eps$ is concentrated on $\{ \sfr_1 =1\}$. Since $\mu_1$ is non atomic, we deduce that also $\alpha_\eps^1:=\pi^1_\sharp \aalpha_\eps$ is non atomic so that, also observing that $\sfH$ is continuous, we can find thanks to Theorem \ref{thm:pra} a Borel map $\mathsf F : \f{C}[\sfX_1] \to \f{C}[\sfX_2]$ such that $\mathsf F_\sharp \alpha_\eps^1 = \alpha_\eps^2 := \pi^2_\sharp \aalpha_\eps$ and 
\begin{equation}\label{eq:comb2}
\int_{\f{C}[\sfX_1]} \sfH(\f{y}_1, \mathsf{F}(\f{y}_1)) \de \alpha_\eps^1(\f{y}_1) - \eps/2 \le
\int_{\pc} \sfH \de \aalpha_\eps.
\end{equation}
We define the Borel maps
\[ \mathsf{T}(x_1):= (\sfx \circ \mathsf{F})([x_1,1]), \quad \sfg(x_1):= (\sfr \circ \mathsf{F})([x_1,1]), \quad x_1 \in \sfX_1,\]
and we notice that $(\mathsf{T},\sfg)_\star \mu_1 = \mu_2$ and
\[ \int_{\f{C}[\sfX_1]} \sfH(\f{y}_1, \mathsf{F}(\f{y}_1)) \de \alpha_\eps^1(\f{y}_1) = \int_{\sfX_1} \sfH([x_1, 1],[\mathsf{T}(x_1), \sfg(x_1)]) \de \mu_1(x_1). \]
This concludes the proof.
\end{proof}

\section{Existence of a maximazing pair for the dual problem}\label{sec:5}
In this section we provide sufficent conditions for the existence of a maximizing pair $(\varphi_1, \varphi_2) \in \Phi_\sfH$ (see Definition \ref{def:thesetadm}) i.e. such that $(\varphi_1, \varphi_2) \in \rmC_b(\sfX_1) \times \rmC_b(\sfX_2)$ and
\begin{align*}
    \varphi_1(x_1) r_1 + \varphi_2(x_2)r_2 \le \sfH([x_1, r_1], [x_2, r_2]) &\text{ for every } ([x_1,r_1],[x_2,r_2]) \in \pc,\\
 \int_{\sfX_1}\varphi_1 \de \mu_1 + \int_{\sfX_2} \varphi_2 \de \mu_2 &= \m_\sfH(\mu_1, \mu_2).
\end{align*}
A somehow complementary set of assumptions for which the same conclusion holds is presented in Appendix \ref{app:1}. The following result is a simple consequence of compactness.

\begin{lemma}\label{le:balls} Let $(\sfX_i, \sfd_i)$ be compact metric spaces and assume that $\Omega \subset \sfX_1 \times \sfX_2$ is an open set such that $\pi^i(\Omega)=\sfX_i$, $i=1,2$. Then there exists a finite set $\mathcal{U}:=\{ x_1^n, x_2^n, r_n \}_{n=1}^N \subset \sfX_1 \times \sfX_2 \times (0, + \infty)$ such that 
\begin{equation}\label{eq:decompss}
\bigcup_n B_{r_n}(x_i^n) = \sfX_i, \quad i=1, 2 \quad \quad \bigcup_{n} \overline{B_{r_n}(x_1^n)} \times \overline{B_{r_n}(x_2^n)} \subset \Omega.
\end{equation}
If $\mu_i \in \meas_+(\sfX_i)$ is such that $\supp{\mu_i}=\sfX_i$, then
\begin{equation}\label{eq:mm} m(\mathcal{U}, \mu_1, \mu_2) := \min_{i=1,2} \min_{n=1, \dots, N} \mu_i \left (B_{r_n}(x_i^n) \right )>0.\end{equation}
\end{lemma}
In order to use this result, in this section, we are going to assume without mentioning it again that $\sfX_1$ and $\sfX_2$ are compact and metrizable spaces. We also assume that
\begin{equation}\label{ass:1}
    \begin{split}
        \sfH: \pc \to [0, + \infty) \text{ is continuous, radially $1$-homogeneous and convex } \\
        \text{and that there exists an open set } \Omega_\sfH \subset \sfX_1 \times \sfX_2 \text{ with } \pi^i(\Omega_\sfH) = \sfX_i \text{ such that }\\
    \lim_{r_1 \downarrow 0} \frac{ \sfH([x_1, r_1], [x_2, 1])-\sfH([\f{0}_1, [x_2,1])}{r_1} =-\infty \quad \text{ for every } (x_1, x_2) \in \Omega_\sfH,\\
    \lim_{r_2 \downarrow 0} \frac{ \sfH([x_1, 1], [x_2, r_2])-\sfH([x_1,1],\f{0}_2)}{r_2} =- \infty \quad \text{ for every } (x_1, x_2) \in \Omega_\sfH.
    \end{split}
\end{equation}
To simplify the notation, we set 
\[
 \sfH_1(x_1) := \sfH([x_1, 1], \f{o}_2), \quad \sfH_2(x_2) := \sfH( \f{o}_1, [x_2, 1]) \quad x_1 \in \sfX_1, \, x_2 \in \sfX_2
\]
and 
\begin{equation}\label{ss22:eq:bdd} \kappa_1:=\int_{\sfX_1} \sfH_1 \de \mu_1 < + \infty, \quad \kappa_2:=\int_{\sfX_2} \sfH_2 \de \mu_2 < + \infty.
\end{equation}
Clearly $\sfH \le \sfH_1+\sfH_2$; we remark that the meaning of \eqref{ass:1} is to impose a control on the derivatives of $\sfH$ at the boundary of $\R^2_+$ for a sufficiently large set of points $(x_1,x_2)$.
\begin{example}
    Both the functions in \eqref{eq:hk} and \eqref{eq:ghk} satisfy \eqref{ass:1} with $\Omega_{\sfH_{\mathsf{GHK}}}=\R^d \times \R^d$ and $\Omega_{\sfH_{\mathsf{HK}}}= \{ (x_1,x_2) \in \R^d \times \R^d : |x_1-x_2|< \pi/2 \}$.
\end{example}
We start with a few preliminary lemmas that provide bounds on pairs in $\Phi_\sfH$.

\begin{lemma} \label{ss22:le:balls} Assume that $\sfH$ is as in \eqref{ass:1} and let $\mathcal{U}=\{ x_1^n, x_2^n, r_n \}_{n=1}^N$ be as in Lemma \ref{le:balls} for distances $\sfd_i$ metrizing $\sfX_i$, $i=1,2$ and $\Omega=\Omega_\sfH$. If $\mu_i \in \meas_+(\sfX_i)$ are such that $\supp{\mu_i} = \sfX_i$, then any pair $(\varphi_1, \varphi_2) \in \Phi_\sfH$ such that
\[ \int_{\sfX_1} \varphi_1 \de \mu_1 + \int_{\sfX_2} \varphi_2 \de \mu_2 \ge 0\]
satisfies also
\[ \max_{\overline{B_{r_n}(x_i^n)}} \varphi_i \ge -\frac{\kappa_1+\kappa_2+1}{m(\mathcal{U}, \mu_1, \mu_2)} \quad \text{ for every } n \in \{1, \dots, N\} \text{ and } i=1,2,\]
where $\kappa_i$ are as in \eqref{ss22:eq:bdd} and $m(\mathcal{U}, \mu_1, \mu_2)$ is as in \eqref{eq:mm}.
\end{lemma}
\begin{proof}We claim that 
\begin{equation} \label{ss22:eq:belowest} \int_{B_{r_i}(x_1^n)} \varphi_1 \de \mu_1 \ge -(\kappa_1+\kappa_2+1), \quad \int_{B_{r_i}(x_2^n)} \varphi_2 \de \mu_2 \ge -(\kappa_1+\kappa_2+1) \quad n=1, \dots, N.
\end{equation}
Indeed, if there exists $i \in \{1,2\}$ (say $i=1$) and $n \in \{1, \dots, N\}$ such that 
\[ \int_{B_{r_n}(x_1^n)} \varphi_1 \de \mu_1 < -(\kappa_1+\kappa_2+1),\]
then 
\[ \int_{\sfX_1} \varphi_1 \de \mu_1 = \int_{B_{r_n}(x_1^n)} \varphi_1 \de \mu_1 + \int_{\sfX_1 \setminus B_{r_n}(x_1^n)} \varphi_1 \de \mu_1 < -(\kappa_1+\kappa_1+1) + \kappa_1= -(\kappa_2+1).\]
Thus
\[ \kappa_2 \ge \int_{\sfX_2} \varphi_2 \de \mu_2 \ge -\int_{\sfX_1} \varphi_1 \de \mu_1 \ge \kappa_2+1,\]
a contradiction. By \eqref{ss22:eq:belowest} we have, for every $i=1,2$ and $n=1, \dots, N$, that
\[  \mu_i \left (B_{r_n}(x_i^n) \right ) \sup_{B_{r_n}(x_i^n)} \varphi_i  \ge \int_{B_{r_n}(x_i^n)} \varphi_i \de \mu_i \ge -(\kappa_1+\kappa_2+1), \]
hence
\[ \sup_{B_{r_n}(x_i^n)} \varphi_i \ge - \frac{\kappa_1+\kappa_2+1}{m(\mathcal{U}, \mu_1, \mu_2)}.\]
\end{proof}

\begin{lemma} \label{ss22:le:staccato} Assume that $\sfH$ is as in \eqref{ass:1} and that $\mu_i \in \meas_+(\sfX_i)$ are such that $\supp{\mu_i} = \sfX_i$. Then there exists a constant $\eps>0$ such that any pair $(\varphi_1, \varphi_2) \in \Phi_\sfH$ such that
\[ \int_{\sfX_1} \varphi_1 \de \mu_1 + \int_{\sfX_2} \varphi_2 \de \mu_2 \ge 0\]
satisfies also
\[\varphi_i(x_i) \le \sfH_i(x_i) - \eps \quad \text{ for every } x_i \in \sfX_i, \quad i=1,2.\]
\end{lemma}

\begin{proof} 

We prove the statement for $i=1$, being the other case completely analogous. Suppose by contradiction that there exists $(\varphi_1^j, \varphi_2^j)_j \subset \Phi_\sfH$ with $\int_{\sfX_1} \varphi_1^j \de \mu_1 + \int_{\sfX_2} \varphi_2^j \de \mu_2 \ge 0$ and $(z_j)_j \subset \sfX_1$ such that
\[ \sfH_1(z_j)-\varphi_1^j(z_j) \to 0\]
as $j \to + \infty$. Up to passing to a subsequence, we can assume that
\[ 0 \le \sfH_1(z_j)-\varphi_1^j(z_j) \le \frac{1}{j} \quad \text{ for every } j \in \N
\]
and, by compactness of $\sfX_1$, the existence of $z \in \sfX_1$ such that $z_j \to z$. Let $\mathcal{U}=\{ x_1^n, x_2^n, r_n \}_{n=1}^N$ be as in Lemma \ref{le:balls} for distances $\sfd_i$ metrizing $\sfX_i$, $i=1,2$ and $\Omega=\Omega_\sfH$; let 
\[ C:=\frac{\kappa_1+\kappa_2+1}{m(\mathcal{U}, \mu_1, \mu_2)}. \]
Since $z \in B_{r_n}(x_1^n)$ (see \eqref{eq:decompss}) for some $n \in \{1, \dots, N\}$, we can assume, up to passing again to a subsequence, that $z_j \in B_{r_n}(x_1^n)$ for every $j \in \N$. By Lemma \ref{ss22:le:balls}, we can find, $y_j \in \overline{B_{r_n}(x_2^n)}$ such that $\varphi_2^j(y_j) \ge -C$. By compactness of $\sfX_2$, we can assume that $y_j \to y \in \overline{B_{r_n}(x_2^n)}$. We have thus proven the existence of $(z_j, y_j) \in \Omega$ such that $(z_j,y_j) \to (z,y) \in \Omega$ with 
\[ 0 \le \sfH_1(z_j)-\varphi_1^j(z_j) \le \frac{1}{j},\quad  \varphi_2^j(y_j) \ge -C \quad \text{ for every } j \in \N.
\]
We have
\[ r_1\left (\sfH_1(z_j)-\frac{1}{j} \right ) -Cr_2 \le \varphi_1^j(z_j) r_1 + \varphi_2^j(y_j)r_2 \le \sfH([z_j,r_1], [y_j,r_2]) \quad \text{ for every } r_1, r_2 \ge 0.\]
Choosing $r_1=1$, we get
\[ \frac{ \sfH([z_j, 1], [y_j, r_2]-\sfH_1(z_j)}{r_2} \ge -C - \frac{1}{jr_2} \quad \text{ for every } j \in \N, \, r_2 >0.\]
Passing first to the limit as $j \to + \infty$ and then to the limit as $r_2 \downarrow 0$, we obtain
\[ \lim_{r_2 \downarrow 0} \frac{ \sfH([z, 1], [y, r_2]-\sfH_1(x_\infty)}{r_2} \ge -C > - \infty,\]
a contradiction with \eqref{ass:1}.
\end{proof}

The following definition is simply the analogue of the classical $\mathsf{c}$-transform (see e.g.~\cite[Definition 6.1.2]{AGS08}).
\begin{definition}\label{def:htr} Let $(\varphi_1, \varphi_2) \in \Phi_\sfH$. We define the Borel functions $\varphi_1^\sfH : \sfX_2 \to \R$, $\varphi_1^{\sfH \sfH}: \sfX_1 \to \R$ as
\begin{align*}
    \varphi_1^{\sfH}(x_2) &:= \inf_{x_1 \in \sfX_1} \inf_{\alpha \ge 0} \biggl \{ \sfH([x_1, \alpha], [x_2, 1])-\alpha \varphi_1(x_1) \biggr \}, \quad x_2 \in \sfX_2,\\
    \varphi_1^{\sfH \sfH}(x_1) &:=\inf_{x_2 \in \sfX_2} \inf_{\alpha \ge 0} \biggl \{ \sfH([x_1, 1], [x_2, \alpha])-\alpha \varphi_1^\sfH(x_2) \biggr \}, \quad x_1 \in \sfX_1.
\end{align*}

\end{definition}

Using the previous lemmas, in the following proposition we prove that the $\sfH$-transform of a pair in $\Phi_\sfH$ can be computed restricting the minimization to a compact set. As a consequence, we obtain uniform bounds and uniform continuity for the admissible pair, as it happens in the classical case (see e.g.~the discussion after \cite[Definition 1.10]{Santambrogio15}).
\begin{proposition} \label{ss22:prop:propertiesphih} Assume that $\sfH$ is as in \eqref{ass:1} and that $\mu_i \in \meas_+(\sfX_i)$ are such that $\supp{\mu_i} = \sfX_i$. Then there exists constants $R>1$ and $M>0$ such that, if $(\varphi_1, \varphi_2) \in \Phi_\sfH$ are such that 
\[ \int_{\sfX_1} \varphi_1 \de \mu_1 + \int_{\sfX_2} \varphi_2 \de \mu_2 \ge 0,\]
then $\|\varphi_1^\sfH \|_\infty, \|\varphi_1^{\sfH \sfH} \|_\infty  \le M$ and 
\begin{align}\label{ss22:eq:varphihr1}
    \varphi_1^{\sfH}(x_2) &= \inf_{x_1 \in \sfX_1} \inf_{0 \le \alpha \le R} \biggl \{ \sfH([x_1, \alpha], [x_2, 1])-\alpha \varphi_1(x_1) \biggr \}, \quad x_2 \in \sfX_2,\\ \label{ss22:eq:varphihr2}
    \varphi_1^{\sfH \sfH}(x_1) &=\inf_{x_2 \in \sfX_2} \inf_{0 \le \alpha \le R} \biggl \{ \sfH([x_1, 1], [x_2, \alpha])-\alpha \varphi_1^\sfH(x_2) \biggr \}, \quad x_1 \in \sfX_1.
\end{align}
In particular, $(\varphi_1^{\sfH\sfH}, \varphi_1^{\sfH}) \in \Phi_\sfH$, $\varphi_1^{\sfH \sfH} \ge \varphi_1$, $\varphi_1^{\sfH} \ge \varphi_2$. Finally, if $\sfd_i$ are distances metrizing $\sfX_i$, $i=1,2$, then $\varphi_1^{\sfH\sfH}$ is $\sfd_1$-uniformly continuous and $\varphi_1^{\sfH}$ is $\sfd_2$-uniformly continuous, both with the same (uniform) $\sfd_1 \otimes_{\f{C}} \sfd_2$-modulus of continuity of $\sfH$ on $\pcr{R}$ (cf.~\eqref{ss22:eq:coner}).
\end{proposition}
\begin{proof}
Let $(\varphi_1, \varphi_2)$ be as in the statement. By Lemma \ref{ss22:le:staccato} we know that there exists $\eps>0$ (not depending on the pair) such that
\[\varphi_1(x_1) \le \sfH_1(x_1) - \eps \quad \text{ for every } x_1 \in \sfX_1.\]
Then, by uniform continuity of $\sfH$ on $\pcr{1}$, we can find $0<\delta<1$ such that 
\[ \left | \sfH([x_1, 1], [x_2, r_2])-\sfH_1(x_1) \right | \le \frac{\eps}{2} \quad \text{ for every } 0 \le r_2 \le \delta.\]
If we define 
$R:= 1+ \frac{1}{\delta} + \frac{2}{\eps} \left (\max_{\sfX_2} \sfH_2 + \max_{\sfX_1} \sfH_1+ 1 \right)$, 
then, for every $\alpha > R$, we have
\begin{align*}
    \sfH([x_1, \alpha], [x_2, 1]) - \alpha \varphi_1(x_1) &= \sfH([x_1, \alpha], [x_2, 1]) - \sfH_1(x_1) \alpha + \alpha \left (\sfH_1(x_1)-\varphi_1(x_1) \right ) \\
    &= \alpha \left ( \sfH([x_1, 1], [x_2, 1/\alpha])-\sfH_1(x_1) \right ) + \alpha \left (\sfH_1(x_1)-\varphi_1(x_1) \right ) \\
    &\ge \alpha \frac{\eps}{2} \\
    & \ge \sfH_2(x_2)+1.
\end{align*}
Thus, for every $x_2 \in \sfX_2$, we get
\begin{align*}
 \inf_{x_1 \in \sfX_1} \inf_{\alpha > R} \left \{ \sfH([x_1, \alpha], [x_2, 1]) - \alpha \varphi_1(x_1) \right \} &\ge \sfH_2(x_2)+1 \\
 &> \inf_{x_1 \in \sfX_1} \inf_{0 \le \alpha \le  R} \left \{ \sfH([x_1, \alpha], [x_2, 1]) - \alpha \varphi_1(x_1) \right \}
\end{align*}
and this proves \eqref{ss22:eq:varphihr1}. The proof of \eqref{ss22:eq:varphihr2} is analogous.\\
The fact that $\varphi_1^\sfH \ge \varphi_2$, $\varphi_1^{\sfH \sfH} \ge \varphi_1$ and 
\[ \varphi_1^{\sfH \sfH}(x_1) r_1 + \varphi_1^{\sfH}(x_2) r_2 \le \sfH([x_1, r_1], [x_2, r_2]) \quad \text{ for every } (x_1, x_2) \in \sfX_1 \times \sfX_2, \, r_1, r_2 \ge 0\]
follows by the definition of $\varphi_1^\sfH$ and $\varphi_1^{\sfH \sfH}$. It is then clear that $\varphi_1^\sfH$ (resp.~$\varphi_1^{\sfH \sfH}$) is bounded from below by $\min_{x_2 \in \sfX_2} \varphi_2$ (resp.~$\min_{x_1 \in \sfX_1} \varphi_1)$ and by above by $\max_{x_2 \in \sfX_2}\sfH_2$ (resp.~$\max_{x_1 \in \sfX_2}\sfH_1$). Let $\sfd_i$ be distances metrizing $\sfX_i$, $i=1,2$; let now $x_2, x_2' \in \sfX_2$, then, recalling \eqref{ss22:eq:prcdist}, we have
\begin{equation} \label{ss22:eq:uniformcont}
\begin{split}
    \left |\varphi_1^{\sfH}(x_2)- \varphi_1^{\sfH}(x_2') \right | &\le \sup_{x_1 \in \sfX_1} \sup_{0 \le \alpha \le R} \left | \sfH([x_1, \alpha], [x_2, 1])- \sfH([x_1, \alpha], [x_2', 1]) \right |\\
    & \le \omega_\sfH^R \left ( (\sfd_1 \otimes_{\f{C}} \sfd_2) \left (([x_1, \alpha], [x_2, 1]), ([x_1, \alpha], [x_2', 1]) \right ) \right ) \\
    &= \omega_\sfH^R ( \sfd_{2, \f{C}}([x_2,1], [x_2',1]))\\
    &\le \omega_\sfH^R(\sfd_2(x_2, x_2')), 
\end{split}
\end{equation}
where $\omega_\sfH^R$ is the (uniform) modulus of continuity of $\sfH$ on $\pcr{R}$ and we have used that $\sfd_{2, \f{C}}([x_2,1], [x_2',1])) \le \sfd_2(x_2, x_2')$ (see \eqref{ss22:eq:distcone}, \eqref{ss22:eq:prcdist} and formula (7.5) in \cite{LMS18}). The analogous statement for $\varphi_1^{\sfH \sfH}$ follows by the same strategy. This proves that $\varphi_1^\sfH$ and $\varphi_1^{\sfH \sfH}$ are uniformly continuous with the same (uniform) modulus of continuity of $\sfH$ on $\pcr{R}$ and concludes the proof that $(\varphi_1^{\sfH\sfH}, \varphi_1^{\sfH}) \in \Phi_\sfH$. Let $\mathcal{U}=\{ x_1^n, x_2^n, r_n \}_{n=1}^N$ be as in Lemma \ref{le:balls} for the distances $\sfd_i$ and $\Omega=\Omega_\sfH$; we define (recalling Lemma \ref{ss22:le:balls})
\[M:= \frac{\kappa_1+\kappa_2+1}{m(\mathcal{U}, \mu_1, \mu_2)} +\omega_\sfH^R(\diam{\sfX_1}) + \omega_\sfH^R(\diam{\sfX_2}) + \max_{x_1 \in \sfX_1} \sfH_1 + \max_{x_2 \in \sfX_2}\sfH_2,\]
we have that $\varphi_1^\sfH \le \sfH_2 \le M$ and, by \eqref{ss22:eq:uniformcont}, we get
\[ \varphi_1^\sfH(x_2) \ge \varphi_1^\sfH (x_2')- \omega_\sfH^R(\sfd_2(x_2, x_2')) \ge - \frac{\kappa_1+\kappa_2+1}{m(\mathcal{U}, \mu_1, \mu_2)} - \omega_\sfH^R(\diam{\sfX_2}) \ge -M \quad \text{ for every } x_2 \in \sfX_2,
\]
where $x_2' \in \sfX_2$ is some point where $\varphi_1^\sfH$ is larger than $-\frac{\kappa_1+\kappa_2+1}{m(\mathcal{U}, \mu_1, \mu_2)}$ (whose existence is given by Lemma \ref{ss22:le:balls}). The proof for $\varphi_1^{\sfH \sfH}$ is the same.
\end{proof}

With the result of Proposition \ref{ss22:prop:propertiesphih} it is straightforward to obtain the existence of a maximizing pair.
\begin{theorem}[Existence of optimal continuous potentials] \label{ss22:theo:potexist} Assume that $\sfH$ is as in \eqref{ass:1} and that $\mu_i \in \meas_+(\sfX_i)$ are such that $\supp{\mu_i} = \sfX_i$. Then there exists $(\varphi_1, \varphi_2) \in \Phi_\sfH$ such that
\[ \int_{\sfX_1} \varphi_1 \de \mu_1 + \int_{\sfX_2} \varphi_2 \de \mu_2 = \m_\sfH(\mu_1, \mu_2). \]
\end{theorem}
\begin{proof} If $\m_\sfH(\mu_1, \mu_2)=0$, we can take $\varphi_1$ and $\varphi_2$ to be the null functions. We thus assume that $\m_\sfH(\mu_1, \mu_2) >0$. If this is the case, we can find a maximizing sequence $(\varphi_1^j, \varphi_2^j)_j \subset \Phi_{\sfH}$ for the dual problem \eqref{ss22:eq:dualformulation} with 
\[ \int_{\sfX_1} \varphi_1^j \de \mu_1 + \int_{\sfX_2} \varphi_2^j \de \mu_2 \ge 0 \quad \text{ for every } j \in \N.\]
If we consider distances $\sfd_i$ metrizing $\sfX_i$, by Proposition \ref{ss22:prop:propertiesphih} we have that $(\varphi_1^{j,\sfH \sfH}, \varphi_1^{j, \sfH})_j \subset \Phi_\sfH$ is a maximizing sequence of equi-uniformly continuous and equi-bounded functions. By Arzelà–Ascoli theorem, we can assume, up to passing to a subsequence, that there exists a pair $(\varphi_1, \varphi_2) \in \Phi_\sfH$ such that $(\varphi_1^{j,\sfH \sfH}, \varphi_1^{j,\sfH}) \to (\varphi_1, \varphi_2)$ uniformly on the compact space $\sfX_1 \times \sfX_2$. By dominated convergence, we have
\[ \int_{\sfX_1} \varphi_1 \de \mu_1 + \int_{\sfX_2} \varphi_2 \de
  \mu_2 = \lim_j \left ( \int_{\sfX_1} \varphi_1^{j, \sfH \sfH} \de
    \mu_1 + \int_{\sfX_2} \varphi_1^{j, \sfH} \de \mu_2 \right ) =
  \m_\sfH(\mu_1, \mu_2).
\qedhere\]
\end{proof}

\section{Optimality conditions}\label{sec:6}
In this section we provide sufficient and necessary conditions for a plan $\aalpha \in \f{H}^1(\mu_1, \mu_2)$ to be optimal. In the following $\sfX_1$ and $\sfX_2$ are completely regular spaces and we will often assume that
\begin{equation}\label{ass:3}
\sfH: \pc \to [0, + \infty] \text{ is a proper, radially $1$-homogeneous, convex and l.s.c.~function}.
\end{equation}

The cyclical monotonicity of the support of an admissible plan plays a crucial role also in the unbalanced setting. 
We recall here the important definition of 
cyclical monotonicity of a given set with respect to 
a cost $\sfH$. 
\begin{definition}[$\sfH$-cyclical monotonicity]
\label{def:cmon}
Let $\Gamma \subset \pc$ and let $\sfH: \pc \to [0,+\infty]$; we say that $\Gamma$ is $\sfH$-cyclically monotone if for every finite family of points $\{(\f{y}_1^i, \f{y}_2^i)\}_{i=1}^N \subset \Gamma$  and every permutation $\sigma$ of $\{1, \dots, N\}$ it holds
\[ \sum_{i=1}^N \sfH(\f{y}_1^i, \f{y}_2^i) \le \sum_{i=1}^N \sfH(\f{y}_1^i, \f{y}_2^{\sigma(i)}).\]
\end{definition}
\noindent
The following result shows that optimal $1$-homogeneous couplings
are concentrated on a \emph{$\sfH$-cyclically monotone} set $\Gamma$
which is also \emph{a radial convex cone}.
It extends 
to the unbalanced optimal transport setting
the classical \cite[Necessity part of Theorem 6.14]{AGS08}.
\begin{proposition}[Necessity of cyclical
  monotonicity]\label{prop:necee} Let $\sfH$ be as in \eqref{ass:3},
  let $\mu_i \in \meas_+(\sfX_i)$ for $i=1,2$, let $\aalpha \in
  \f{H}^1_o(\mu_1, \mu_2)$ be optimal and suppose that $\int_{\pc}
  \sfH \de \aalpha < + \infty$. Then $\aalpha$ is concentrated on a $\sigma$-compact (thus Borel) 
  radial convex cone $\Gamma \subset \pc$ which is $\sfH$-cyclically
  monotone.
\end{proposition}
\begin{proof} Let $\{ (\varphi_1^k, \varphi_2^k) \}_{k \ge 1} \subset \Phi_\sfH$ be a maximizing sequence for the dual problem \eqref{ss22:eq:dualformulation} and let us define
\[ \sfH_k([x_1, r_1], [x_2, r_2]) := \sfH([x_1, r_1], [x_2, r_2]) - \varphi_1^k(x_1)r_1 - \varphi_2^k(x_2)r_2, \quad ([x_1, r_1], [x_2, r_2 ]) \in \pc.\] 
Since 
\[\lim_{k\to\infty}\int \sfH_k([x_1, r_1], [x_2, r_2])\,\de\aalpha=0,\]
there exist a subsequence $m \mapsto k(m)$ and a $\sigma$-compact (thus Borel) subset $G \subset \pc$ on which $\aalpha$ is concentrated
s.t. $\sfH_{k(m)} \to 0$ on $G$ as $m \to + \infty$. Since
$\sfH_{k(m)}$ is a radially convex and $1$-homogeneous function,
this convergence takes place also on the
radial convex cone $\Gamma:=\hmg{\ce G}$ generated by $G$:
\[  \Gamma   :=
  \left \{ \Big(\big [x_1, \sum_{i=0}^2\lambda_ir_1^i], [x_2,
    \sum_{i=0}^2\lambda_ir_2^i]) \mid ([x_1,r_1^i],[x_2, r_2^i]) \in \Gamma,\
    \lambda_i\ge0\Big)
  \right \}.\]
 Writing $G$ as a countable union of an increasing sequence
of compact sets $K_n\subset \pc$,
it is easy to see that $ \Gamma =\cup_{n\in \N} \hat K_n$, where
$\hat K_n:=\hmg{\ce{K_n}}$ and it is not restrictive to assume that
$(\f o_1, \f o_2)\in K_n$ . Since each $\ce{K_n}$ is clearly compact in $\pc$
and $\hmg{\ce{K_n}}=\cup_{m\in \N}\big\{(m \f y_1,m \f y_2): (\f y_1,\f
y_2)\in \ce{K_n}\big\} $ is the union of a countable family of compact
sets,
we conclude that $\Gamma$ is $\sigma$-compact as well. 
Let now $\{(\f{y}_1^i, \f{y}_2^i)\}_{i=1}^N \subset \Gamma$ be a
finite family of points and let $\sigma$ be a permutation of $\{1,
\dots, N\}$.
Then
\begin{align*}
  \sum_{i=1}^N \sfH(\f{y}_1^i, \f{y}_2^{\sigma(i)})
  &\ge \sum_{i=1}^N
    \left (
    r_1^i
    \varphi_1^{k(m)}(x_1^i)
    +
    r_2^{\sigma(i)}
    \varphi_2^{k(m)}(x_2^{\sigma(i)})
    \right )
  \\
  &= \sum_{i=1}^N \left (
    r_1^i \varphi_1^{k(m)}(x_1^i)
    +
    r_2^{i} \varphi_2^{k(m)}(x_2^{i})  \right )
  \\ &= \sum_{i=1}^N\left ( \sfH(\f{y}_1^i, \f{y}_2^i) -
       \sfH_{k(m)}(\f{y}_1^i, \f{y}_2^i) \right ).
\end{align*}
Letting $m \to + \infty$, we obtain the sought $\sfH$-cyclical
monotonicity of $\Gamma$.
\end{proof}

\begin{remark}
   The previous proof shows that
  the radial convex cone $\hmg{\ce G}$  generated by a $\sigma$-compact set
  $G\subset \pc$
  is $\sigma$-compact.
  In particular, for a finite Radon measure $\aalpha\in \meas_+(\pc)$
  the following properties are equivalent:
  \begin{enumerate}[(i)]
  \item $\aalpha$ is concentrated on a Borel set $G$ such that
    the generated radial convex cone  $\hmg{\ce G}$  is $\sfH$-cyclically
    monotone;
  \item $\aalpha$ is concentrated on a $\sigma$-compact radial convex
    cone $\Gamma$ which is $\sfH$-cyclically
    monotone.
  \end{enumerate}
\end{remark}

We devote the remaining part of this section to formulate a converse
statement to Proposition \ref{prop:necee}.
We first introduce a few notions
related to the natural directed graph structures
induced by $\sfH$ and subsets $\Gamma$ of
$\pc$. A similar approach has already been considered when dealing with optimality conditions for possibly infinite costs in the classical Optimal Transport theory \cite{beiblo,biancara}.

\subsection{Simple directed graphs and oriented walks}
\label{subs:graph}

Recall that a directed graph $\mathcal G$ in a set $\mathsf Z$ is an ordered
pair $(V,\mathcal A)$ where $V\subset \mathsf Z$ and $\mathcal A$ is a set of ordered pairs in
$V\times V$.
A (oriented) $\mathcal A$-walk $P$ in $V$
is just a sequence of
elements $(\f y^0,\cdots,\f y^N)$ in $V^{N+1}$, $N\in \N_+$, such that
each pair of consecutive elements $(\f y^{h-1},\f y^{h})$ belongs to
$\mathcal A$, $h=1,\cdots, N$.
We denote by $\mathcal P(\f y'\to \f y''|V,\mathcal A)$ the
collections of $\mathcal A$-walks in $V$ whose first 
and last elements are $\f y'$ and $\f y''$ respectively.
We will omit to write $V$ when $V=\mathsf Z$ and we will also omit to
write $\mathcal A$ if the set of arcs is clear from the context.
When $\f y'=\f y''$ then we say that $\mathcal P$ is a cycle.

If $P'\in \mathcal P(\fy^1\to\fy^2)$ and
$P''\in \mathcal P(\fy^2\to\fy^3)$ 
we can construct a new walk $P=P'+P''
\in \mathcal P(\fy^1\to\fy^3)$ by joining $P'$ with $P''$.

We say that a (ordered) pair of points $(\f y',\f y'')\in \mathsf Z$
is $\mathcal A$-connected if $\mathcal P(\f y'\to\f y'')$ is not empty. 
A directed graph $(V,\mathcal A)$ is connected if every pair of
points in $V$ is $\mathcal A$-connected.

Given a function $ \sfH:\pc\to [0,+\infty]$
and a set
$\Gamma\subset \rmD(\sfH)\subset \pc$,
we construct a directed (bipartite) graph $\mathcal G_{\sfH,\Gamma}$
whose vertices belong to 
\[
  \mathsf Z=\mathsf Z[\sfX_1,\sfX_2]:=\f
  C[\sfX_1]\sqcup \f C[\sfX_2],
  \quad\text{the disjoint union of
$\f
C[\sfX_1]$ and $\f C[\sfX_2]$.}
\]

We first consider the set of
arcs $\mathcal A_\sfH$ consisting of all
the (ordered) pairs $(\f y_2,\f y_1)\in \mathsf Z \times \mathsf Z  $ such that $\f y_i\in \f C[\sfX_i]$ and
$\sfH(\f y_1,\f y_2)<+\infty$.
We can also easily identify $\Gamma$ with a set of arcs, i.e.~all
the (ordered) pairs $(\f y_1,\f y_2)\in \mathsf Z \times \mathsf Z  $ such that
$(\f y_1,\f y_2)\in \Gamma$
(notice that $\Gamma$ can be canonically identified
with a subset of $\mathsf Z\times \mathsf Z$).

We eventually set
\[
  \mathcal A_{\sfH,\Gamma}:=\mathcal A_\sfH\cup \Gamma.
\]
When $\sfH$ is the null function $0$ (or any finite function)
we have $\mathcal A_0=\f C[\sfX_2]\times \f C[\sfX_1]$ and
\[
  \mathcal A_{0,\Gamma}=\big(\f C[\sfX_2]\times \f C[\sfX_1]\big)\cup \Gamma.
\]
Recalling that $\pi^i:\pc\to\f C[\sfX_i]$ are the canonical
projections, we also consider the subsets
\begin{equation}
    \label{eq:projectionG}
  \Gamma_i:=\pi^i(\Gamma)\subset \f C[\sfX_i],\quad
  V_\Gamma :=\piGamma1\sqcup
  \piGamma2\subset \mathsf Z;
\end{equation}
$V_\Gamma$ is the collection of all the
vertices obtained by applying to $\Gamma$
the two projections on
$\f C[\sfX_i]$, $i=1,2$, canonically identified with the corresponding 
subsets of $\mathsf Z$. We set
\[ \arc_{\sfH, \Gamma} := \mathcal A_{\sfH,\Gamma} \cap (V_\Gamma \times V_\Gamma).\]

Notice that if the initial and final points $\f y^0, \f y^{N}$ of a
$\mathcal A_{\sfH,\Gamma}$-walk $P$ belong to
$\f C[\sfX_1]$ we will have 
\begin{equation}
  \label{eq:16}
  \begin{gathered}
    N=2n\text{ even},\quad
    \f y_1^k:=\f y^{2k}\in \f C[\sfX_1]\quad\text{for $0\le k\le n$,}\quad
    \f y_2^k:=\f y^{2k+1}\in \f C[\sfX_2]\quad\text{for $0\le k<n$},\\
    (\f y^{2k},\f y^{2k+1})=(\f y_1^k,\f y_2^k)\in \Gamma,\quad
    \sfH(\f y^{2k+2},\f y^{2k+1})= \sfH(\f y_1^{k+1},\f
    y_2^k)<+\infty,\\
    \fy^{2k}=\fy_1^k\in \piGamma1,\  \fy^{2k+1}=\fy_2^k\in
    \piGamma2\quad\text{for every }k=0,\cdots,n-1.
  \end{gathered}
\end{equation}
In particular, if also $\fy^N$ belongs to $\piGamma1$ then $P\in
\mathcal P(\fy^0\to\fy^N|V_\Gamma,\mathcal \arc_{\sfH,\Gamma})$.  More
generally,
if the initial and final points of a $\mathcal A_{\sfH,\Gamma}$-walk
$P$ belong to $V_\Gamma$ then all the points of $P$ belong to $V_\Gamma$.

\begin{definition}[$\sfH$-connectedness]
  We say that a set $\Gamma\subset \rmD(\sfH)$ is $\sfH$-connected if
  the graph $(V_\Gamma,\mathcal \arc_{\sfH,\Gamma})$ 
  is connected or,
  equivalently, if every pair of points of $\Gamma_1$ is
  $\mathcal A_{\sfH,\Gamma}$-connected.
\end{definition}

 Let us now define the
``oriented'' cost function
$\check\sfH:\mathsf Z\times \mathsf Z\to
\bar \R$ starting from a cost function $\sfH$ as in \eqref{ass:3}: 
\[
  \check\sfH(\f y',\f y''):=
  \begin{cases}
    -\sfH(\f y',\f y'')&\text{if }\f y'\in \f C[\sfX_1],\ \f y''\in \f
    C[\sfX_2],\\
    +\sfH(\f y'',\f y')&\text{if }\f y''\in \f C[\sfX_1],\ \f y'\in \f
    C[\sfX_2],\\
    0&\text{otherwise}.
  \end{cases}
\]
The cost $\Theta(P)$ of a walk $P=(\f y^0,\cdots,\f y^N)$ in $(\mathsf
Z,\mathcal A_{0,\Gamma})$ with $\Gamma \subset \rmD(\sfH)$  is defined by
\[
  \Theta(P):=\sum_{h=1}^N \check\sfH(\f y^{h-1},\f y^h);
\]
notice that $\Theta$ is well defined and takes values in 
$\R\cup\{+\infty\}$ since $\Gamma\subset \rmD(\sfH)$, so that 
the negative contributions of the arcs to $\Theta(P)$ are always
finite. It is also easy to check that 
\begin{equation}
  \label{eq:33}
  \text{$\Theta(P)$ is finite if and only if $P$ is a $\mathcal A_{\sfH,\Gamma} $-walk}.
\end{equation}
In the case of \eqref{eq:16} we can equivalently write
\[
  \Theta(P)= \sum_{ k=0}^{ n-1} \Big(\sfH(\f y_1^{k+1},\f y_2^k)-\sfH(\f y_1^{k},\f y_2^k)\Big).
\]
 Clearly, $\Gamma \subset \rmD(\sfH)$ is $\sfH$-cyclically monotone according to Definition \ref{def:cmon} if and only if
\begin{equation}
\label{eq:cyclic-standard}
\text{for every cycle $P\in \mathcal P(\f y\to \f y)$ in $(\mathsf Z,\mathcal A_{0,\Gamma})$
  we have }\quad
\Theta(P)\ge 0.
\end{equation}
Recalling \eqref{eq:33} it is immediate to see that it is sufficient
to check
condition \eqref{eq:cyclic-standard} on $\mathcal
A_{\sfH,\Gamma}$-cycles in $V_\Gamma$.
 We are going to use the following notation: if   $P=(\fy^0,\cdots, \fy^N)$  and  
$A=(\fy^k,\fy^{k+1})$ is an internal arc for some $0< k<N-1$, we can ``remove'' the arc
$A$ from $P$ obtaining a new walk $P^{-k}$ 
by setting $P^{-k}:=(\fy^0, \cdots, \fy^{k-1}, \fy^{k+1}, \fy^N)$.

\begin{lemma}[The effect of removing the vertex $(\fo_1,\fo_2)$]
  \label{le:remove}

  Let $\sfH$ be satisfying the standard assumption
   \eqref{ass:3}
  and let $\Gamma\subset \rmD(\sfH)$ be a radial cone.
  
  If $P$ is a $\mathcal A_{\sfH,\Gamma}$-walk 
  and $P_\fo$ is the walk obtained removing from $P$ all the internal arcs of the form $(\fo_1,\fo_2)$, then 
  \begin{equation}
    \label{eq:34}
    P_\fo\text{ is a $\mathcal A_{\sfH,\Gamma}$-walk and}\quad 
    \Theta(P_\fo)\le \Theta(P).
  \end{equation}
\end{lemma}
\begin{proof}
  Let $P=(\fy^0,\cdots, \fy^N)$ be a $\mathcal A_{\sfH,\Gamma}$-walk.
  The proof follows by a simple induction argument if we show that 
  $\Theta$ decreases if we remove an internal arc of the form
  $(\fo_1,\fo_2)$. We can thus assume that 
  $\fy^k=\fo_1,\fy^{k+1}=\fo_2$ for some $0< k<N$
    and we set 
    $P' :=P^{-k}$.
  We notice that $\fy^{k-1}=[x^{k-1},r^{k-1}]\in \f C[\sfX_2]$ and
  $\fy^{k+2}=
  [x^{k+2},r^{k+2}]\in \f
  C[\sfX_1]$ so that 
  \begin{align*}
    \Theta(P)-\Theta(P')
    &=\check\sfH(\fy^{k-1},\fo_1)+\check\sfH(\fo_2,\fy^{k+2})-
      \check\sfH(\fy^{k-1},\fy^{k+2})
      \\&=
    \sfH(\fo_1, \fy^{k-1})+\sfH(\fy^{k+2}, \fo_2)-
    \sfH(\fy^{k+2},\fy^{k-1})
    \\&=
    \sfH_{x^{k+2},x^{k-1}}(0,r^{k-1})+
    \sfH_{x^{k+2},x^{k-1}}(r^{k+2},0)-
    \sfH_{x^{k+2},x^{k-1}}(r^{k+2},r^{k-1})
    \ge0
  \end{align*}
  thanks to the subadditivity of $\sfH_{x^{k+2},x^{k-1}}$.
  The same argument shows that $(\fy^{k-1},\fy^{k+2})\in \mathcal
  A_{\sfH}$ so that $P'$ is an $\mathcal A_{\sfH,\Gamma}$-walk. 

\end{proof}

\begin{corollary}
  Under the same assumptions of Lemma \ref{le:remove},
  if $\Gamma$ is $\sfH$-connected then
  also $\Gamma_\fo:=\Gamma\setminus\{(\fo_1,\fo_2)\}$
  is $\sfH$-connected and if $\Gamma_\fo$ is
  $\sfH$-cyclically monotone then also $\Gamma$ is $\sfH$-cyclically
  monotone (i.e.~the arc $(\f o_1,\f o_2)$ is irrelevant for $\sfH$-connectedness
  and $\sfH$ cyclical monotonicity).
  \end{corollary}
  \begin{proof}
    Let $\Gamma_{\fo,i}:=\pi^i(\Gamma_\fo)$
    and let $\fy',\fy''\in \Gamma_{\fo,1}.$
    If both are different from
    $\fo_1$, since $\Gamma$ is $\sfH$-connected and $\fy', \fy'' \in \Gamma_1$, we can find a $\mathcal A_{\sfH,\Gamma}$-walk $P\in
    \mathcal P(\fy'\to\fy'')$ and, by the previous Lemma \ref{le:remove} we can remove
    from $P$ all the internal arcs of the form $(\fo_1,\fo_2)$ 
    obtaining
    a $\mathcal A_{\sfH,\Gamma_\fo}$-walk $P_\fo\in \mathcal
    P(\fy'\to\fy'')$.

    If $\fy'=\fo_1\neq \fy''$, since $\fo_1\in \pi^1(\Gamma_\fo)$, we
    can find
    $\fy_2'\in \f C[\sfX_2]\setminus \{\fo_2\}$ such that
    $(\fo_1,\fy_2')\in \Gamma$
    and, by the same argument above, a $\mathcal
    A_{\sfH,\Gamma_\fo}$-walk $P'\in \mathcal P(\fy_2'\to \fy'')$. It
    follows that $P:=(\fo_1,\fy_2')+P'$ is a $\mathcal
    A_{\sfH,\Gamma_\fo}$-walk connecting $\fo_1$ to $\fy''$.

    If $\fy'\neq \fo_1=\fy''$, 
    we can find $\fy_2''\in \f C[\sfX_2]\setminus \{\fo_2\}$
    such that
    $(\fo_1,\fy_2'')\in \Gamma\subset \rmD(\sfH)$, so that
    in particular $\sfH(\fo_1,\fy_2'')<\infty$ and $(\fy_2'',\fo_1)\in
    \mathcal A_\sfH$. By the same argument above, we can find $P''\in \mathcal P(\fy'\to \fy_2'')$ which is a $\mathcal
    A_{\sfH,\Gamma_\fo}$-walk
     so that   $P:=P''+(\fy_2'',\fo_1)\in \mathcal P(\fy'\to\fo_1)$
    is a $\mathcal
    A_{\sfH,\Gamma_\fo}$-walk as well.

    If $\fy'=\fy''=\fo_1$, the two arguments above show that we can find $\fy'_2, \fy''_2 \in \f C[\sfX_2]\setminus \{ \fo_2\}$ such that $(\fo_1, \fy_2') \in \Gamma$ and $(\fy_2'', \fo_1) \in \mathcal A_\sfH$. Since both $\fy_2', \fy_2'' \in \Gamma_2$ are different from $\fo_2$, we can apply again Lemma \ref{le:remove} and obtain a $\mathcal{A}_{\sfH, \Gamma_\fo}$-walk $P''' \in \mathcal{P}(\fy_2'\to \fy_2'')$. We can thus define $P:=(\fo_1, \fy_2') + P''' + (\fy_2'', \fo_1) \in \mathcal{P}(\fy'\to\fy'')$ which is a $\mathcal{A}_{\sfH, \Gamma_\fo}$-walk.  
    
    Let us now suppose that $\Gamma_\fo$ is $\sfH$-cyclically monotone 
    and let $P$ be a $\mathsf A_{\sfH,\Gamma}$-cycle in $V_\Gamma$.
    It is not restrictive to suppose that $P$ contains at least an
    element
    $\fy\in V_\Gamma\setminus\{\fo_1,\fo_2\}$ (otherwise
    $\Theta(P)=0$). By a cyclic permutation of the arcs of $P$ we
    can assume that the initial and final point of $P$ is $\fy$. 
    By Lemma \ref{le:remove} we can construct a cycle $P_\fo$ in
    $(V_{\Gamma_\fo},\mathcal A_{\sfH,\Gamma_\fo})$ obtained by removing all
    the arcs $(\fo_1,\fo_2)$ from $P$ with the same initial and final
    point. Since $\Gamma_\fo$ is $\sfH$-cyclically monotone, we 
    have $\Theta(P_\fo)\ge0$. The inequality \eqref{eq:34} shows that
    $\Theta(P)\ge0$ as well.   
  \end{proof}

The next Proposition shows how $\sfH$-cyclical monotonicity reflects on walks that are not precisely a cycle (cf.~\eqref{eq:cyclic-standard}) but are of the form $P\in \mathcal P(\f y\to\lambda\f y)$  for some $\lambda>0$.

Given $(\f y_1, \f y_2)=([x_1, r_1], [x_2, r_2])\in \rmD(\sfH)$,  we denote by $\partial \sfH(\f y_1,\f y_2)$ the subdifferential of $\sfH_{x_1, x_2}$ 
at the point $(r_1, r_2) \in \R_+^2$, defined as
\[ 
\partial \sfH(\f y_1,\f y_2) := \left\{ (a,b) \in \R^2\ :\vert \begin{array}{l}
    \sfH_{x_1, x_2}(s_1, s_2) - \sfH_{x_1, x_2}(r_1, r_2) \ge a(s_1-r_1)+ b(s_2-r_2) \\
    \text{ for every } (s_1, s_2) \in \R_+^2
  \end{array}\right\}.
  \]
Notice that 
$\partial\sfH(\lambda \f y_1,\lambda \f y_2)=
\partial\sfH(\f y_1,\f y_2)$ for every $\lambda>0$. 
As usual, the 
proper domain of the subdifferential is denoted by
$ \rmD(\partial \sfH) := \left \{ (\f y_1, \f y_2) \in \f C[X_1,X_2]
\mid \partial \sfH(\f y_1,\f y_2)\ne \emptyset \right \}$.
It is well known that 
$\rmD(\partial \sfH)$ 
contains 
the \emph{radial interior}
\[
\rintt{\rmD(\sfH)}:=
\Big\{(\f y_1,\f y_2)\in 
\f C[X_1,X_2]:
(r_1,r_2)\in \intt{\rmD(\sfH_{x_1,x_2})}\Big\},
\]
where the interior of $\rmD(\sfH_{x_1,x_2})$
refers to the usual topology of $\R^2$; 
in particular, if $(\f y_1, \f y_2) \in \rintt{\rmD(\sfH)}$, then $r_1>0, r_2>0$. 

 For every $(\f y_1,\f y_2)\in \rintt{\rmD(\sfH)}$ we can consider the quantity 
\begin{equation}
    \label{eq:support}
   \mathsf a(\f y_1,\f y_2):=
   \sup \Big\{|a|:
   (a,b)\in \partial\sfH(\f y_1,\f y_2)\Big\};
 \end{equation}
 notice that
 \begin{align*}
   \mathsf a(\lambda \f y_1,\lambda \f y_2)&=
   \mathsf a(\f y_1,\f y_2)&&\quad\text{for every }\lambda>0,\\   
   -\mathsf a(\f y_1,\f y_2) |b| \le ab&\le \mathsf a(\f y_1,\f y_2) |b|&&\quad
   \text{for every }(a,b)\in \partial\sfH(\f y_1,\f y_2).
 \end{align*}
 If $\fy_2=\fo_2$ and $(\fy_1,\fo_2)\in \rmD(\sfH)$ we also set
 \[
   \mathsf a(\fy_1,\fo_2):=\sfH([x_1,1],\fo_2)=
   r_1^{-1}\sfH(\fy_1,\fo_2)\quad
   \text{where }
   \fy_1=[x_1,r_1].
 \]
 
\begin{proposition}\label{ss22:prop:Subdiff} Let $\sfH$ be as in
  \eqref{ass:3}, let $\Gamma \subset  \rmD(\sfH)$
  be
  a 
 $\sfH$-cyclically monotone radial cone, and
  let $(\fy_1,\fy_2)=([x_1,r_1],[x_2,r_2])\in\Gamma_\fo= \Gamma \setminus \{(\fo_1, \fo_2)\}$,
  such that
  $(\fy_1,\fy_2)\in \rintt{\rmD(\sfH)}$
  (and therefore $r_i>0$)
  or $\fy_2=\fo_2$ (and therefore $r_1>0$).  Let $\f y_1'=[x_1,r_1'] \in \f C[\sfX_1]$ with $r_1'>0$; then:
  \begin{enumerate}
  \item
     for every $\eps>0$ 
    there exists a $\mathcal A_{\sfH,\Gamma_\fo}$-walk $P_\eps\in \mathcal
    P(\f y_1'\to\f y_1)$
    with
    \begin{equation}
      \label{eq:21}
      \Theta(P_\eps)\le \mathsf a(\f y_1,\f y_2)|r_1'-r_1|+\eps;
    \end{equation}
  \item if $P\in \mathcal P(\f y_1\to \f y_1')$ is any $\mathcal A_{\sfH,\Gamma}$-walk, 
    then
    \begin{equation}
      \label{eq:22}
      \Theta(P)\ge
       -\mathsf a(\f y_1,\f y_2)|r_1'-r_1|.
    \end{equation}
  \end{enumerate}

\end{proposition}
\begin{proof}
  Notice that $\fy_1,\fy_1'\in \f C[\sfX_1] \setminus \{ \fo_1\}$;
  if $\fy_2\neq \fo_2$ we set 
$\bar{\sfH}:=\sfH_{x_1,x_2}$, $q:= \frac{r_2}{r_1} >0$; 

\smallskip\noindent
Claim (1).
The case $\fy_2=\fo_2$ is simple: since $(\fy_1,\fo_2)\in \Gamma
\subset \rmD(\sfH)$
we have $\mathsf a(\fy_1,\fo_2)= r_1^{-1} \sfH(\fy_1,\fo_2)<\infty$;
since $\Gamma$ is a radial cone and $(\fy_1',\fo_2)=r_1'/r_1
(\fy_1,\fo_2)$ we have $(\fy_1',\fo_2)\in \Gamma$ with
$\sfH(\fy_1',\fo_2)= r_1' \mathsf a(\fy_1,\fo_2)$.
The walk $P=(\fy_1',\fo_2,\fy_1)$ belongs
to $(V_{\Gamma_\fo},\arc_{\sfH,\Gamma_\fo})$
and
\begin{displaymath}
  \Theta(P)=\sfH(\fy_1,\fo_2)-\sfH(\fy_1',\fo_2)=
  (r_1-r_1') \mathsf a(\fy_1,\fy_2) \le \mathsf a(\fy_1,\fy_2)|r_1-r_1'|.  
\end{displaymath}

Let us now consider the case
$\fy_2\neq \fo_2$. For every $n \in \N$, we define 
 
$\vartheta_n:=\left (\frac{r_1}{r_1'} \right)^{1/n}$ 
 and we consider the points
\[ \f{y}_1^{k} := \left [ x_1, r_1'  \left ( {\vartheta_n} \right
    )^k  \right ], \quad \f{y}_2^{k} := \left [  x_2, r_1'q \left
      ( {\vartheta_n} \right )^{k}  \right ], \quad k=0, \dots,
  n-1,\quad
  \f{y}_1^{n} := \left [ x_1, r_1\right ]=\f y_1,
\]
inducing a walk $P_n$ connecting $\f y_1'$ to $\f y_1$ according to \eqref{eq:16},
since
\[ \f{y}_1^{0}= [x_1,
 r_1' ] = \f{y}_1', \quad (\f{y}_1^{k}, \f{y}_2^{k}) = (\rho_{n,k} [x_1, r_1], \rho_{n,k}[x_2, r_2]) \in {\Gamma}, \quad k=0, \dots, n-1,\]
where $\rho_{n,k} = \frac{r_1'}{r_1} \left ( {\vartheta_n} \right
)^{ k}$. We have

\begin{align*}
\Theta( P_n ) &= \sum_{k=0}^{n-1} \left ( \sfH(\f{y}_1^{k+1}, \f{y}_2^{k})
            - \sfH(\f{y}_1^{k}, \f{y}_2^{k}) \right )
  \\
          &= \sum_{k=0}^{n-1} \left ( \bar{\sfH}
            \left ( r_1' \left({\vartheta_n} \right )^{k+1}, 
            r_1' q \left({\vartheta_n} \right )^{k}\right ) -
            \bar{\sfH} \left ( r_1' \left ({\vartheta_n} \right )^{k}
         ,
            r_1'q \left ({\vartheta_n} \right )^{k}\right ) \right ) \\
&= \sum_{k=0}^{n-1}  r_1'  \left ({\vartheta_n} \right )^{k} \left ( \bar{\sfH}({\vartheta_n} ,q) - \bar{\sfH}\left (1, q \right ) \right ) \\
&\le \sum_{k=1}^n r_1'  \left ({\vartheta_n} \right )^{k} \left ( 1- {\vartheta_n} \right ) a_n 
 = r_1' a_n  \left ( 1 - \left ({\vartheta_n}\right)^n 
 \right ) 
=  a_n (r_1'-r_1)
\end{align*}
where 
$(a_n, b_n) \in \partial \bar{H} \left ( {\theta_n}, q \right ) \ne \emptyset$ 
for $n$ sufficiently large, since $\left ({\theta_n}, q \right ) \to (1, q)$ 
as $n \to + \infty$ and 
$\partial \bar{\sfH}(1,q)=\partial \bar{\sfH}(r_1,r_2)$ 
with $(r_1,r_2) \in \intt{\rmD(\partial \bar{\sfH})}
=
\intt{\rmD(\bar{\sfH})}$. This proves that $P_n$ is a walk in $\mathcal{A}_{\sfH, \Gamma_\fo}$ and, passing to the $\limsup$ as $n\to\infty$ we have
\[
  \limsup_{n\to\infty} a_n  (r_1'-r_1)\le \mathsf a(\f
  y_1,\f y_2)|r_1'-r_1|
\]
which yields \eqref{eq:21}.

\smallskip\noindent
Claim (2).
If $P\in \mathcal P(\fy_1 \to \fy_1')$
is any $\mathcal A_{\sfH,\Gamma}$-walk 
connecting
$\f y_1$ to $\f y_1'$, we can consider the cycle $\tilde P_\eps=P+P_\eps$
starting and ending in $\f y_1$
obtained by joining $P$ with the walk $P_\eps$ given by the previous
claim. \eqref{eq:cyclic-standard} and the $\sfH$-cyclycal monotonicity of $\Gamma$ yield
\begin{displaymath}
  0\le \Theta(\tilde P_\eps)=\Theta(P)+\Theta(P_\eps)
\end{displaymath}
which immediately gives \eqref{eq:22}.
\end{proof}

We can now give sufficient conditions for the $\sfH$-connectedness of
a set $\Gamma\subset \rmD(\sfH)$, with $\sfH: \pc \to [0,+\infty]$ as in \eqref{ass:3}. 
We set
\[
  \sfH_{\rm inf}(\f y_1,\f y_2):=\inf\Big\{\sfH(\lambda_1\f y_1,\lambda_2 \f y_2):\lambda_1,\lambda_2>0\Big\}, \quad (\f y_1, \fy_2) \in \f C[\sfX_1] \times \f C[\sfX_2]. 
\]
Notice that 
we have
\begin{equation}
  \label{eq:25}
  \text{if }\f y_i=[x_i,r_i]\neq \f o_{ i}\text{ then }
  \sfH_{\rm inf}(\f y_1,\f y_2)<\infty\quad\Leftrightarrow\quad
  \exists\,s_i>0:\sfH([x_1,s_1],[x_2,s_2])<\infty
\end{equation}
whereas
\begin{equation}
  \label{eq:26}
  \begin{gathered}
    \sfH_{\rm inf}(\f y_1,\f o_{ 2})<\infty\quad\Leftrightarrow\quad
    \sfH(\f y_1,\f o_{ 2})<\infty,\\
    \sfH_{\rm inf}(\f o_{ 1},\f y_2)<\infty\quad\Leftrightarrow\quad
    \sfH(\f o_{ 1},\f y_2)<\infty.
  \end{gathered}
\end{equation}
We can also easily see that
$\sfH_{\rm inf}(\f y_1,\f y_2)<\infty$ if and only
if it is possibile to connect $\fy_1$ (resp.~$\fy_2$)
to a multiple of $\fy_2$ (resp.~of $\fy_1$) with finite $\sfH$-cost:
\[
  \begin{aligned}
    \sfH_{\rm inf}(\f y_1,\f y_2)<\infty\quad&\Leftrightarrow\quad
    \exists\,\lambda_2>0:\sfH(\fy_1,\lambda_2\fy_2)<\infty\\
    &\Leftrightarrow\quad
    \exists\,\lambda_1>0:\sfH(\lambda_1\fy_1,\fy_2)<\infty.
  \end{aligned}
\]

\begin{theorem}[Sufficient conditions for
  $\sfH$-connectedness]
  \label{thm:connectedness}
  If $\sfH$ is as in \eqref{ass:3} and $\Gamma\subset \rmD(\sfH)$ is a $\sfH$-cyclically monotone radial cone satisfying at least
  one of the following conditions:
  \begin{enumerate}
  \item $\Gamma$ is $\sfH_{\rm inf}$-connected
    and
    $\Gamma\setminus \Big(\{\f o_{ 1}\}\times \f C[\sfX_2]\cup \f
    C[\sfX_1]\times \{\f o_{ 2}\}\Big)\subset \rintt{\rmD(\sfH)}$,
  \item $\sfH_{\rm inf}$ is finite on $\piGamma1\times
    \piGamma2$ 
     (recall \eqref{eq:projectionG}) 
    and
    \begin{equation}
      \label{eq:36}
      \Gamma\cap \rintt{\rmD(\sfH)}\neq \emptyset,
    \end{equation}
  \item $\sfH_{\rm inf}$ is finite on $\piGamma1\times
    \piGamma2$ and
    \begin{equation}
    \text{there exist points $\bar \fy_i\in \f C[\sfX_i]\setminus\{\fo_i\}$ such that
    $(\bar \fy_1,\fo_2)\in \Gamma$ and $(\fo_1,\bar \fy_2)\in
    \Gamma$},
  \label{eq:11} 
  \end{equation}
  \end{enumerate}
  then
    $\Gamma$ is $\sfH$-connected.
\end{theorem}
\begin{proof} We divide the proof in claims.

\medskip\noindent
  Claim (1). It is sufficient to prove that
  if a pair of points $(\f y_2,\f y_1)\in \piGamma2\times 
  \piGamma1$
  satisfies $\sfH_{\rm inf}(\f y_1,\f y_2)<+\infty$
  then we can find a $\mathcal A_{\sfH,\Gamma}$ walk $P\in \mathcal P(\f y_2\to \f y_1)$.
  If $\f y_1=\f o_1$ or $\f y_2=\f o_2$ we have $\sfH (\f y_1,\f
  y_2)<+\infty$ thanks to \eqref{eq:26},
  so that the arc $(\f y_2,\f y_1)$ belongs to $\mathcal A_{\sfH}$.

  We can thus suppose that $\f y_i=[x_i,r_i]$ with $r_i>0$.
  Recalling \eqref{eq:25} and using the
  $1$-homogeneity of $\sfH$ we can find $r_1'>0$ such
  that, setting $\f y_1'=[x_1,r_1']$, we get  $\sfH(\f y_1',\f y_2)<+\infty$. On the other hand, since $\f y_1\in
  \piGamma1$ we can also find $\f y_2'$ such that
  $(\f y_1,\f y_2')\in \Gamma$.

  If $\f y_2'=\f o_2$ then $(\f y_1',\f o_2)\in \Gamma$ as well, since
  $\Gamma$ is a radial cone, and $\sfH(\f y_1,\f o_2)<+\infty$ since
  $\Gamma\subset \rmD(\sfH)$. We conclude that
  $P=(\f y_2,\f y_1', \f o_2,\f y_1)$ is an admissible walk in
  $(V_\Gamma,\arc_{\sfH,\Gamma})$.

  If $\f y_2'\neq \f o_2$ then $(\f y_1,\f y_2')\in \Gamma\cap
  \rintt{\rmD(\sfH)}$ and we can apply
  the first claim of Proposition \ref{ss22:prop:Subdiff} to find a
  walk
  $P\in \mathcal P(\f y_1'\to \f y_1)$ 
  in
  $(V_\Gamma,\arc_{\sfH,\Gamma})$. Joining $(\f y_2,\f y_1')$
  with $P$ we obtain a connection from $\f y_2$ to $\f y_1$.

  \medskip\noindent
  Claim (2).
  Let us pick $(\bar{\f y}_1,\bar{\f y}_2)\in \Gamma\cap
  \rintt{\rmD(\sfH)}$, $\bar{\f y}_i=[\bar x_i,\bar r_i]$ with $\bar
  r_i>0$ and $\bar q:=\bar r_2/\bar r_1$. It is sufficient to show that for every $\f y_1\in
  \piGamma1$ we can find a walk $P''\in \mathcal P(\bar{\f y}_1\to \f y_1)$ and a walk $P'\in \mathcal P(\f y_1\to\bar{\f y}_1)$ in $(V_\Gamma,\mathcal A_{\sfH,\Gamma})$.

  Since $\sfH_{\rm inf}(\f y_1,\bar {\f y}_2)<\infty$,
  we can find $\f y_2=[\bar x_2,r_2]$ with $r_2>0$
  such that $\sfH(\f y_1,\f y_2)<+\infty$, so that $(\f y_2,\f y_1)\in
  \mathcal A_{\sfH,\Gamma}$.
  Since $\Gamma$ is a radial cone, setting $\f y_1':=[\bar
  x_1,r_2/\bar q]$ we have
  $(\f y_1',\f y_2)\in \Gamma$; moreover $(r_2/\bar q,r_2)=r_2/\bar
  r_2(\bar r_1,\bar r_2)\in \intt{\rmD(\sfH_{\bar x_1,\bar x_2})}$ so
    that  $(\f y_1',\f y_2)\in \rintt{\rmD(\sfH)}$.
    By Claim (1) of Proposition
  \ref{ss22:prop:Subdiff}
  we can eventually find a walk $\tilde P\in \mathcal P(\bar {\f y}_1\to \f y_1')$
  in $(V_\Gamma,\arc_{\sfH,\Gamma})$. Thus joining $\tilde P$ with $(\f y_1',\f y_2,\f
  y_1)$ we obtain a walk $P''$ connecting $\bar {\f y}_1$ to $\f y_1$.

  In order to find the second walk connecting $\f y_1$ to $\bar {\f
    y}_1$,
  we first select $\f y_2=[x_2,r_2]$ so that $(\f y_1,\f y_2)\in \Gamma$ (using
  the fact that $\f y_1\in \piGamma1$). Since $\sfH_{\rm inf}(\bar
  {\f y}_1,\f y_2)<\infty$ we find $\lambda_1>0$ such that 
  $\sfH( \lambda_1\bar
  {\f y}_1, \f y_2)<\infty$. By Proposition \ref{ss22:prop:Subdiff}
  we can eventually join $\lambda_1\bar {\f y}_1$
  to $\bar {\f y}_1$ by a walk $\tilde P$ in
  $(V_\Gamma,\arc_{\sfH,\Gamma})$.
  The walk $P'$ obtained by joining $(\f y_1,\fy_2,\lambda_1\bar {\f
    y_1})$
  with $\tilde P$ provides the requested connection from $\f y_1$ to $\bar{\f
    y}_1 $ in $(V_\Gamma,\arc_{\sfH,\Gamma})$.

      \medskip\noindent
      Claim (3). The proof is quite similar to the previous claim. We show that, for every $\fy_1\in \piGamma1$, we can connect $\fy_1$ to $\fo_1$ and $\fo_1$ to $\fy_1$ with a $\mathcal A_{\sfH,\Gamma}$-walk. Let $\fy_1\in \piGamma1$ be fixed. 
      Since $\sfH_{\rm inf}(\f y_1,\bar {\f y}_2)<\infty$,
      we can find $\f y_2=\lambda_2 \bar{\fy}_2$, $\lambda_2>0$, 
  such that $\sfH(\f y_1,\f y_2)<+\infty$, so that $(\f y_2,\f y_1)\in
  \mathcal A_{\sfH,\Gamma}$. The $\mathcal A_{\sfH,\Gamma}$-walk
  $(\fo_1,\fy_2,\fy_1)$
  connects $\fo_1$ to $\fy_1$.

  In order to find the second walk connecting $\f y_1$ to $\fo_1$ 
  we first select $\f y_2=[x_2,r_2]$ so that $(\f y_1,\f y_2)\in \Gamma$ (using
  the fact that $\f y_1\in \piGamma1$). Since $\sfH_{\rm inf}(\bar
  {\f y}_1,\f y_2)<\infty$ we find $\lambda>0$ such that 
  $\sfH( \lambda_1\bar
  {\f y}_1, \f y_2)<\infty$. Since $\Gamma$ is a radial cone, 
  $(\lambda \bar {\fy_1},\fo_2)\in \Gamma$ so that 
  the walk $P'=(\fy_1,\fy_2,\lambda \bar{\fy}_1,\fo_2,\fo_1)$ connects
  $\fy_1$ to $\fo_1$ and belongs to $(V_\Gamma,\arc_{\sfH,\Gamma})$.
\end{proof}

\begin{remark}\label{rem:common} If 
  there is a common
  vector
   $(r_1,r_2)\neq (0,0)$ 
  in the intersection $\bigcap\limits_{(x_1,x_2)\in \sfX_1\times \sfX_2}
  \rmD(\sfH_{x_1,x_2})$ of all the domains of the functions
  $\sfH_{x_1,x_2}$,
  then
  $\sfH_{\rm inf}$ is finite in $(\f C[\sfX_1] \setminus \{\fo_1\})  \times (\f C[\sfX_2] \setminus \{\fo_2\})$. Then, if $\Gamma \subset (\f C[\sfX_1] \setminus \{\fo_1\})  \times (\f C[\sfX_2] \setminus \{\fo_2\})$ and $\Gamma \subset \rintt{\rmD(\sfH)}$, condition (1) of Theorem
  \ref{thm:connectedness} is satisfied and $\Gamma$ is $\sfH$-connected.
  If $\sfH$ is finite 
  then it is immediate to check that any $\Gamma \subset \f C[\sfX_1] \times \f C[\sfX_2]$ is  $\sfH$-connected. 
\end{remark}

\subsection{Sufficient conditions for optimality}
The next result provides
sufficient conditions in order to
guarantee the optimality of a $1$-homogeneous coupling concentrated on a
$\sfH$-cyclically monotone radial cone $\Gamma$.
The main assumptions concern 
$\sfH$-connectedness of $\Gamma$
(see the previous Theorem
\ref{thm:connectedness}
for simple conditions guaranteeing this property),
the fact that $\Gamma$ has nonempty intersection with the radial interior
$\rintt{\rmD(\sfH)}$ of $\rmD(\sfH)$
(or, alternatively, that \eqref{eq:11} holds)
and an 
integrability condition \eqref{ss22:cond1}
which minimic the usual condition
stated in the framework of balanced OT.
To avoid trivial cases, we will
assume that the measures $\mu_i$ have strictly positive mass.
We also mention that the result below provides the existence of a relaxed solution for the dual problem, i.e.~a pair of optimal potentials that are only Borel measurable. 

In order to treat potentials that may take infinite values, we use the notation
\[ \zeta_1(x_1) r_1 +_o \zeta_2(x_2)r_2 := \lim_{n \to + \infty}  (-n \vee \zeta_1(x_1)r_1 \wedge n) +  (-n \vee \zeta_1(x_1)r_2 \wedge n) \]
for functions $\zeta_i : \sfX_i \to \R \cup \{\pm \infty\}$, with the convention that $\pm \infty \cdot 0=0$. In particular $\zeta_1(x_1) r_1 +_o \zeta_2(x_2)r_2=0$ in case $\zeta_1(x_1)=\pm \infty$, $\zeta_2(x_2)=\mp \infty$ and $r_1, r_2 >0$.

\begin{theorem}[Sufficiency of $\sfH$-cyclical monotonicity]
  \label{ss22:suff1} Let $\sfH$ be as in \eqref{ass:3},
  let $\mu_i \in \meas_+(\sfX_i)$ with $\mu_i(\sfX_i)>0$ for $i=1,2$, and let $\aalpha \in
  \f{H}^1(\mu_1,\mu_2)$ be an admissible $1$-homogeneous coupling
  concentrated on a $\sigma$-compact \emph{radial cone $\Gamma\subset \rmD(\sfH)$}
  such that
  $\Gamma$ is $\sfH$-cyclically monotone and $\sfH$-connected.
We set  
\[
  \piGamma i:=\pi^i(\Gamma),\quad
  S_i:= \sfx (\piGamma i\setminus\{\fo_i\}).
\]

  \begin{enumerate}
    \item If $\Gamma$ satisfies one of the two conditions
      \eqref{eq:36} or \eqref{eq:11}
      then the measures $\mu_i$ are concentrated on the sets $S_i$
      (in particular $S_i\neq \emptyset$) and
      there exist Borel functions $\varphi_i:\sfX_i\to \R \cup \{\pm \infty\}$ which are real valued on $S_i$  and such that
      \begin{align}
        \label{eq:37}
        \varphi_1(x_1)r_1 +_o \varphi_2(x_2)r_2&\le
        \sfH([x_1,r_1],[x_2,r_2])\quad\text{for every }x_i\in \sfX_i,\
                                             r_i\ge0,\\ \label{eq:38}
        \varphi_1(x_1)r_1+\varphi_2(x_2)r_2&=
        \sfH([x_1,r_1],[x_2,r_2])\quad\text{if
                                             }([x_1,r_1],[x_2,r_2])\in \Gamma.
      \end{align}
    \item If moreover there exist nonnegative Borel functions $\varrho_i\in
      L^1_+(\sfX_i,\mu_i)$ such that 
\begin{align} \begin{split} \label{ss22:cond1} &\mu_1 \left ( \left \{ x_1 \in \sfX_1 \mid \int_{\sfX_2} \sfH([x_1, \varrho_2(x_2)], [x_2, 1]) \de \mu_2(x_2) < + \infty \right \} \right ) >0, \\
&\mu_2 \left ( \left \{ x_2 \in \sfX_2 \mid \int_{\sfX_1} \sfH([x_1, 1], [x_2,\varrho_1(x_1)]) \de \mu_1(x_1) < + \infty \right \} \right ) >0,
\end{split}
\end{align}
\end{enumerate}

then $\aalpha$ is optimal, $\displaystyle \int_{\pc} \sfH \de \aalpha
< + \infty$, the functions $\varphi_i$ belong to
$\mathcal{L}^1(\sfX_i,\mu_i)$ and provide a relaxed solution 
for the dual problem \eqref{ss22:eq:dualformulation} i.e.
\[
  \begin{gathered}
     \varphi_1(x_1)r_1+_o\varphi_2(x_2)r_2\le
    \sfH([x_1,r_1],[x_2,r_2])\quad\text{for every }x_i\in \sfX_i,\
                                             r_i\ge0,\\
    \int_{\sfX_1}\varphi_1 \de \mu_1 + \int_{\sfX_2} \varphi_2 \de
    \mu_2 = \m_\sfH(\mu_1, \mu_2).
  \end{gathered}
\]
\end{theorem}

\begin{proof}
 We set
  \[
    \Gamma_{\f o}:=\Gamma\setminus \{(\f o_1,\f o_2)\},\quad
    \f C_{\f o}[\sfX_1,\sfX_2]:=\pc\setminus \{(\f o_1,\f o_2)\}.
\]
  We first construct a candidate function $\Phi:\f C[\sfX_1] \to \R$
  inducing a potential $\zeta_1$ via the identity
  $\Phi([x_1,r_1])=\zeta_1(x_1)r_1$.

  \medskip\noindent
  \textbf{Step 1 (Definition of $\Phi$ in $\piGamma 1$)}. We pick a point $(\bar{\f y}_1,\bar {\f y}_2) = ([\bar{x}_1, \bar{r}_1], [\bar{x}_2, \bar{r}_2]) \in \Gamma_\fo$ such that
$(\bar{\f y}_1,\bar {\f y}_2) \in
\rintt{\rmD(\sfH)}$ (condition \eqref{eq:36}) or $\bar \fy_2=\fo_2$
(condition \eqref{eq:11}).
We set (see \eqref{eq:support})
\[
  \bar \sfa:=\sfa(\bar \fy_1,\bar \fy_2).
\]

We define $\Phi: \f{C}[\sfX_1]\to [-\infty, + \infty]$ as
\begin{equation}
  \label{eq:29}
  \Phi(\f{y}_1):=
  \inf \Big\{ \Theta(P)+\bar \sfa
  |r_1-\lambda \bar r_1|: P\in \mathcal P(\lambda\bar \fy_1\to\fy_1),
  \lambda>0
  \Big\},
  \quad \f{y}_1 =[x_1,r_1]\in \f{C}[\sfX_1]. 
\end{equation}
Since $\Gamma$ is $\sfH$-connected and $\lambda \bar{\f y}_1 \in \Gamma_1$,
for every $\fy_1\in \piGamma1$ we can find
$P\in \mathcal P(\lambda \bar \fy_1\to\fy_1)$ with finite cost $\Theta(P)$
so that
\[
  \Phi(\fy_1)<+\infty\quad\text{for every }\fy_1\in \piGamma1.
\]
In particular, choosing $\fy_1:=[\bar x_1,r]=\alpha \bar \fy_1$, $r>0$, 
with $\alpha:=r/\bar r_1$, $\lambda=\alpha$, 
and $P=(\alpha \bar \fy_1,\alpha \bar \fy_2,\alpha \bar \fy_1)$ we have $\Theta(P)=0$ and we immediately get
  \begin{equation}
    \Phi([\bar{x}_1, r]) \le 0\quad
    \text{for every }r> 0.\label{eq:9}
\end{equation}
  
By Proposition \ref{ss22:prop:Subdiff},
every walk $P\in \mathcal P(\lambda\bar \fy_1 \to\bar\fy_1)$, $\lambda>0$
satisfies
\begin{displaymath}
  \Theta(P)
  \ge -\mathsf a(\bar {\f y}_1,\bar {\f
        y}_2)|\lambda \bar r_1-\bar r_1| \quad \text{ for every } \lambda>0,
\end{displaymath}
so that \eqref{eq:29} yields $\Phi(\bar \fy_1)\ge0$;
combining with \eqref{eq:9} we obtain
\[
    \Phi(\bar\fy_1) =0.
\]
Since for every $\f y_1 \in \f C[\sfX_1]$, $\lambda>0$, $P=(\lambda \bar\fy_1, \fy^1, \cdots, \fy^{N-1} ,\fy_1)
\in \mathcal P(\lambda\bar\fy_1 \to\fy_1)$
and every $\alpha>0$ we have
$P_\alpha:=(\alpha\lambda\bar\fy_1,\alpha \fy^1 ,\cdots,\alpha \fy^{N-1},\alpha\fy_1)\in
\mathcal P(\alpha\lambda\bar \fy_1,\alpha\fy_1)$
with
\begin{displaymath}
  \Theta(P_\alpha)+\bar \sfa|\alpha r_1-\alpha\lambda \bar r_1|=
  \alpha\Big(\Theta(P)+\bar\sfa|r_1-\lambda \bar r_1|\Big)
\end{displaymath}
(since $\Gamma$ is a radial cone and $\sfH$ is positively $1$-homogeneous)
we get that
\[
  \Phi(\alpha \f{y}_1) = \alpha\Phi(\f{y}_1) \quad \text{ for
      every } \f{y}_1 \in \f{C}[\sfX_1],\  \alpha >0. 
\]
If now $\fy_1\in \piGamma1$ and $\lambda>0$,
we use the $\sfH$-connectedness of $\Gamma$ to find a $\mathcal
A_{\sfH,\Gamma}$-walk
$\bar P\in \mathcal P(\fy_1\to \bar \fy_1)$.
If $P\in \mathcal P(\lambda \bar \fy_1\to\fy_1)$, $P+\bar P$ is a walk
in $\mathcal P(\lambda\bar\fy_1\to\bar\fy_1)$ 
so that \eqref{eq:22} and  
the identity $\sfa(\lambda \bar \fy_1,\lambda\bar \fy_2) =
\sfa(\bar \fy_1,\bar \fy_2) =\bar \sfa$
yield
\begin{displaymath}
  -\bar\sfa|\lambda-1|\bar r_1\le \Theta(P+\bar P)=
  \Theta(P)+\Theta(\bar P)
\end{displaymath}\
so that 
\[
  \Phi(\fy_1)\ge -\Theta(\bar P) - \sfa |r_1-\bar{r}_1| >-\infty.
\]
Arguing as in \cite[Step 1 of Therem 6.14]{AGS08} and using the fact that $\Gamma$ is a $\sigma$-compact
set, we can see that
$\Phi$ is a Borel function. 

Let us now consider $\f y_1,\f{y}_1' \in \piGamma1$,
$\f y_2\in \f C[\sfX_2]$ such that 
$(\f{y}_1, \f{y}_2) \in {\Gamma}$,
and any walk $P\in \mathcal P(\lambda \bar \fy_1 \to \fy_1)$;
if we consider the walk $P'=P+(\fy_1,\fy_2,\fy_1')$,
$P'$ is an admissible walk joining $\lambda\bar \fy_1$ to $\fy_1'$
so that 
\[
\Phi(\f{y}_1') \le \sfH(\f{y}_1', \f{y}_2) - \sfH(\f{y}_1, \f{y}_2) +
\Theta(P)+\bar\sfa|r_1-\lambda
\bar r_1|
.
\]
Passing then to the infimum among all the walks $P\in \mathcal P
(\lambda \bar \fy_1,\fy_1)$, $\lambda >0$, 
we deduce that
\[
   \Phi(\f y_1')\le
  \Phi(\f  y_1)+\sfH(\f y_1',\f y_2)-\sfH(\f y_1,\f y_2)\quad
  \text{for every }\f y_1'\in \piGamma1,
  \ (\f y_1,\f y_2)\in \Gamma.
\]
Restricting $\Phi$ to $\piGamma1$ (notice that $\Gamma_i$ is Borel since $\Gamma$ is $\sigma$-compact) we have proven that there exists a Borel
function $\Phi:\piGamma1 \to \R $ such that
\begin{alignat}{2} \label{eq:la18}
    \Phi([\bar{x}_1, r_1]) &=0 \quad &&\quad \text{ for every } r_1 >0,\\
    \label{ss22:eq:phi5} \Phi(\lambda \f{y}_1) &= \lambda
    \Phi(\f{y}_1) \quad &&\quad \text{ for every } \f{y}_1 \in
    \piGamma1, \                     
    \lambda >0,\\
    \label{ss22:eq:phi7}  \Phi(\f y_1')&\le
  \Phi(\f  y_1)+\sfH(\f y_1',\f y_2)-\sfH(\f y_1,\f y_2)
  \quad &&\quad\text{ for every } \f{y}_1' \in \piGamma1, \ (\f{y}_1, \f{y}_2) \in {\Gamma}.
\end{alignat}

\quad \\
\textbf{Step 2 (Definition of $\Psi$)}. We define $\Psi: \piGamma2 \to [-\infty, + \infty]$ as
\[ \Psi(\f{y}_2):= \inf_{\f{y}_1 \in \piGamma1} \{ \sfH(\f{y}_1, \f{y}_2) - \Phi(\f{y}_1) \}, \quad \f y_2 \in \Gamma_2. \]
It is clear from the definition that 
\[ \Phi(\f y_1) + \Psi(\f{y}_2) \le \sfH(\f{y}_1, \f{y}_2) \quad
  \text{ for every } (\f{y}_1, \f{y}_2) \in \piGamma1\times \piGamma2. \]
Since $\Phi$ is  real valued on $\piGamma1$  and for every $\fy_2\in \piGamma2$ there
exists
$\fy_1\in \piGamma1$ such that $(\fy_1,\fy_2)\in \Gamma$ so that
$\sfH(\fy_1,\fy_2)<\infty$,
we immediately get $\Psi(\fy_2)<\infty$ for every $\fy_2\in
\piGamma2$. 
By the definition of $\Psi$, the fact that $\Gamma$ is a radial cone,
the radial $1$-homogeneity of $\sfH$ and of $\Phi$ given by
\eqref{ss22:eq:phi5},
it also easily follows that
\[ \Psi(\lambda \f{y}_2) = \lambda \Psi(\f{y}_2) \quad \text{ for
    every } \f{y}_2 \in \piGamma2, \ 
  \lambda >0.\]
The inequality in \eqref{ss22:eq:phi7} immediately yields
\[ \Psi(\f{y}_2) + \Phi(\f{y}_1) =\sfH(\f{y}_1, \f{y}_2)
  \quad \text{ for every } (\f{y}_1, \f{y}_2) \in {\Gamma} \]
and also gives that $\Psi(\f{y}_2) \in \R$ for every $\f{y}_2 \in
\piGamma 2$.
The Borel measurability of $\Psi$ can be checked as in \cite[Step 2 of
Theorem 6.14]{AGS08}.

Summarizing the second step, we have proven that there exists a Borel function $\Psi:\piGamma2 \to \R$ such that
\begin{alignat}{2}
    \label{ss22:eq:psi1} \Phi(\fy_1)+\Psi(\fy_2) &\le \sfH(\fy_1,\fy_2) \quad
    &&\quad \text{ on } \piGamma1\times \piGamma2,\\
     \label{ss22:eq:psi2} \Psi(\lambda \f{y}_2) &= \lambda \Psi(\f{y}_2) \quad &&\quad \text{ for every } \f{y}_2 \in \piGamma2, \ \lambda >0,\\
     \label{ss22:eq:psi5} \Phi(\f{y}_1)+\Psi(\f{y}_2)  &=\sfH(\f{y}_1, \f{y}_2) \quad &&\quad \text{ for every } (\f{y}_1, \f{y}_2) \in {\Gamma}.
\end{alignat}
\quad \\
\textbf{Step 3 (Definition of $\varphi_1$ and $\varphi_2$)}. First of all notice that $S_i$ is Borel since $\Gamma$ is $\sigma$-compact. The following chain of inequalities shows that $\mu_i$ is concentrated on $S_i$:
 \begin{align*}
 \mu_i \left ( \sfX_i \setminus S_i \right ) &= \f{h}_i^1(\aalpha)\left ( \sfX_i \setminus \sfx_i \left(\Gamma \setminus \{ \f{y}_i=\f{o}_i\} \right) \right ) \\
 &=(\sfr_i \aalpha) \left (\sfx_i^{-1} \left ( \sfX_i \setminus \sfx_i \left(\Gamma \setminus \{ \f{y}_i=\f{o}_i\} \right) \right ) \right )\\
 &=(\sfr_i \aalpha) \left ( \pc \setminus \sfx_i^{-1} \left (\sfx_i \left(\Gamma \setminus \{ \f{y}_i=\f{o}_i\} \right) \right ) \right )\\
 &\le (\sfr_i \aalpha) \left ( \pc \setminus \left(\Gamma \setminus \{ \f{y}_i=\f{o}_i\} \right) \right )\\
 &\le (\sfr_i \aalpha) \left (\pc \setminus \Gamma \right ) + (\sfr_i \aalpha) \left ( \Gamma \cap \{ \f{y}_i=\f{o}_i\} \right )\\
 &\le \int_{\pc \setminus \Gamma}\sfr_i \de \aalpha + \int_{\{\sfr_i=0\}} \sfr_i \de \aalpha\\
 &=0.
\end{align*}
Since we have assumed that the two measures $\mu_i$ are positive, this also shows that $S_i \ne \emptyset$, $i=1,2$.
Let us
define $\zeta_i: S_i\to \R$ as
\[ \zeta_1(x_1) := \Phi([x_1, 1]), \quad \zeta_2(x_2) :=
  \Psi([x_2, 1]), \quad x_i \in S_i= \mathsf x (\piGamma i\setminus \{\fo_i\}).\]
Notice that $\zeta_1, \zeta_2$ are Borel functions.
We claim that \begin{align}
                \label{ss22:eq:varphi1}\zeta_1(x_1)r_1+ \zeta_2(x_2)r_2
                &\le \sfH([x_1, r_1],[x_2, r_2]) \quad \text{ for
                  every } x_i\in S_i, r_i\ge 0,\\
                \label{ss22:eq:varphi2}\zeta_1(x_1)r_1+ \zeta_2(x_2)r_2
                &= \sfH([x_1, r_1],[x_2, r_2]) \quad \text{ if } ([x_1, r_1], [x_2,r_2]) \in {\Gamma}.
\end{align}
Notice that the product $\zeta(x_i)r_i$ in \eqref{ss22:eq:varphi2} has to be intended equal to $0$ in case $x_i \notin S_i$ and $r_i=0$.
We distinguish four cases:
\begin{enumerate}[label=(\roman*)]
    \item if $r_1=r_2=0$, both sides in \eqref{ss22:eq:varphi1} and \eqref{ss22:eq:varphi2} are equal to $0$;
    \item if $r_1,r_2\neq 0$, then $\Phi([x_1,
      r_1])=r_1\zeta_1(x_1)$ and $\Psi([x_2,
      r_2])=r_2\zeta_2(x_2)$ by \eqref{ss22:eq:phi5}
      and \eqref{ss22:eq:psi2}, 
      so that \eqref{ss22:eq:varphi1} corresponds to
      \eqref{ss22:eq:psi1} and
      \eqref{ss22:eq:varphi2}
      corresponds to
      \eqref{ss22:eq:psi5}; 
    \item if $r_1>0=r_2$,  then $[x_2, r_2]= \f o_2 $ and
      $\zeta_2(x_2)r_2=0$:
      for \eqref{ss22:eq:varphi1} we can just pass to the limit as
      $r_2\downarrow0$
      in the same inequality, using the fact that
      $r_2\mapsto \mathsf H([x_1,r_1],[x_2,r_2])$ is continuous at $0$. 
    Regarding \eqref{ss22:eq:varphi2}, observe that saying that $([x_1,r_1],\f o_2) \in \Gamma$ gives in particular that $\f o_2 \in \Gamma_2$ so that \eqref{ss22:eq:psi2} forces $\Psi(\f o_2)=0$. Then \eqref{ss22:eq:varphi2}, is exactly \eqref{ss22:eq:psi5} with $\fy_2=\fo_2$;  
    \item the case $r_1=0<r_2$ is completely analogous to the previous
      one also using \eqref{eq:la18}.
\end{enumerate}
Now we can use the $\sfH$-transform (compare also with Definition \ref{def:htr}) to extend the potentials $(\zeta_1, \zeta_2)$ to the whole spaces $\sfX_i$ while keeping intact the relations in \eqref{ss22:eq:varphi1} and \eqref{ss22:eq:varphi2}: we define $\varphi_i: \sfX_i \to \R \cup \{\pm \infty\}$ as
\begin{align} \label{eq:defphi1}
    \varphi_2(x_2) &:= \inf_{x_1 \in S_1, r_1 \ge 0} \left \{ \sfH([x_1,r_1],[x_2,1])-\zeta_1(x_1)r_1 \right \}, \quad x_2 \in \sfX_2,\\ \label{eq:defphi2}
    \varphi_1(x_1) &:= \inf_{x_2 \in \sfX_2, r_2 \ge 0} \left \{ \sfH([x_1,1],[x_2,r_2])-\varphi_2(x_2)r_2 \right \}, \quad x_1 \in \sfX_1,
\end{align}
with the convention that $\pm \infty \cdot 0 =0$ and $\sfH([x_1,1],[x_2,r_2])-\varphi_2(x_2)r_2 = + \infty$ if $\sfH([x_1,1],[x_2,r_2])=+\infty$, $\varphi_2(x_2)r_2=+\infty$.
It is not difficult to check that $\varphi_i$ are Borel function, $\varphi_i(x_i) \in \R$ if $x_i \in S_i$ and 
\begin{align} \label{ss22:eq:varphi11}
        \varphi_1(x_1)r_1+_o\varphi_2(x_2)r_2&\le
        \sfH([x_1,r_1],[x_2,r_2])\quad\text{for every }x_i\in \sfX_i,\
                                             r_i\ge0,\\ \label{ss22:eq:varphi12}
        \varphi_1(x_1)r_1+\varphi_2(x_2)r_2&=
        \sfH([x_1,r_1],[x_2,r_2])\quad\text{if
                                             }([x_1,r_1],[x_2,r_2])\in \Gamma.
      \end{align}
Notice that on $\Gamma$ there is no need to use $+_o$ since either $x_i \in S_i$ (hence $\varphi(x_i) \in \R$) or $r_i=0$ (hence $\varphi(x_i)r_i=0$).
\quad \\
\textbf{Step 4 (Conclusion)}. Since $\mu_1$ is concentrated on $S_1$, using \eqref{ss22:cond1}, we can find some $x_1 \in S_1$
such that $\int_{S_2} \sfH([x_1,
\varrho_2(x_2)x],[x_2,1])\de \mu_2(x_2)< +\infty$;
by \eqref{ss22:eq:varphi11} we get that 
\[ \varphi_1(x_1)\varrho_2(x_2)+\varphi_2(x_2) \le \sfH([x_1, \varrho_2(x_2)],[x_2,1]) \quad \text{ for every } x_2 \in S_2\]
so that 
\[ \varphi_2^+(x_2) \le \sfH([x_1, \varrho_2(x_2) ],[x_2, 1])+\varphi_1(x_1)^-\varrho_2(x_2) \quad \text{ for every } x_2 \in  S_2,\]
where we denoted by $u^+$ and $u^-$ the positive and negative part
respectively of a real number $u$.
This gives that $\varphi_2^+ \in \mathcal{L}^1(\sfX_2, \mu_2)$. The argument for $\varphi_1$ is the same. We can thus conclude that 
\begin{align*}
    \int_{\pc} \sfH \de \aalpha &= \int_{\Gamma} \sfH \de \aalpha\\
    & =\int_{\Gamma} (\varphi_1(x_1)r_1 + \varphi_2(x_2)r_2 ) \de \aalpha([x_1,r_1],[x_2,r_2]) \\
    & = \int_{\sfX_1} \varphi_1 \de \mu_1 + \int_{\sfX_2} \varphi_2 \de \mu_2,
\end{align*}
showing that $\varphi_i \in \mathcal{L}^1(\sfX_i, \mu_i)$. Moreover, if $\tilde{\aalpha} \in \f{H}^1(\mu_1, \mu_2)$, setting $\tilde{\alpha}_i:=\pi^i_\sharp \tilde{\aalpha}$, we have
\[ \int_{\f{C}[\sfX_i]} (\varphi_i \sfr)^{\pm} \de \tilde{\alpha}_i = \int_{\f{C}[\sfX_i]} \varphi_i^{\pm} \sfr \de \tilde{\alpha}_i = \int_{\sfX_i} \varphi_i^{\pm} \de \mu_i \in \R \Rightarrow \int_{\f{C}[\sfX_i]} \varphi_i \sfr \de \tilde{\alpha}_i = \int_{\sfX_i} \varphi_i \de \mu_i \]
so that $(\varphi_i \circ \sfx_i) \sfr_i \in L^1(\pc, \tilde{\aalpha})$ and thus $(\varphi_1 \circ \sfx_1) \sfr_1 + (\varphi_2 \circ \sfx_2) \sfr_2 \in L^1(\pc, \tilde{\aalpha})$. We deduce that the everywhere defined function
\[ ([x_1,r_1], [x_2,r_2]) \mapsto \varphi_1(x_1) r_1 +_o \varphi_2(x_2)r_2\]
is a representative of the $L^1(\pc, \tilde{\aalpha})$-equivalence class $(\varphi_1 \circ \sfx_1) \sfr_1 + (\varphi_2 \circ \sfx_2) \sfr_2$ (notice that the set where the sum would be undefined has null $\tilde{\aalpha}$-measure). Hence, for every $\tilde{\aalpha} \in \f{H}^1(\mu_1, \mu_2)$, we have
\begin{align*}
\int_{\pc} \sfH \de \tilde{\aalpha} &\ge \int_{\pc} (\varphi_1(x_1)r_1 +_o \varphi_2(x_2)r_2 ) \de \tilde{\aalpha}([x_1,r_1],[x_2,r_2]) \\
& = \int ((\varphi_1 \circ \sfx_1) \sfr_1 + (\varphi_2 \circ \sfx_2) \sfr_2) \de \tilde{\aalpha}\\
& = \int_{\f{C}[\sfX_1]} \varphi_1 \sfr \de \tilde{\alpha}_1 + \int_{\f{C}[\sfX_1]} \varphi_2 \sfr \de \tilde{\alpha}_2 \\
& = \int_{\sfX_1} \varphi_1 \de \mu_1 + \int_{\sfX_2} \varphi_2 \de \mu_2 \\
&= \int_{\Gamma} (\varphi_1(x_1)r_1 + \varphi_2(x_2)r_2 ) \de \aalpha ([x_1,r_1],[x_2,r_2]) \\
&= \int_{\pc} \sfH \de \aalpha,
\end{align*}
showing both that $\aalpha$ is optimal and that 
\[ \int_{\sfX_1}\varphi_1 \de \mu_1 + \int_{\sfX_2} \varphi_2 \de \mu_2 = \m_\sfH(\mu_1, \mu_2).
\qedhere\]
\end{proof}

\begin{remark}[$\supp{\mu_i}$ and $S_i$] \label{rem:si}
In the same setting of Theorem \ref{ss22:suff1}, since $\mu_i$ is concentrated on $S_i$, we clearly have that $\supp{\mu_i} \subset \overline{S_i}$. Moreover, we can assume that $S_i \subset \supp{\mu_i}$ since, up to defining 
\[ \Gamma':= \Gamma \cap \hmg{\supp{\aalpha}},\]
we see that $\aalpha$ is still concentrated on $\Gamma'$ and $\Gamma'$ is also a $\sigma$-compact radial cone contained in $\rmD(\sfH)$ which is $\sfH$-cyclically monotone and $\sfH$-connected. It is also easy to check that $\sfx(\tilde{\Gamma}_i \setminus \{ \f{0}_i\}) \subset \supp{\mu_i}$. In this case we thus have $S_i \subset \supp{\mu_i}=\overline{S_i}$.
\end{remark}
\begin{remark}[The case of $\sfH$ finite]\label{rem:hfin} In the same setting of Theorem \ref{ss22:suff1}, in case the cost function $\sfH$ does not attain the value $+\infty$, it is easy to check that the potentials $(\varphi_1, \varphi_2)$ defined via the $\sfH$-transform of the pair ($\zeta_1, \zeta_2)$ as in \eqref{eq:defphi1} and \eqref{eq:defphi2} cannot attain the value $+\infty$ so that the use of $+_o$ is not needed in \eqref{eq:37}.
\end{remark}

We conclude this section showing that, under suitable additional assumptions on $\sfH$, an optimal plan $\aalpha$ must me induced by a transport-growth map as in Definition \ref{def:tgmap} in the sense that
\[ \aalpha = ([\text{id}_{\R^d}, 1],[\mathsf{T}, \mathsf{g}])_\sharp \mu_1,\quad
 \mu_2=(\mathsf T,\mathsf g)_\star \mu_1.\]
The proof is based on the classical approach to the existence of optimal transport maps (see e.g.~\cite[Theorem 6.2.4]{AGS08}): one shows via the existence of optimal potentials as in Theorem \ref{ss22:suff1} that the set on which $\aalpha$ is concentrated is a graph. Condition (1) below is used to guarantee that the optimal potential $\varphi_1$ is approximately differentiable at a.e.~point, while condition (3) is the analogue of the classical twist condition (see e.g.~\cite{twist}). On the other hand condition (2) (cp.~with the condition in \eqref{ass:1})  is used to prevent that $\f o_1$ belongs to the first projection of the support of an optimal plan $\aalpha$: this corresponds to the fact that some of the mass of $\mu_2$ does not come from $\mu_1$ but it is created: clearly in this case there cannot be any transport-growth map connecting $\mu_1$ to $\mu_2$.
\begin{theorem}[Existence of an optimal transport-growth map]\label{theo:themap}
Let $\sfX_1=\sfX_2=\R^d$, $\sfH$ be as in \eqref{ass:3} and finite, $\mu_i \in \meas_+(\R^d)$ with $\mu_i(\R^d)>0$ for $i=1,2$. Assume that condition \eqref{ss22:cond1} is satisfied, $\mu_1 \ll \mathcal{L}^d$ and that $\m_\sfH(\mu_1, \mu_2) <+\infty$. Assume in addition that:
\begin{enumerate}
    \item \label{item1} $\sfH$ is partially differentiable w.r.t.~$x_1,r_1$ in 
    $(\f C(\R^d)\setminus \{\f o_1\})
    \times (\f C(\R^d)\setminus \{\f o_2\})$
    with continuous partial derivatives;
    \item \label{item3a}for every $x_2 \in \supp{\mu_2}$ there exist $x_1 \in \supp{\mu_1}$ and $\eps >0$ such that $\mathrm{B}(x_1, \eps) \times \{x_2\} \subset A_\sfH$ where
        \[
            A_\sfH:= \left \{(x,y) \in \R^d \times \R^d : \partial_{r_1} \sfH(x,0;y,1)=\lim_{r_1 \downarrow 0} \frac{\sfH([x,r_1],[y,1])- \sfH(\f o_1, [y,1])}{r_1}=-\infty \right \};
        \]
    \item \label{item2} for every $x_1 \in \R^d$ the map 
    \[  [y, q] \mapsto \begin{pmatrix}
           \partial_{x_1} \sfH([x_1,1],[y,q]) \\
           \partial_{r_1} \sfH([x_1,1],[y,q])
         \end{pmatrix}
         \]
         is injective.
\end{enumerate}
Then there exist Borel maps $(\mathsf{T},\mathsf{g}):\R^d \to \R^d \times [0,+\infty)$ with $\sfg \in L^1(\sfX_1, \mu_1)$ such that 
\[ \mu_2 = (\mathsf T,\mathsf g)_\star \mu_1 = \mathsf{T}_\sharp (\mathsf{g} \mu_1),  
\quad \int_{\R^d} \sfH([x_1, 1], [\mathsf{T}(x_1),\mathsf g(x_1)]) \de \mu_1(x_1) = 
\m_\sfH(\mu_1, \mu_2)\]
given by the formula
\[
 (\mathsf{T}, \mathsf{g})(x_1)= (\partial_{x_1,r_1} \sfH)^{-1} (\tilde{\nabla} \varphi_1(x_1), \varphi_1(x_1)) \quad \text{for $\mu_1$-a.e.~}x_1 \in \R^d,
\]
where $(\partial_{x_1,r_1} \sfH)^{-1}$ denotes the inverse of the map in item \ref{item2} and $\tilde{\nabla}$ stands for the approximate differential.

\end{theorem}
\begin{proof}
Since $\m_\sfH(\mu_1, \mu_2)<+\infty$ by assumption, Theorem \ref{ss22:prop:mh} gives the existence of an optimal $1$-homogeneous coupling $\aalpha \in \f{H}^1_o(\mu_1, \mu_2)$ concentrated on $\pc \setminus \{ (\f{0}_1, \f{0}_2)\}$. On the other hand, using Proposition \ref{prop:necee} we get that $\aalpha$ is concentrated on a $\sigma$-compact and $\sfH$-cyclically monotone radial convex cone $\tilde{\Gamma}$. As in Remark \ref{rem:si}, we can define
\[ \Gamma:= ( \hmg{\supp{\aalpha}} \cap \rmD(\sfH) \cap \tilde{\Gamma}) \setminus \{ (\f{0}_1, \f{0}_2)\} \]
and we obtain that $\Gamma$ is still a $\sigma$-compact radial cone which is $\sfH$-cyclically monotone, contained in $\rmD(\sfH)$ and such that $(\f{0}_1, \f{0}_2) \notin \Gamma$. Clearly $\aalpha$ is concentrated on $\Gamma$ and by Remark \ref{rem:si} we also have that $S_i:= \sfx(\Gamma_i \setminus \{\f{0}_i\})$ is such that $S_i \subset \supp{\mu_i} = \overline{S_i}$. We want to apply Theorem \ref{ss22:suff1} to deduce the existence of optimal potentials. To do so, we need to check that $\Gamma$ is $\sfH$-connected and that at least one of the two conditions \eqref{eq:36} or \eqref{eq:11} is satisfied: since $\sfH$ is everywhere finite and radially convex, then
\[\rintt{\rmD(\sfH)} = \f{C}_{\f o}[\R^d] \times \f{C}_{\f o}[\R^d] = \{ (\f y_1, \f y_2) \in \f{C}[\R^d] \times\f{C}[\R^d] : \f y_i \ne \f 0_i\}.\]
This in particular gives that at least one of the conditions \eqref{eq:36} or \eqref{eq:11} must be satisfied by $\Gamma$ (otherwise at least one between $\mu_1$ and $\mu_2$ must be the null measure, which is not allowed). Moreover, $\Gamma$ is clearly $\sfH$-connected since $\sfH$ is everywhere finite.

\medskip
\quad \\
By Theorem \ref{ss22:suff1} and Remark \ref{rem:hfin}, we can find Borel functions $\varphi_i:\sfX_i \to [-\infty, + \infty)$ such that $\varphi_i \in \mathcal{L}^1(\R^d, \mu_i)$, $\varphi_i(x_i) \in \R$ if $x_i \in S_i$ and 

\begin{align}
        \label{eq:37tris}
        \varphi_1(x_1)r_1+\varphi_2(x_2)r_2&\le
        \sfH([x_1,r_1],[x_2,r_2])\quad\text{for every }x_i\in \R^d,\
                                             r_i\ge0,\\  \label{eq:38tris}
        \varphi_1(x_1)r_1+\varphi_2(x_2)r_2&=
        \sfH([x_1,r_1],[x_2,r_2])\quad\text{if
                                             }([x_1,r_1],[x_2,r_2])\in \Gamma.
\end{align}
\medskip
\quad \\
We want to show that $\varphi_1$ defined as above is approximately differential $\mu_1$-a.e.~in $\R^d$. Let us define, for every $R>0$, the functions
\[ \varphi_1^R(x_1):= \inf_{(x_2,r_2) \in \mathrm{B}(0,R) \times [0,R]} \left \{ \sfH([x_1,1],[x_2,r_2])- \varphi_2(x_2)r_2\right \}, \quad x_1 \in \R^d.\]
By assumption \ref{item1} we deduce that $\varphi_1^R$ is locally Lipschitz and therefore differentiable $\mathcal{L}^d$-a.e. for $R$ sufficiently large. Moreover, by definition of $S_1$ and since $\Gamma$ is a radial cone, for every $x_1 \in S_1$, we can find $[x_2,r_2] \in \f{C}[\R^d]$ such that $([x_1,1],[x_2,r_2]) \in \Gamma$ so that \eqref{eq:38tris} holds for $([x_1,1],[x_2,r_2])$. This, together with \eqref{eq:37tris}, implies that for $\mu_1$-a.e.~$x_1 \in S_1$ the decreasing family of sets $\{ \varphi < \varphi_1^R\}$ has a $\mu_1$-negligible intersection, i.e.~$\mu_1$-a.e.~$x_1 \in S_1$ belongs to $\{\varphi_1 = \varphi_1^R\}$ for $R$ large enough. It follows that for $\mu_1$-a.e.~$x_1 \in S_1$ the following two conditions are satisfied: $x_1$ is a point of density $1$ of $\{ \varphi_1= \varphi_1^R\}$ for some $R$ and $\varphi_1^R$ is differentiable at $x_1$. We deduce that $\varphi_1$ is approximately differentiable at $x_1$ and $\tilde{\nabla}\varphi_1(x_1) = \nabla \varphi_1^R(x_1)$. Let us denote by $A_1 \subset S_1$ the full $\mu_1$-measure set where this happens.

\medskip
\quad \\
Let now $([x_1,r_1],[x_2,r_2]) \in \Gamma$ with $x_1 \in A_1$. Notice that the map
\[ (x_1',r_1') \mapsto \sfH([x_1',1],[x_2,r_2]) - \varphi_1(x_1')r_1'\] 
attains its minimum at $(x_1,r_1)$. Let us show that it must be that $r_1>0$: if, by contradiction $r_1=0$, then we get
    \[ \sfH([x_1',r_1'],[x_2,r_2])- \varphi_1(x_1')r_1' \ge \sfH(\f o_1, [x_2,r_2]) \text{ for every } (x_1',r_1') \in \R^d \times [0,+\infty).\]
Since $[x_2,r_2] \ne \f o_2$ then $x_2 \in S_2$; by assumption \ref{item3a} and since $\supp{\mu_1}$ cannot contain isolated points, we can find some $x_1' \in S_1$ such that $(x_1', x_2) \in A_\sfH$. We can divide this expression by $r_2>0$ so that
\[ -\infty = \lim_{r_1' \downarrow 0} \frac{ \sfH([x_1', r_1'/r_2],[x_2,1])-\sfH(\f o_1, [x_2,1])}{r_1'} \ge \frac{1}{r_2}\varphi_1(x_1'),\]
a contradiction with the fact that $\varphi_1(x_1') \in \R$.

\medskip
\quad \\
By differentiation we get that
\begin{align*}
\tilde{\nabla}\varphi_1(x_1)r_1 &= \partial_{x_1} \sfH([x_1,r_1],[x_2,r_2]),\\
\varphi_1(x_1)&= \partial_{r_1}\sfH([x_1,r_1],[x_2,r_2]).
\end{align*}
Using the radial $1$-homogeneity of $\sfH$ and the radial $0$-homogeneity of its subdifferential in $r_1$, we deduce that 
\begin{align*}
\tilde{\nabla}\varphi_1(x_1) &= \partial_{x_1} \sfH([x_1,1],[x_2,r_2/r_1]),\\
\varphi_1(x_1)&= \partial_{r_1}\sfH([x_1,1],[x_2,r_2/r_1]).
\end{align*}
By the invertibility assumption \ref{item2}, we deduce that 
\[ [x_2,r_2/r_1]=(\partial_{x_1,r_1} \sfH )^{-1}(\tilde{\nabla}\varphi_1(x_1), \varphi_1(x_1)).\]
Since the set $\Gamma \cap \sfx_1^{-1}(A_1)$ has full $\aalpha$ measure (this follows by the fact that $0 \notin \sfr_1(\Gamma)$), this concludes the proof of the theorem.
\end{proof}
\begin{remark}
    Theorem \ref{theo:themap} could also be proven under a different set of hypotheses, namely dropping the assumption that $\sfH$ is finite everywhere and substituting condition \ref{item3a} by imposing that $\sfH$ is finite on an open cone (cp.~with \eqref{ass:2}): there exists numbers $q_i \ge 0$ such that, setting
\[ U_{q_1 q_2} := \left \{ ([x_1, r_1], [x_2,r_2]) \in \pc \mid r_2 > r_1q_1, \, r_1 > r_2 q_2 \right \},\]
we have $\rmD(\sfH)=U_{q_1q_2}$. Clearly in this case the partial differentiability condition in \ref{item1} and the map in \ref{item2} should be restricted to $\{ x_1 \in \R^d : \sfH([x_1,1],[x_2,r_2])<+\infty\}$ and $\{ [y,q] \in \f{C}[\R^d] : ([x_1,1],[y,q]) \in \rmD(\sfH) \}$, respectively. The proof stays unchanged and we only have to observe that in this case:
\begin{enumerate}
    \item $\rintt{\rmD(\sfH)}=\rmD(\sfH)=U_{q_1 q_2}$ so that $\Gamma \subset \rintt{\rmD(\sfH)}$, hence condition \eqref{eq:36} is satisfied. By Remark \ref{rem:common} $\Gamma$ is also $\sfH$-connected;
    \item whenever $([x_1,r_1],[x_2,r_2]) \in \Gamma$, since $\Gamma \subset U_{q_1 q_2}$, we have $r_1>0$.
\end{enumerate}
\end{remark}

\begin{example}
The function $\sfH_{\mathsf{GHK}}$ in \eqref{eq:ghk} satisfies the hypotheses of Theorem \ref{theo:themap}: $\sfH_{\mathsf{GHK}}$ is a finite, radially $1$-homogeneous, convex and continuous function and condition \eqref{ss22:cond1} is satisfied with $\varrho_i=1$. Condition (1) is clearly satisfied. Condition (3) follows by
\[
    \frac{\sfH_{\mathsf{GHK}}([x_1,r_1],[x_2,1])-\sfH_{\mathsf{GHK}}(\f o_1, [x_2,1])}{r_1} = 1-\frac{2}{\sqrt{r_1}} e^{-|x_1-x_2|^2/2} \to - \infty \text{ as } r_1 \downarrow 0,
\]    
for every $x_1,x_2 \in \R^d$.
Condition (2) is easily seen to be satisfied since, if
\[
\begin{pmatrix}
           \partial_{x_1} \sfH_{\mathsf{GHK}}([x_1,1],[y,q]) \\
           \partial_{r_1} \sfH_{\mathsf{GHK}}([x_1,1],[y,q])
         \end{pmatrix} = 
\begin{pmatrix}
           \sqrt{q}e^{-|x_1-y|^2/2}(x_1-y) \\
           1-\sqrt{q} e^{-|x_1-y|^2/2}
         \end{pmatrix} 
         = 
\begin{pmatrix}
           v_0 \\
           c_0
         \end{pmatrix} 
\]
for some $(v_0,c_0) \in \R^d \times \R$, then
\[ [y,q]=\left [x_1-v_0(1-c_0)^{-1},(1-c_0)^2 e^{|v_0|^2(1-c_0)^{-2}}\right ] \text { if } c_0 \ne 1, \quad [y,q] =\f o_2 \text { if } c_0=1.
         \]
\end{example}

\appendix
\section{Metric and topological properties}\label{sec:7}
In this section we study a few metric and topological properties of the Unbalanced Optimal Transport functional assuming that the cost function is related to a metric on the cone. As usual we fix a completely regular space $\sfX$ and an exponent $p \in [1,+\infty)$. We consider a function $\dcc: \f{C}[\sfX, \sfX] \to [0,+\infty]$ 
such that
\begin{equation}\label{ass:4}
\begin{split}
    \dcc \text{ is an extended metric on } \f{C}[\sfX] \text{ such that } \dcc^p \text{ is radially $1$-homogeneous, proper and lsc.} 
\end{split}
\end{equation}
We consider the Unbalanced Optimal Transport functional induced by $\dcc^p$ on $\meas_+(\sfX)$ and the (extended) Wasserstein $p$-metric \cite{Villani09, AGS08} induced by $\dcc$ on $\prob(\f{C}[\sfX])$.

\begin{definition} Let $\dcc$ be as in \eqref{ass:4}. We define $\dist: \meas_+(\sfX) \times \meas_+(\sfX) \to [0, +\infty]$ and $\wass: \prob(\f{C}[\sfX]) \times \prob(\f{C}[\sfX]) \to [0, + \infty]$ as
\begin{align*}
\dist(\mu_1, \mu_2) := \m_{\dcc^p}^{1/p}(\mu_1, \mu_2) = \left ( \min \left \{ \int_{\pcs} \dcc^p \de \aalpha : \aalpha \in \f{H}^1(\mu_1, \mu_2) \right \} \right )^{1/p}, \quad  \mu_1, \mu_2 \in \meas_+(\sfX),\\
\wass(\alpha_1, \alpha_2) := \mathsf{OT}_{\dcc^p}^{1/p}(\mu_1, \mu_2)= \left ( \min \left \{ \int_{\pcs} \dcc^p \de \ggamma : \ggamma \in \Gamma(\alpha_1, \alpha_2) \right \} \right )^{1/p}, \quad  \alpha_1, \alpha_2 \in \prob(\f{C}[\sfX]),
\end{align*}
where $\mathsf{OT}_{\dcc^p}$ is as in Definition \ref{def:wass}.
Finally we set 
\begin{align*}
    \prob_{ \dcc,p}(\f{C}[\sfX]) &:= \left \{ \alpha \in \prob(\f{C}[\sfX] : \int_{\f{C}[\sfX]} \dcc^p(\f{y},\f{o}) \de \alpha(\f{y}) < +\infty \right \},\\
    \meas_{\dcc,p}(\sfX) &:= \left \{ \mu \in \meas_+(\sfX) : \int_\sfX \dcc^p([x,1],\f{o}) \de \mu(x) < + \infty \right \}.
\end{align*}
\end{definition}

\begin{remark}
If $\mu \in \meas_{\dcc, p}(\sfX)$ then every $\alpha \in \prob(\f{C}[\sfX])$ such that $\f{h}^1(\alpha)=\mu$ is an element of $\prob_{\dcc, p} (\f{C}[\sfX])$.
\end{remark}

\begin{remark} As explained in Remark \ref{rem:tp}, we limit our analysis to the case in which $\dcc^p$ is radially $1$-homogeneous. In case $\dcc^p$ is radially $q$-homogeneous, $q\in (1,+\infty)$,
we can argue as in Remark \ref{rem:tp} and consider its composition with $\mathsf T_q$, which  is a radially $1$-homogeneous metric, obtaining 
\[ \m_{\dcc^p \circ \mathsf{T}_q}(\mu_1, \mu_2) = \inf \left \{ \int_{\pcs} \dcc^p \de \aalpha : \aalpha \in \f{H}^q(\mu_1, \mu_2) \right \}. \]
As a relevant example, the case $\dcc=\mathsf{d}_{\f{C}}$ (see \eqref{ss22:eq:distcone}) and $p=q=2$ fits into this setting and the resulting functional $\dist$ is the Helliger-Kantorovich metric on non-negative measures introduced in \cite{LMS18}.
\end{remark}
 
The next two theorems are generalizations of \cite[Corollary 7.14, Theorem 7.15]{LMS18} and, although the proofs are similar, there are some modifications to be taken into account so that we report them. Similar results are also treated in \cite{DP20, DM22}. 
First of all we show that $\dist$ is indeed a metric.
 
 \begin{theorem}\label{ss22:thmmetric} Let $\dcc$ be as in \eqref{ass:4}. Then $(\meas_+(\sfX), \dist)$ is an extended metric space and $(\meas_{\dcc,p}(\sfX), \dist)$ is a metric space.
\end{theorem}
\begin{proof}
Define $\mathsf{T}: \pcs \to \pcs$ as
\[ \mathsf{T}(\f{y}_1,\f{y}_2) = (\f{y}_2,\f{y}_1), \quad (\f{y}_1,\f{y}_2) \in \pcs.\]
Then, for every $\mu_1, \mu_2 \in \meas_+(\sfX)$, $\mathsf{T}_\sharp: \f{H}^1(\mu_1, \mu_2) \to \f{H}^1(\mu_2, \mu_1)$ is a bijection satisfying
\[\int_{\pcs} \dcc^p \de \aalpha = \int_{\pcs} \dcc^p \de \mathsf{T}_\sharp \aalpha\]
by the symmetry of $\dcc$. This gives that $\dist(\mu_1, \mu_2) = \dist(\mu_2, \mu_1)$ for every $\mu_1, \mu_2 \in \meas_+(\sfX)$.\\
If $\mu \in \meas_+(\sfX)$ and we define
\[ \aalpha = ((\text{id}_{\f{C}[\sfX]}, \text{id}_{\f{C}[\sfX]}) \circ \f{p} )_{\sharp} (\mu \otimes \delta_1) \in \f{H}^1(\mu, \mu),\]
we obtain that 
\[ \dist^p(\mu, \mu) \le \int_{\pcs} \dcc^p \de \aalpha = \int_{\pcs} \dcc^p(\f{y}, \f{y}) \de (\f{p}_{\sharp}(\mu \otimes \delta_1))(\f{y})=0.\]
If, on the other hand, $\mu_1, \mu_2 \in \meas_+(\sfX)$ are s.t.~$\dist(\mu_1, \mu_2) =0$ and $\aalpha \in \f{H}^1_o(\mu_1, \mu_2)$ is optimal, we get that $\aalpha$ is concentrated on the diagonal $\{ (\f{y}, \f{y}) : \f{y} \in \f{C}[\sfX] \}$, so that $\mu_1 = \f{h}^1_1(\aalpha) = \f{h}_2^1(\aalpha) =\mu_2$. This proves that $\dist(\mu_1, \mu_2)=0$ if and only if $\mu_1=\mu_2$.\\ Finally if $\mu_1, \mu_2, \mu_3 \in \meas_+(\sfX)$, we can find, thanks to Proposition \ref{ss22:lem:comparison}, $\alpha_1, \alpha_2, \alpha_3 \in \prob(\f{C}[\sfX])$ such that 
 \[ \f{h}^1(\alpha_i) = \mu_i, \quad \dist(\mu_{i-1}, \mu_i) = \wass(\alpha_{i-1}, \alpha_{i}) \quad  i=2,3.\]
Using again Proposition \ref{ss22:lem:comparison}, we have
\[ \dist(\mu_1, \mu_3) \le \wass(\alpha_1, \alpha_3) \le \wass(\alpha_1, \alpha_2) + \wass(\alpha_2, \alpha_3) = \dist(\mu_1, \mu_2) + \dist(\mu_2, \mu_3).\]
This proves that $\dist$ satisfies the triangle inequality and concludes the proof that $(\meas_+(\sfX), \dist)$ is an extended metric space. \\
If $\mu \in \meas_{\dcc,p}(\sfX)$ and $\alpha \in \prob(\f{C}[\sfX])$ is s.t. $\f{h}^1(\alpha) = \mu$, then
\[ \int_{\f{C}[\sfX]} \dcc^p(\f{y}, \f{o}) \de \alpha(\f{y}) = \int_{\f{C}[\sfX]} r \dcc^p([x,1], \f{o}) \de \alpha([x,r]) = \int_{\sfX} \dcc^p([x,1],\f{o}) \de \mu(x) < + \infty,\]
so that $\alpha \in \prob_{\dcc,p}(\f{C}[\sfX])$. Then, again from Proposition \ref{ss22:lem:comparison}, if $\mu_1, \mu_2 \in \meas_{\dcc,p}(\sfX)$ we can find $\alpha_1, \alpha_2 \in \prob_{ \dcc,p}(\f{C}[\sfX])$ s.t.~$\f{h}^1(\alpha_i) = \mu_i$ for $i=1,2$ so that
\[ \dist(\mu_1, \mu_2) \le \wass(\alpha_1, \alpha_2) < + \infty.\]
\end{proof}

The following theorem extends to the unbalanced setting the well known result for the Wasserstein distance (see e.g.~\cite[Remark 7.1.11]{AGS08}).
\begin{theorem}\label{ss22:thmmnarrow} Let $\dcc$ be a metric in $\f C[\sfX]$ inducing the  topology  
of $\f C[\sfX]$ such that $\dcc^p$ is radially $1$-homogeneous.
If $(\mu_n)_n \subset \meas_{\dcc, p}(\sfX)$ and $\mu \in \meas_{\dcc, p}(\sfX)$, then
\[ \lim_{n \to + \infty} \dist(\mu_n, \mu) = 0 \quad \Longleftrightarrow  \quad \begin{cases} \mu_n \weakto \mu \\ \int_{\sfX} \dcc^p([x,1],\f{o}) \de \mu_n(x) \to \int_{\sfX} \dcc^p([x,1],\f{o}) \de \mu(x) \end{cases}.\]
In particular if $\sfX$ is separable, also $(\meas_{\dcc,p}(\sfX), \dist)$ is separable.
\end{theorem}
\begin{proof}  
Let us first observe that 
\eqref{eq:o-neighborhoods}
yields  
\begin{equation}\label{ss22:eq:claim}
    \text{there exists $a>0$ such that } \dcc^p([x,1], \f{o}) \ge a \quad \text{ for every } x \in \sfX.
\end{equation}
We first prove the $\Rightarrow$ implication. Notice that, denoting by $\bm{0}_\sfX$ the null measure in $\sfX$, we have
\begin{equation}\label{ss22:eq:moment}
 \dist(\nu, \bm{0}_\sfX) = \int_{\sfX}\dcc^p([x,1], \f{o})\de \nu(x) \quad \text{ for every } \nu \in \meas_{\dcc, p}(\sfX)
\end{equation}
so that, by triangle inequality, we get
\[ \int_{\sfX} \dcc^p([x,1], \f{o}) \de \mu_n(x) = \dist^p(\mu_n, \bm{0}_\sfX) \to \dist^p(\mu, \bm{0}_\sfX) = \int_{\sfX} \dcc^p([x,1], \f{o}) \de \mu(x).\]
We show that $\mu_n \weakto \mu$ by contradiction: assume that there exist $\xi \in \rmC_b(\sfX)$ and a (unrelabeled) subsequence s.t.
\begin{equation}\label{ss22:eq:tobec}
    \inf_n \left | \int_{\sfX} \xi \de \mu_n - \int_{\sfX} \xi \de \mu \right | > 0.
\end{equation}
Observe that 
\[ \mu_n(\sfX) \le \frac{1}{a} \int_\sfX \dcc^p([x,1],\f{o}) \de \mu_n(x) \to \frac{1}{a}\int_\sfX \dcc^p([x,1],\f{o}) \de \mu(x) < + \infty,\]
so that $R:= (\sup_n \mu_n(\sfX) + \mu(\sfX)) < + \infty$.
By Theorem \ref{ss22:prop:mh} and Lemma \ref{ss22:le:mineq}, we can find $(\aalpha_n)_n \subset \prob(\pcrs{R})$ such that $\aalpha_n \in \f{H}^1_o(\mu_n,\mu)$ is optimal for every $n \in \N$. Let us define $\alpha^1_n := \pi^1_\sharp \aalpha_n$, $\alpha^2_n := \pi^2_\sharp \aalpha_n$, $n \in \N$. Since $\f{h}^1(\alpha^2_n)=\mu$ for every $n \in \N$, we obtain by Lemma \ref{ss22:lemma:generalmarginals} the existence of a subsequence $k \mapsto n(k)$ and $\alpha_2 \in \prob(\f{C}_R[\sfX])$ with $\f{h}^1(\alpha_2)=\mu$ such that $\alpha^2_{n(k)} \weakto \alpha_2$. Moreover
\[ \int_{\f{C}[\sfX]} \dcc^p(\f{y}, \f{o}) \de \alpha^2_n = \int_{\sfX} \dcc^p([x,1], \f{0}) \de \mu(x) \quad \text{ for every } n \in \N,\]
giving that (see e.g.~\cite[Proposition 7.1.5]{AGS08}) $\wass(\alpha^2_{n(k)}, \alpha_2) \to 0$. Then
\[ \wass(\alpha_{n(k)}^1, \alpha_2) \le \wass(\alpha_{n(k)}^1, \alpha_{n(k)}^2) + \wass(\alpha_{n(k)}^2, \alpha_2) = \dist(\mu_{n(k)}, \mu) + \wass(\alpha_n^2, \alpha_2) \to 0,\]
where we used Proposition \ref{ss22:lem:comparison}. Thus $\wass(\alpha_{n(k)}^1, \alpha_2) \to 0$ and, in particular, $\alpha_{n(k)}^1 \weakto \alpha_2$ so that 
\begin{equation}\label{ss22:eq:contradiction}
    \int_{\sfX} \xi \de \mu_{n(k)} = \int_{\f{C}[\sfX]} \xi(x) r \de \alpha_{n(k)}^1([x,r]) \to \int_{\f{C}[\sfX]} \xi(x) r \de \alpha_2([x,r]) = \int_{\sfX} \xi \de \mu,
\end{equation}
where we used that the map 
\[ [x,r] \mapsto r \xi(x)\]
belongs to $\rmC_b(\f{C}_R[\sfX])$ and $\alpha^1_{n(k)}$ is concentrated on $\f{C}_R[\sfX]$ for every $k \in \N$. Since \eqref{ss22:eq:contradiction} is a contradiction with \eqref{ss22:eq:tobec}, this concludes the proof of the $\Rightarrow$ implication.\\
Let us prove the $\Leftarrow$ implication. If $\mu = \bm{0}_\sfX$, we have already by \eqref{ss22:eq:moment} that $\dist(\mu_n, \mu) \to 0$. Let us then assume that $m:=\mu(\sfX)>0$. Up to passing to a (unrelabeled) subsequence, we can assume that $m_n:=\mu_n(\sfX) \ge m/2 > 0$ for every $n \in \N$.
 Let us define $\alpha_n, \alpha \in \prob(\f{C}[\sfX])$ as
\[ \alpha:= \f{p}_{\sharp} \left ( m^{-1} \mu \otimes \delta_{m} \right ), \quad \alpha_n:= \f{p}_{\sharp} \left ( m_n^{-1} \mu_n \otimes \delta_{m_n} \right ) \quad n \in \N.\]
It is easy to check that $\f{h}^1(\alpha_n) = \mu_n$, $n \in \N$, $\f{h}^1(\alpha) = \mu$ and $\alpha_n \weakto \alpha$. To conclude is then sufficient  to show that $\wass(\alpha_n, \alpha) \to 0$ and then apply Lemma \ref{ss22:lem:comparison}. The $1$-homogeneity of $\varrho^p$ yields
\begin{align*}
    \int_{\f{C}[\sfX]} \dcc^p(\f{y}, \f{o}) \de \alpha_n(\f y) &= 
    \int_{\sfX} m_n^{-1}\dcc^p([x,m_n], \f{o}) \de \mu_n(x)
    \\&=
    \int_{\sfX} \dcc^p([x,1], \f{o}) \de \mu_n(x) 
    \stackrel{n\to\infty}\to \int_{\sfX} \dcc^p([x,1], \f{o}) \de \mu(x) 
    \\&=
     \int_{\sfX}m^{-1} \dcc^p([x,m], \f{o}) \de \mu(x) =
    \int_{\f{C}[\sfX]} \dcc^p(\f{y}, \f{o}) \de \alpha(\f y),
\end{align*}
and we get that $\wass(\alpha_n, \alpha) \to 0$ applying \cite[Proposition 7.1.5]{AGS08}.
\end{proof}

\section{The case of a cost function finite on an open cone}\label{app:1}

In this appendix we repeat the constructions of Section \ref{sec:5} under different assumptions on $\sfH$. As in Section \ref{sec:5}, we assume that $\sfX_1$ and $\sfX_2$ are compact and metrizable spaces. Given non negative numbers $q_i$, $i=1,2$, we define the open cone
\[ U_{q_1 q_2} := \left \{ ([x_1, r_1], [x_2,r_2]) \in \pc \mid r_2 > r_1q_1, \, r_1 > r_2 q_2 \right \}.\]
\[\]
Using this notation, we assume that
\begin{equation}\label{ass:2}
    \begin{split}
        \sfH: &\pc \to [0, + \infty] \text{ is continuous, radially $1$-homogeneous and convex and } \\
        &\text{that there exists non-negative numbers $q_1, q_2$ such that } \rmD(\sfH)=U_{q_1 q_2} \text{ and }\\
    &\lim_{r_1 \downarrow q_2} \inf_{(x_1,x_2) \in \sfX_1 \times \sfX_2} \sfH([x_1,r_1],[x_2,1]) =\lim_{r_2 \downarrow q_1} \inf_{(x_1,x_2) \in \sfX_1 \times \sfX_2} \sfH([x_1,1],[x_2,r_2]) = + \infty.
\end{split}
\end{equation}

Notice that, if $q_1=q_2=0$, we are simply assuming that $\sfH$ is finite on the whole open cone.

\begin{proposition} \label{ss22:prop:boundpair1} Assume that $\sfH$ is as in \eqref{ass:2} and that $\mu_i \in \meas_+(\sfX_i)$ are such that
\[ \supp{\mu_i}=\sfX_i, \, i=1,2, \quad q_1 < \frac{\mu_2(\sfX_2)}{\mu_1(\sfX_1)}, \quad q_2 < \frac{\mu_1(\sfX_1)}{\mu_2(\sfX_2)},\]
where $q_i$ are as in \eqref{ass:2}. Then there exists a constant $C>0$ such that, for every $(\varphi_1, \varphi_2) \in \Phi_\sfH$ with 
\[ \int_{\sfX_1}\varphi_1 \de \mu_1 + \int_{\sfX_2} \varphi_2 \de \mu_2 \ge 0,\]
it holds
\[ \varphi_1(x_1) \le C, \quad \varphi_2(x_2) \le C \quad \text{ for every } (x_1, x_2) \in \sfX_1 \times \sfX_2 \]
and there exists $(\bar{x}_1, \bar{x}_2) \in \sfX_1 \times \sfX_2$ such that $\varphi_1(\bar{x}_1) \ge -C$, $\varphi_2(\bar{x}_2) \ge -C$.
\end{proposition}
\begin{proof}
We start from the last claim for $\varphi_1$. Assume by contradiction that there exists a sequence $(\varphi_1^j, \varphi_2^j)_j \subset \Phi_\sfH$ with $\int_{\sfX_1}\varphi_1^j \de \mu_1 + \int_{\sfX_2} \varphi_2^j \de \mu_2 \ge 0$ such that $\max_{x_1 \in \sfX_1} \varphi_1^j(x_1) \to - \infty$. Let $(x_1^j)_j \subset \sfX_1$ be the sequence of points where the maxima are attained. We thus have
\[ \varphi_1^j(x_1^j) \mu_1(\sfX_1) + \int_{\sfX_2} \varphi_2^j \de \mu_2 \ge 0 \quad \text{ for every } j \in \N\]
so that we can find $(x_2^j)_j \subset \sfX_2$ such that
\[ \varphi_2^j(x_2^j) \ge - \frac{\varphi_1^j(x_1^j)\mu_1(\sfX_1)}{ \mu_2(\sfX_2)} \quad \text{ for every } j \in \N.\]
Since $(\varphi_1^j, \varphi_2^j) \in \Phi_\sfH$, we have
\[ \varphi_1^j(x_1^j) \left ( r_1 - \frac{\mu_1(\sfX_1)}{ \mu_2(\sfX_2)} r_2 \right ) \le \varphi_1^j(x_1^j)r_1 + \varphi_2^j(x_2^j) \le \sfH([x_1^j,r_1], [x_2^j, r_2]) \quad \text{ for every } r_1, r_2 \ge 0.\]
We can assume, up to passing to a subsequence, that $(x_1^j, x_2^j) \to (x_1, x_2) \in \sfX_1 \times \sfX_2$.
Thanks to \eqref{ass:2} and the assumptions on $\mu_i$, we can find $\bar{r}_1, \bar{r}_2 >0$ such that
\[ \sfH([x_1, \bar{r}_1], [x_2, \bar{r}_2]) < + \infty, \quad \bar{r}_1 - \frac{\mu_1(\sfX_1)}{\mu_2(\sfX_2)} \bar{r}_2 < 0.\]
 We thus have that 
\[ + \infty \le \sfH([x_1, \bar{r}_1], [x_2, \bar{r}_2],\]
a contradiction with \eqref{ass:2}. The proof for $\varphi_2$ is the same; we have thus proven that there exists a constant $D>0$ independent of $(\varphi_1, \varphi_2)$ and a point $(\bar{x}_1, \bar{x}_2) \in \sfX_1 \times \sfX_2$ such that
\[ \varphi_1(\bar{x}_1) \ge -D, \quad \varphi_2(\bar{x}_2) \ge -D.\]
Thus, if we set 
\[ C:= D + \max_{(x_1, x_2) \in \sfX_1 \times \sfX_2} \sfH([x_1, \bar{r}_1], [x_2,\bar{r}_2]),\]
where $\bar{r}_1$ and $\bar{r}_2$ are as above, we get that 
\[ \varphi_1(x_1) \le \frac{1}{\bar{r}_1} \left ( \sfH([x_1, \bar{r}_1], [\bar{x}_2,\bar{r}_1]) - \bar{r}_2\varphi_2(\bar{x}_2) \right ) \le C \quad \text{ for every } x_1 \in \sfX_1\]
and the corresponding statement for $\varphi_2$.
\end{proof}
In the following result, which is the analogue of Proposition \ref{ss22:prop:propertiesphih}, we use the notation of Definition \ref{def:htr}.

\begin{proposition} \label{ss22:prop:uniformcont1} Assume that $\sfH$ is as in \eqref{ass:2} and that $\mu_i \in \meas_+(\sfX_i)$ are such that
\[ \supp{\mu_i}=\sfX_i, \, i=1,2, \quad q_1 < \frac{\mu_2(\sfX_2)}{\mu_1(\sfX_1)}, \quad q_2 < \frac{\mu_1(\sfX_1)}{\mu_2(\sfX_2)},\]
where $q_i$ are as in \eqref{ass:2}. Then there exist constants $a_i, a_s, b_i, b_s, M>0$ such that, if $(\varphi_1, \varphi_2) \in \Phi_\sfH$ are such that 
\[ \int_{\sfX_1} \varphi_1 \de \mu_1 + \int_{\sfX_2} \varphi_2 \de \mu_2 \ge 0,\]
then $\|\varphi_1^\sfH \|_\infty, \|\varphi_1^{\sfH \sfH} \|_\infty  \le M$ and
\begin{align}\label{ss22:eq:trans11}
    \varphi_1^{\sfH}(x_2) &= \inf_{x_1 \in \sfX_1} \inf_{a_i \le \alpha \le a_s} \biggl \{ \sfH([x_1, \alpha], [x_2, 1])-\alpha \varphi_1(x_1) \biggr \}, \quad x_2 \in \sfX_2,\\ \label{ss22:eq:trans21}
    \varphi_1^{\sfH \sfH}(x_1) &=\inf_{x_2 \in \sfX_2} \inf_{b_i \le \alpha \le b_s} \biggl \{ \sfH([x_1, 1], [x_2, \alpha])-\alpha \varphi_1^\sfH(x_2) \biggr \}, \quad x_1 \in \sfX_1.
\end{align}
Moreover the sets 
\begin{align*} 
\f{C}_0^1 &:=\{(\f{y}_1, \f{y}_2) \in \pc \mid a_i \le \sfr(\f{y}_1) \le a_s, \, \sfr(\f{y}_2)=1 \},\\
\quad \f{C}_0^2 &:=\{(\f{y}_1, \f{y}_2) \in \pc \mid b_i \le \sfr(\f{y}_2) \le b_s, \, \sfr(\f{y}_1)=1 \}
\end{align*}
are compact subset of $U_{q_1 q_2}$ and $(\varphi_1^{\sfH\sfH}, \varphi_1^{\sfH}) \in \Phi_\sfH$, $\varphi_1^{\sfH \sfH} \ge \varphi_1$, $\varphi_1^{\sfH} \ge \varphi_2$. Finally, if $\sfd_i$ are distances metrizing $\sfX_i$, $i=1,2$, then $\varphi_1^{\sfH\sfH}$ is $\sfd_1$-uniformly continuous with the same (uniform) $\sfd_1 \otimes_{\f{C}} \sfd_2$-modulus of continuity of $\sfH$ on $\f{C}_0^2$ and  $\varphi_1^{\sfH}$ is $\sfd_2$-uniformly continuous with the same (uniform) $\sfd_1 \otimes_{\f{C}} \sfd_2$-modulus of continuity of $\sfH$ on $\f{C}_0^1$.
\end{proposition}

\begin{proof}
Let $(\varphi_1, \varphi_2)$ be as in the statement. Let us set
\[ \delta := \frac{1}{2} \frac{1-q_1 q_2}{1+q_1+q_2}\]
so that, for every $0<\eps\le \delta$, we have $q_2+\eps < (q_1+\eps)^{-1}$. Let us fix a point $\bar{\alpha} \in (q_2+\delta, \frac{1}{q_1+\delta}) $ and let us define
\[ m:= \max_{(x_1, x_2) \in \sfX_1 \times \sfX_2} \sfH([x_1, \bar{\alpha}], [x_2,1]) < + \infty,\]
since $([x_1, \bar{\alpha}], [x_2,1]) \in U_{q_1 q_2}$ for every $(x_1, x_2) \in \sfX_1 \times \sfX_2$. \\
By \eqref{ass:2}, we know that for every $L>0$, there exists $\eps_L>0$ such that
\begin{align*}
    \sfH([x_1, r_1],[x_2, 1] &\ge L \quad \text{ for every } (x_1, x_2) \in \sfX_1 \times \sfX_2, \,  0 \le r_1 < q_2 + \eps_L,\\
    \sfH([x_1, 1],[x_2, r_2] &\ge L \quad \text{ for every } (x_1, x_2) \in \sfX_1 \times \sfX_2, \,  0 \le r_2 < q_1 + \eps_L.
\end{align*}
Let 
\[ L:= \max \left \{ m + C ( \bar{\alpha}+q_2+\delta), (q_1+\delta)(m+\bar{\alpha}C)+C \right\}, \]
where $C$ comes from Proposition \ref{ss22:prop:boundpair1}, and let us take any $a_i, a_s >0$ such that
\[ q_2 < a_i < q_2+\eps_L \wedge \delta , \quad \frac{1}{q_1+\eps_L \wedge \delta} < a_s< \frac{1}{q_1}, \]
so that $0<a_i < a_s$, $\f{C}_0^1 \subset U_{q_1 q_2}$ and $a_i \le \bar{\alpha} \le a_s$.\\
If $\alpha>a_s$, then, for every $(x_1, x_2) \in \sfX_1 \times \sfX_2$, we have
\begin{align*}
    \sfH([x_1, \alpha], [x_2, 1]) - \alpha \varphi_1(x_1)
    &= \alpha \left ( \sfH([x_1, 1], [x_2, 1/\alpha])-\varphi_1(x_1) \right ) \\
    &\ge \alpha (L-C) \\
    & \ge a_s(L-C) \\
    & \ge m+ \bar{\alpha}C \\
    & \ge \sfH([\bar{x}_1,\bar{\alpha}]), [x_2,1]-\bar{\alpha}\varphi_1(\bar{x}_1),
\end{align*}
where $\bar{x}_1$ comes from Proposition \ref{ss22:prop:boundpair1}. If $\alpha<a_i$, then, for every $(x_1, x_2) \in \sfX_1 \times \sfX_2$, we have
\begin{align*}
    \sfH([x_1, \alpha], [x_2, 1]) - \alpha \varphi_1(x_1)
    & \ge L - \alpha C \\
    & \ge L- Ca_i\\
    & \ge m+ \bar{\alpha}C \\
    & \ge \sfH([\bar{x}_1,\bar{\alpha}], [x_2,1])-\bar{\alpha}\varphi_1(\bar{x}_1).
\end{align*}
Thus, for every $x_2 \in \sfX_2$, we get
\begin{align*}
 \inf_{x_1 \in \sfX_1} \inf_{0 \le \alpha < a_i \vee \alpha > a_s} \left \{ \sfH([x_1, \alpha], [x_2, 1]) - \alpha \varphi_1(x_1) \right \} &\ge \sfH([\bar{x}_1,\bar{\alpha}], [x_2,1])-\bar{\alpha}\varphi_1(\bar{x}_1) \\
 &> \inf_{x_1 \in \sfX_1} \inf_{a_i \le \alpha \le  a_s} \left \{ \sfH([x_1, \alpha], [x_2, 1]) - \alpha \varphi_1(x_1) \right \}
\end{align*}
and this proves \eqref{ss22:eq:trans11}. The proof of \eqref{ss22:eq:trans21} is analogous.\\
The remaining part of the proof is identical to the one of Proposition \ref{ss22:prop:propertiesphih}.
\end{proof}
By Proposition \ref{ss22:prop:uniformcont1} we obtain the analogue of Theorem \ref{ss22:theo:potexist} also in this setting with exactly the same proof. 
\begin{theorem}\label{ss22:theo:potexist2} Assume that $\sfH$ is as in \eqref{ass:2} and that $\mu_i \in \meas_+(\sfX_i)$ are such that
\[ \supp{\mu_i}=\sfX_i, \, i=1,2, \quad q_1 < \frac{\mu_2(\sfX_2)}{\mu_1(\sfX_1)}, \quad q_2 < \frac{\mu_1(\sfX_1)}{\mu_2(\sfX_2)},\]
where $q_i$ are as in \eqref{ass:2}. Then there exists $(\varphi_1, \varphi_2) \in \Phi_\sfH$ such that
\[ \int_{\sfX_1} \varphi_1 \de \mu_1 + \int_{\sfX_2} \varphi_2 \de \mu_2 = \m_\sfH(\mu_1, \mu_2). \]
\end{theorem}

\printbibliography
\end{document}